\documentclass[pdflatex]{sn-jnl}
\usepackage{indentfirst}
\usepackage{subcaption}
\usepackage{graphicx}
\usepackage{amsmath}
\usepackage{amssymb}
\usepackage{xcolor}
\usepackage[shortlabels]{enumitem}

\providecommand{\doi}[1]{\url{https://doi.org/#1}}
\providecommand{\doiurl}[1]{\url{https://doi.org/#1}}

\numberwithin{equation}{section}

\theoremstyle{thmstyleone}
\newtheorem{thm}{Theorem}[section]
\newtheorem{lem}[thm]{Lemma}
\newtheorem{prop}[thm]{Proposition}

\theoremstyle{thmstylethree}
\newtheorem{deff}{Definition}[section]

\theoremstyle{thmstyletwo}
\newtheorem{rem}[thm]{Remark}
\begin{document}

\title{Spectral Theory for Schrödinger Operators with Interface Transmission Conditions}

\author[1]{\fnm{Linlin} \sur{Su}}
\email{sull@sustech.edu.cn}
\equalcont{These authors contributed equally to this work.}
\author*[2]{\fnm{Yantao} \sur{Wang}}
\email{ytwang@nenu.edu.cn}
\equalcont{These authors contributed equally to this work.}
\author[1]{\fnm{Zhannan} \sur{Zhuang}}
\email{zhuangzn2022@mail.sustech.edu.cn}
\equalcont{These authors contributed equally to this work.}

\affil[1]{\orgdiv{Department of Mathematics}, \orgname{Southern University of Science and Technology}, \orgaddress{Shenzhen, China}}
\affil[2]{\orgdiv{School of Mathematics and Statistics}, \orgname{Northeast Normal University}, \orgaddress{Changchun, China}}

\abstract{
Interfaces influence transport across biological membranes, habitat boundaries, and composite media. Motivated by these applications, we study two classes of Schrödinger eigenvalue problems with transmission conditions allowing jumps in the solution or its normal derivative. We develop a weighted variational framework that accommodates asymmetric interface coupling and yields self-adjoint realizations in suitable weighted Hilbert spaces. For both problems and their adjoints, we establish discrete spectra, variational eigenvalue characterizations, and eigenfunction completeness in arbitrary spatial dimensions. We prove spectral continuity and identify limiting spectra in singular parameter regimes. We also extend Courant’s nodal domain theorem to eigenfunctions that may be discontinuous across the interface. The spectral results further extend to divergence-form elliptic operators under suitable assumptions on the coefficients. These results provide a foundation for analyzing elliptic and parabolic equations with interface transmission conditions.
}



\maketitle

\section*{Acknowledgements}

The research of Linlin Su was supported by the National Natural Science Foundation of China (Grant Nos. 12071203 and 12371214). The research of Yantao Wang was supported by the National Natural Science Foundation of China (Grant No. 12371216).


\section{Introduction}\label{sec1}


 Transmission conditions determine how elliptic equations in adjoining subdomains combine into a global spectral problem. Asymmetric interface coupling can obscure the underlying variational structure, making the choice of function spaces essential to spectral analysis. We investigate this interplay for Schrödinger operators with two classes of transmission conditions, focusing on how the interface parameters shape the spectrum.


We first introduce the geometric setup. In this paper, a domain refers to a connected open set. Let $\Omega\subset\mathbb{R}^N$ be a bounded Lipschitz domain partitioned into two disjoint open subdomains $\Omega_1$ and $\Omega_2$ with  $\overline{\Omega}=\overline{\Omega_1}\cup\overline{\Omega_2}$. Denote the interface by $\Sigma=\partial{\Omega_1}\cap\partial{\Omega_2}\cap \Omega$ and the exterior boundary parts by $\Gamma_{i}=\partial\Omega_{i}\cap\partial\Omega$ for $i=1,2$. Two configurations are considered (see Fig.\,\ref{fig1}): (a) $\Gamma_{1}=\emptyset$; (b) both $\Gamma_{1}$ and $\Gamma_{2}$ are nonempty.
In both cases, we assume that  $\Omega_1$ and $\Omega_2$ are Lipschitz domains.
The unit normal $\nu(x)$ on $\partial\Omega$ and $\Sigma$ is taken to point outward with respect to $\Omega$ and $\Omega_1$, respectively.

\begin{center}
\begin{figure}[h!]
\centering
\includegraphics[scale=0.7]{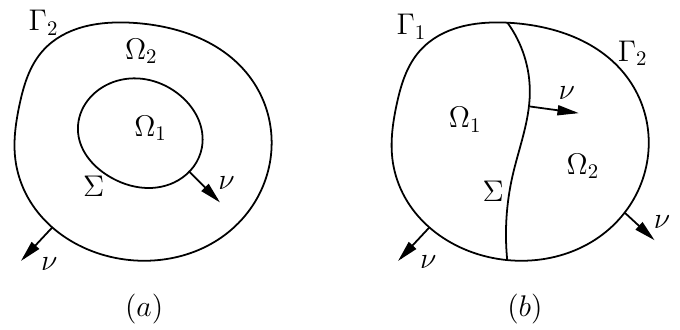}
\caption{\centering\small The interface $\Sigma$ partitions $\Omega$ into $\Omega_1$ and $\Omega_2$}\label{fig1}
\end{figure}
\end{center}

We  investigate the following two eigenvalue problems, where $c_i\in L^{\infty}(\Omega_i)$ ($i=1,2$) are given  real-valued functions. The first problem with interface parameters $\gamma_1,\gamma_2>0$ reads
\begin{align}\label{evp1}
\begin{cases}
-\Delta u_i+c_i(x)u_i=\lambda u_i, & x\in \Omega_i,\\
\partial_\nu u_i=\gamma_i(u_2-u_1), & x\in \Sigma,\\
\partial_\nu u_i=0, & x\in\Gamma_i, 
\end{cases}
\qquad i=1,2.
\end{align}
The second problem with interface parameter $\alpha>0$ reads
\begin{align}\label{evp2}
\begin{cases}
-\Delta \varphi_i+c_i(x)\varphi_i=\mu \varphi_i, & x\in \Omega_i,\\
\partial_\nu \varphi_1=\alpha\partial_\nu \varphi_2,\;\;\varphi_1=\varphi_2, & x\in \Sigma,\\
\partial_\nu \varphi_i=0, & x\in\Gamma_i,
\end{cases}
\qquad i=1,2.
\end{align}

The jump condition  \eqref{evp1}$_2$ models a permeable interface, where $\gamma_1$ and $\gamma_2$ denote the scaled transmissivities from $\Omega_2$ to $\Omega_1$ and from $\Omega_1$ to $\Omega_2$, respectively. It has origins in several fields. In population genetics,  it describes the effect of a geographical barrier on gene frequencies \cite{Nagylaki2016}. In thermodynamics, it reduces to a special instance of the Kedem-Katchalsky transmission condition in the symmetric case $\gamma_1=\gamma_2$ \cite{kedem1961physical}. The existence and uniqueness of a principal eigenvalue (for which the corresponding eigenfunctions do not change sign in $\Omega$) via the Krein-Rutman theorem, together with its characterization and related properties, have been investigated from various perspectives in \cite{Chen2006,WANG2022} for $\gamma_1=\gamma_2$ and in \cite{MAIA2024494,2023Some} for $\gamma_1\neq\gamma_2$.  Semilinear elliptic and parabolic problems with this type of interface condition were also investigated in these papers. For a parabolic system of $m$ species, each of which has a symmetric transmission parameter (i.e., $\gamma^{(j)}_{1}=\gamma^{(j)}_{2}$ for $j=1,\ldots, m$), Ciavolella and Perthame \cite{ciavolella2021existence} established the existence of global-in-time weak solutions.


The condition  \eqref{evp2}$_2$ originates from the theory  of skew Brownian motion  \cite{ito1965diffusion} and was introduced into ecological modeling by Cantrell and Cosner \cite{cantrell1998skew,cantrell1999diffusion} to describe biased movement at habitat edges.
In their formulation, the condition takes the form $(1-a) \partial_\nu \varphi_1=a\partial_\nu \varphi_2$ for $0<a<1$, which corresponds to setting $\alpha=a/(1-a)$ in \eqref{evp2}$_2$. This means that when individuals encounter the interface, they move from $\Omega_1$ to $\Omega_2$ with probability $a$,  and from $\Omega_2$ to $\Omega_1$ with probability $1-a$. In one-dimensional settings, Cantrell and Cosner \cite{cantrell1998skew,cantrell1999diffusion} studied the resulting eigenvalue problem and showed, among other results, how the principal eigenvalue  depends on key parameters such as $a$, patch sizes, and habitat qualities.



Our analysis of problems \eqref{evp1} and \eqref{evp2}  also yields corresponding results for their adjoint problems. 
The adjoint problem of \eqref{evp1} is
\begin{align}\label{evp1*}
\begin{cases}
-\Delta v_i+c_i(x)v_i=\lambda v_i, & x\in \Omega_i,\\
\partial_\nu v_1=\partial_\nu v_2=\gamma_2v_2-\gamma_1v_1, & x\in \Sigma,\\
\partial_\nu v_i=0, & x\in\Gamma_i,
\end{cases}
\qquad i=1,2.
\end{align}
The adjoint problem of \eqref{evp2} is
\begin{align}\label{evp2*}
\begin{cases}
-\Delta \psi_i+c_i(x)\psi_i=\mu \psi_i, & x\in \Omega_i,\\
\partial_\nu \psi_1=\partial_\nu \psi_2,\;\;\alpha\psi_1=\psi_2, & x\in \Sigma,\\
\partial_\nu \psi_i=0, & x\in\Gamma_i,
\end{cases}
\qquad i=1,2.
\end{align}
It is clear that $(\lambda; u_1,u_2)$ satisfies \eqref{evp1} if and only if  $(\lambda; v_1=u_1/\gamma_1, v_2=u_2/\gamma_2)$ satisfies \eqref{evp1*}. Similarly, $(\mu; \varphi_1, \varphi_2)$ satisfies \eqref{evp2} if and only if  $(\mu; \psi_1=\varphi_1, \psi_2=\alpha\varphi_2)$ satisfies \eqref{evp2*}.

Problems \eqref{evp1*} and \eqref{evp2*} are also of independent interest. The condition  \eqref{evp1*}$_2$ is known as a mass transfer condition,  which can be viewed either as a linear approximation of a general nonlinear function or as a Robin-type boundary condition (see, e.g., \cite{LISUWW2021}).  The condition  \eqref{evp2*}$_2$, known as a partition condition, naturally arises in the modeling of analyte transport in composite media, where the discontinuity in concentration across material interfaces is governed by the partition coefficient $\alpha$ (see, e.g., \cite{Carr-March2018,crank1975mathematics,Trefry-Whyte1999}). In one-dimensional settings, this condition was also introduced into biological modeling to describe biased movement at patch boundaries \cite{maciel2020evolutionarily,Maciel-Lutscher2013,Ovaskainen-Cornell2003,StClair-Nevai-Schugart2024}.
More recently, the work \cite{WANG2026} provides a detailed analysis of the principal eigenvalue of a reaction-diffusion model with interface conditions of type \eqref{evp2*}$_{2}$ in a one-dimensional three-patch landscape (extendable to arbitrarily many patches), including its asymptotic behavior with respect to diffusion rates, patch preferences, and patch qualities.

The studies of population models discussed above mainly concern principal eigenvalues, typically obtained via the Krein–Rutman theorem in a strong-solution setting. Here we work with weak solutions in a variational framework. Related full-spectrum results are available in other settings. Moutal and Grebenkov \cite{MoutalGrebenkov2019} studied one-dimensional diffusion with piecewise constant diffusivities, no potential terms, and symmetric semipermeable barriers, obtaining spectral and nodal results that include a special case of \eqref{evp1}. Behrndt, Langer, and Lotoreichik \cite{BehrndtLangerLotoreichik2013} and Behrndt, Exner, and Lotoreichik \cite{BehrndtExnerLotoreichik2014} studied hypersurface interactions on whole-space partitions containing unbounded subdomains. Their $\delta’$-interaction forms are related to the symmetric case ($\gamma_1=\gamma_2$) of \eqref{evp1}, although their geometries and spectral questions differ from ours.

Here we treat \eqref{evp1} and \eqref{evp2}, together with their adjoints, on bounded Lipschitz domains in arbitrary dimensions, allowing asymmetric interface parameters and spatially varying potentials. Weighted formulations place the basic spectral theory within the established framework of symmetric forms. We further establish continuity, permeability monotonicity at fixed parameter ratio, and singular limits for all indexed eigenvalues, retaining multiplicities. We also prove geometric nodal bounds accounting for interface jumps and extend the analysis to divergence-form elliptic operators under suitable coefficient assumptions.

%

\section{Main Results}\label{sec2}

Before stating the main results, we introduce the notation and function spaces used throughout this paper.  For convenience, we write $\gamma=(\gamma_1,\gamma_2)$. We use boldface symbols to denote two-component function pairs: for $\mathbf{u}=(u_1,u_2)$, the component $u_i$ is defined on $\Omega_i$, $i=1,2$. 
 All eigenvalues and eigenfunctions in this paper are understood in the weak sense (see Section \ref{sec3} for precise definitions).

%


%
%

To accommodate the
discontinuity allowed by the interface conditions \eqref{evp1}$_2$, we
introduce the product spaces
\begin{align*}
L^2:=L^2(\Omega_1)\times L^2(\Omega_2),\qquad
H^1:=H^1(\Omega_1)\times H^1(\Omega_2).
\end{align*}
Although $L^2(\Omega_1)\times L^2(\Omega_2)$ can be identified with the standard space $L^2(\Omega)$, we retain the notation $L^2$ to emphasize its two-component structure. We equip $L^2$ with the following inner product and use the corresponding natural norms on $L^2$ and $H^1$: 
\begin{equation}
(\mathbf{u},\mathbf{v})_{L^2}=\int_{\Omega_1}u_1v_1\,\mathrm{d}x+\int_{\Omega_2}u_2v_2 \,\mathrm{d}x,\label{L2}	
\end{equation}
\begin{equation}
	\|\mathbf{u}\|_{L^2}=\left(\|u_1\|_{L^2(\Omega_1)}^2+\|u_2\|_{L^2(\Omega_2)}^2\right)^{1/2},
\end{equation}
\begin{equation}
	 \|\mathbf{u}\|_{H^1}=\left(\|u_1\|_{H^1(\Omega_1)}^2+\|u_2\|_{H^1(\Omega_2)}^2\right)^{1/2}.
\label{L2H1}
\end{equation}

To formulate problem \eqref{evp1} in a symmetric variational framework, we also equip $L^2$ with the
$\gamma$-weighted inner product 
\begin{equation}\label{ip-1}
	(\mathbf{u},\mathbf{v})_{\gamma}
=\sum_{i=1}^{2}\frac{1}{\gamma_i}
\int_{\Omega_i}u_iv_i\,\mathrm{d}x,
\end{equation}
and define the symmetric bilinear form
$a_\gamma:H^1\times H^1\to\mathbb{R}$ by
\begin{equation}\label{bf-1}
	a_{\gamma}(\mathbf{u},\mathbf{v})
=\sum_{i=1}^{2}\frac{1}{\gamma_i}
\int_{\Omega_i}\left[
\nabla u_i\cdot\nabla v_i+c_i(x)u_iv_i
\right]\,\mathrm{d}x
+\int_{\Sigma}(u_2-u_1)(v_2-v_1)\,\mathrm{d}s.
\end{equation}
We say that $\mathbf{u},\mathbf{v}\in L^2$ are {\it $\gamma$-orthogonal}, denoted $\mathbf{u}\perp_{\gamma}\mathbf{v}$, if $(\mathbf{u},\mathbf{v})_{\gamma}=0$. For a linear subspace $V\subset L^2$, we write $\mathbf{u}\perp_{\gamma} V$ if $\mathbf{u}\perp_{\gamma}\mathbf{v}$ for all $\mathbf{v}\in V$. 

For problem \eqref{evp2}, the eigenfunctions are continuous across the
interface, whereas their normal derivatives are proportional. We therefore
introduce the closed subspace
\begin{equation}\label{H1Sigma}
H_\Sigma^1:=\{(u_1,u_2)\in H^1:\;u_1=u_2\;\mbox{on}\;\Sigma\}.
\end{equation}
Under the usual trace identification, $H_\Sigma^1$ can be identified with the
standard space $H^1(\Omega)$. 
To formulate problem \eqref{evp2} in a symmetric variational framework, we equip $L^2$ with the weighted inner product
\begin{equation}\label{ip-2}
	\langle\mathbf{u},\mathbf{v}\rangle_\alpha=\int_{\Omega_1}u_1v_1\,\mathrm{d}x+\alpha\int_{\Omega_2}u_2v_2\,\mathrm{d}x,
\end{equation}
and define the symmetric bilinear form
$b_\alpha:H^1_{\Sigma}\times H^1_{\Sigma}\to\mathbb{R}$ by 
\begin{equation}\label{bf-2}
	b_\alpha(\mathbf{u},\mathbf{v})=\int_{\Omega_1}\left[\nabla u_1\cdot\nabla v_1+c_1(x)u_1v_1\right]\,\mathrm{d}x+\alpha\int_{\Omega_2}\left[\nabla u_2\cdot\nabla v_2+c_2(x)u_2v_2\right]\,\mathrm{d}x.
\end{equation}
%
We say that $\mathbf{u},\mathbf{v}\in L^2$ are {\it $\alpha$-orthogonal},
denoted $\mathbf{u}\perp^\alpha\mathbf{v}$, if
$\langle\mathbf{u},\mathbf{v}\rangle_\alpha=0$. For a linear subspace
$V\subset L^2$, we write $\mathbf{u}\perp^\alpha V$ if
$\mathbf{u}\perp^\alpha\mathbf{v}$ for every $\mathbf{v}\in V$.
The Rayleigh quotients associated with the two problems \eqref{evp1} and \eqref{evp2} are, respectively,  denoted by
\begin{equation}\label{R(u)}
R_\gamma(\mathbf{u})=\frac{a_{\gamma}(\mathbf{u},\mathbf{u})}
{(\mathbf{u},\mathbf{u})_{\gamma}},
\end{equation}
and
\begin{equation}\label{hatR(u)}
\hat{R}_\alpha(\mathbf{u})=\frac{b_{\alpha}(\mathbf{u},\mathbf{u})}
{\langle \mathbf{u},\mathbf{u} \rangle_{\alpha}}.
\end{equation}

With these conventions, we state the main results. Theorems \ref{thm1} and
\ref{thm2} establish the fundamental spectral theory for problems
\eqref{evp1} and \eqref{evp2}, respectively. In the absence of an interface,
they reduce to the classical theory (see, e.g., \cite{Henrot2006}); the
interface conditions require the adapted variational structures introduced
above.  Throughout the following variational formulas, $V_j$ denotes a $j-$dimensional linear
subspace of the indicated ambient space.



\begin{thm}\label{thm1} 
For problem \eqref{evp1}, the following statements hold.
\begin{enumerate}[label=(\roman*),ref=\roman*]
\item\label{thm1:spectrum} All the eigenvalues are real and, counted with multiplicity,
can be enumerated as
\begin{align*}
\lambda_1<\lambda_2\leq\lambda_3\leq\cdots\leq \lambda_k\leq \cdots,\qquad \lim_{k\rightarrow\infty}\lambda_k=\infty.
\end{align*}
\item\label{thm1:regularity} Every real-valued eigenfunction $\mathbf{u}=(u_1,u_2)$ satisfies $u_i\in C(\overline{\Omega_i})$ for $i=1,2$.
\item\label{thm1:principal} The eigenspace corresponding to $\lambda_1$ is one-dimensional, and every real-valued eigenfunction $\mathbf{u}$ associated with $\lambda_1$ 
satisfies either 
\begin{equation*}
	u_i>0\quad \text{on}\quad \overline{\Omega_i},\quad i=1,2,
\end{equation*}
or
\begin{equation*}
	u_i<0\quad \text{on}\quad \overline{\Omega_i},\quad i=1,2.
\end{equation*}
 \item\label{thm1:eigenbasis} There exists a sequence of eigenfunctions
$\{\mathbf{e}_k\}_{k=1}^{\infty}$, with $\mathbf{e}_k$ corresponding to
$\lambda_k$ for every $k\geq1$, such that
\[
(\mathbf{e}_j,\mathbf{e}_\ell)_\gamma=\delta_{j\ell}
\qquad\text{for all }j,\ell\geq1,
\]
\[
\lambda_k=R_{\gamma}(\mathbf{e}_k)
=\min\{R_{\gamma}(\mathbf{u}):\;\mathbf{u}\in H^1,\;\mathbf{u}\neq0,\;\mathbf{u}\perp_\gamma E_{k-1}\}
\qquad\text{for every }k\geq1,
\]
where $E_{k-1}:=\operatorname{span}\{\mathbf{e}_1,\ldots,\mathbf{e}_{k-1}\}$ for $k\geq2$, and $E_0:=\{0\}$.
Moreover, this sequence forms  a basis for both $L^2$ and $H^1$.
\item\label{thm1:minmax} For every $k\geq1$, the eigenvalue $\lambda_k$ also   admits the following characterizations.
\begin{enumerate}[label=(\alph*),ref=\alph*]
\item\label{thm1:courant} Courant type:
\begin{align*}
\lambda_k=\max_{V_{k-1}\subset L^2 } \min\{R_\gamma(\mathbf{u}):\;\mathbf{u}\in H^1,\;\mathbf{u}\neq 0,\;\mathbf{u}\perp_{\gamma} V_{k-1}\}. 
\end{align*}
\item\label{thm1:poincare} Poincar\'{e} type:
\begin{align*}
\lambda_k=\min_{V_k \subset H^1  } \max\{R_\gamma(\mathbf{u}):\;\mathbf{u}\in V_k ,\;\mathbf{u}\neq 0\}.
\end{align*}
\end{enumerate}

\end{enumerate}
\end{thm}

\begin{thm}\label{thm2} For problem \eqref{evp2}, the following statements hold.
\begin{enumerate}[label=(\roman*),ref=\roman*]
\item\label{thm2:spectrum} All the eigenvalues are real and, counted with multiplicity,
can be enumerated as
\begin{align*}
\mu_1<\mu_2\leq\mu_3\leq\cdots\leq \mu_k\leq \cdots,\qquad \lim_{k\rightarrow\infty}\mu_k=\infty.
\end{align*}
\item\label{thm2:regularity} Every real-valued eigenfunction $ \hat{\mathbf{u}}$ belongs to  $C(\overline{\Omega})$. 
\item\label{thm2:principal} The eigenspace corresponding to $\mu_1$ is one-dimensional, 
and every real-valued eigenfunction $\mathbf{u}$ associated with $\mu_1$  
satisfies either 
\begin{equation*}
	\hat{\mathbf{u}}>0\quad \text{in}\quad \overline{\Omega}\quad \text{or}  \quad \hat{\mathbf{u}}<0\quad \text{in}\quad \overline{\Omega}.
\end{equation*}
\item\label{thm2:eigenbasis} There exists a sequence of eigenfunctions
$\{\hat{\mathbf{e}}_k\}_{k=1}^{\infty}$, with $\hat{\mathbf{e}}_k$
corresponding to $\mu_k$ for every $k\geq1$, such that
\[
\langle\hat{\mathbf{e}}_j,\hat{\mathbf{e}}_\ell\rangle_\alpha
=\delta_{j\ell}
\qquad\text{for all }j,\ell\geq1,
\]
\[
\mu_k=\hat{R}_\alpha(\hat{\mathbf{e}}_k)
=\min\left\{\hat{R}_\alpha(\mathbf{u}):\mathbf{u}\in H_\Sigma^1,\ \mathbf{u}\neq0,\ \mathbf{u}\perp^\alpha\hat{E}_{k-1}\right\}
\qquad\text{for every }k\geq1,
\]
where $\hat{E}_{k-1}:=\operatorname{span}\{\hat{\mathbf{e}}_1,\ldots,\hat{\mathbf{e}}_{k-1}\}$ for $k\geq2$, and $\hat{E}_0:=\{0\}$.
Moreover, this sequence forms a basis for both $L^2$ and $H_\Sigma^1$.
\item\label{thm2:minmax} For every $k\geq1$, the eigenvalue $\mu_k$ also admits the following
characterizations.
\begin{enumerate}[label=(\alph*),ref=\alph*]
\item\label{thm2:courant} Courant type:
\begin{align*}
\mu_k=\max_{V_{k-1}\subset L^2} \min\{\hat R_\alpha(\mathbf{u}):\;\mathbf{u}\in H_\Sigma^1,\;\mathbf{u}\neq 0,\;\mathbf{u}\perp^{\alpha} V_{k-1}\}.
\end{align*}
\item\label{thm2:poincare} Poincar\'{e} type:
\begin{align*}
\mu_k=\min_{V_k\subset H_\Sigma^1} \max\{\hat R_\alpha(\mathbf{u}):\;\mathbf{u}\in V_k,\;\mathbf{u}\neq 0\}.
\end{align*}
\end{enumerate}
%
%
\end{enumerate}
\end{thm}

Throughout the rest of this paper, when we wish to emphasize the dependence on the interface parameters, we shall write the eigenvalues in Theorems \ref{thm1} and \ref{thm2} as $\lambda_k(\gamma)$ and $\mu_k(\alpha)$, respectively. The following result examines this dependence, establishing continuity in general and, in a special case, monotonicity as well.

\begin{thm}\label{thm3} 
 For every $k\ge 1$, the following statements hold.
\begin{enumerate}[(i)]
  \item $\lambda_k(\gamma)$ depends continuously on $\gamma$.
  \item $\mu_k(\alpha)$ depends continuously on $\alpha$.
  \item If $\gamma=(\alpha\gamma_2,\gamma_2)$ for a fixed constant $\alpha>0$, then $\lambda_k(\gamma)$ is nondecreasing in $\gamma_2>0$.
\end{enumerate}
\end{thm}
%

The following result presents the asymptotic properties of the eigenvalues $\lambda_k(\gamma)$ of problem \eqref{evp1} as $\gamma\to (\infty,\infty)$ and simultaneously reveals their connection to the eigenvalues $\mu_k(\alpha)$ of problem \eqref{evp2}.
\begin{thm}\label{thm4}
For every $k\ge 1$, $\lambda_k(\gamma)\leq\mu_k(\gamma_1/\gamma_2)$. Moreover, as $\gamma\to  (\infty,\infty)$  with $\gamma_1/\gamma_2\to\alpha\in(0,\infty)$,
\begin{equation}\label{gamma-infty}
\lambda_k(\gamma)\to  \mu_k(\alpha)\,.
\end{equation}
\end{thm}

%

The inequality $\lambda_k(\gamma)\leq\mu_k(\gamma_1/\gamma_2)$ is analogous to the classical comparison between Robin and Dirichlet eigenvalues. Here, however, the interface penalty acts on the jump $u_2-u_1$. As both permeability parameters tend to infinity with their ratio converging to a positive finite value, the penalty enforces the continuity constraint $u_1=u_2$ on $\Sigma$, rather than a homogeneous Dirichlet condition on either component.

The following result describes the asymptotic behavior of the eigenvalues $\lambda_k(\gamma)$ of problem \eqref{evp1} as $\gamma\to (0,0)$. For $i=1,2$, let $\sigma_k^N(\Omega_i)$  denote the $k$-th eigenvalue of the Neumann problem on the subdomain $\Omega_i$:
\begin{align}\label{eqlambdaN}
\begin{cases}
-\Delta u_i+c_i(x)u_i=\sigma u_i, & x\in \Omega_i,\\
\partial_\nu u_i=0, & x\in\partial{\Omega_i}.
\end{cases}
\end{align}
Merge the two sequences $\{\sigma_k^N(\Omega_1)\}_{k=1}^\infty$ and $\{\sigma_k^N(\Omega_2)\}_{k=1}^\infty$ into a single nondecreasing sequence $\{\sigma_k^N\}_{k=1}^\infty$, counting multiplicities.


\begin{thm}\label{thm5}
 For every $k\ge 1$, we have
\begin{equation}\label{gamma-0}
\lambda_k(\gamma)\to  \sigma_k^N \quad \text{as}\;\gamma\to (0,0).
\end{equation}
\end{thm}

 Next, we present the asymptotic behavior of the eigenvalues $\mu_k(\alpha)$ of problem \eqref{evp2} as $\alpha\to 0$ and $\alpha\to\infty$. For $i=1,2$, let $\sigma_k^{DN}(\Omega_i)$ be the $k$-th eigenvalue of the following mixed boundary value problem on $\Omega_i$:
\begin{align}\label{eqlambdaDN}
\begin{cases}
-\Delta u_i+c_i(x)u_i=\sigma u_i, & x\in \Omega_i,\\
u_i=0, & x\in\Sigma,\\
\partial_\nu u_i=0, & x\in \Gamma_i.
\end{cases}
\end{align}
Merge $\{\sigma_k^N(\Omega_1)\}_{k=1}^\infty$ and $\{\sigma_k^{DN}(\Omega_2)\}_{k=1}^\infty$ into a single nondecreasing sequence $\{\sigma_k^{N,DN}\}_{k=1}^\infty$, counting multiplicities.
Similarly, merge $\{\sigma_k^{DN}(\Omega_1)\}_{k=1}^\infty$ and $\{\sigma_k^N(\Omega_2)\}_{k=1}^\infty$ into a single nondecreasing  sequence $\{\sigma_k^{DN,N}\}_{k=1}^\infty$.


\begin{thm}\label{thm6}
For every $k\ge 1$, we have 
\begin{equation}\label{alpha-0,infty}
\mu_k(\alpha)\rightarrow \sigma_k^{N,DN}\; \text{ as }\alpha\rightarrow 0\,,\quad
\mu_k(\alpha)\rightarrow \sigma_k^{DN,N}\; \text{ as }\alpha\rightarrow\infty\,.
\end{equation}
\end{thm}

The following result concerns the number of nodal domains of eigenfunctions. 
In the absence of an interface, an eigenfunction is continuous in $\Omega$ by the De Giorgi-Nash theorem, and its nodal domains are the connected components of $\left\{x\in \Omega:\; u(x)\neq 0\right\}$. The classical Courant nodal domain theorem asserts that the $k$-th eigenfunction has at most $k$ nodal domains (see, e.g., \cite[p.~19]{chavel1984eigenvalues} and \cite{Henrot2006}). 
When an interface is present, each eigenfunction $\mathbf{e}_k$ of problem \eqref{evp1} may jump across the interface and therefore cannot be treated globally as a continuous function. This prevents a direct application of the classical theorem. To resolve this difficulty, we extend the classical definition of nodal sets for  continuous functions to functions with interface jumps. 




%


\begin{deff}\label{def-nodal}
For $\mathbf{u}=(u_1,u_2)$ satisfying
$u_i\in C(\overline{\Omega_i})$ for $i=1,2$, 
the nodal set of $\mathbf{u}$ is
\begin{equation}\label{eq:sajdopi45jroei56}
Z(\mathbf{u}):=\bigcup_{i=1}^{2}\{x\in\Omega_i:\,u_i(x)=0\}
\bigcup \, \{x\in\Sigma:\, u_1(x)u_2(x)\le 0\}.
\end{equation}
A  connected component of $\Omega\setminus Z(\mathbf{u})$ is called a nodal domain of $\mathbf{u}$, and the number of nodal domains is denoted by $\omega(\mathbf{u})$.
\end{deff}


%

\begin{thm}\label{thmnodal} Fix $k\ge1$. Let $\mathbf{u}$ be any real-valued eigenfunction
of problem \eqref{evp1} corresponding to $\lambda_k$, and let
$\hat{\mathbf{u}}$ be any real-valued eigenfunction of problem
\eqref{evp2} corresponding to $\mu_k$. 
Then
\begin{equation}\label{omega}
\omega(\mathbf{u})\le k,\quad \omega(\hat{\mathbf{u}})\le k.
\end{equation}
\end{thm}

The spectral and parameter-dependence results above extend to the
divergence-form interface problems under the coefficient assumptions
in Remark \ref{rem-generalization}; see also Remark \ref{rem-g2}.
The nodal-domain bound in Theorem \ref{thmnodal} extends under the
additional regularity assumption in Proposition \ref{rem-nodal-generalization}.
Without this additional assumption, Proposition \ref{prop1} gives
a multiplicity-dependent nodal bound.



The rest of this paper is organized as follows. Section \ref{sec3} proves Theorems \ref{thm1} and \ref{thm2}, Section \ref{sec4} proves Theorems \ref{thm3}--\ref{thm6}, and Section \ref{sec5} proves Theorem \ref{thmnodal}. Finally, Section \ref{sec6} presents one-dimensional examples illustrating the dependence of eigenvalues on the interface parameters and the nodal-domain estimates for eigenfunctions.

\section{Proofs of Theorems \ref{thm1} and \ref{thm2}}\label{sec3}


We first define the complexification of a real Banach space $X$ by $X_\mathbb{C}:=X \oplus \sqrt{-1}X$. In the following two definitions, the real bilinear forms and inner products introduced in Section \ref{sec2} are understood through their complex-bilinear extensions to the corresponding complexified spaces.  With this convention, we give the precise definitions of eigenvalues and eigenfunctions for problems \eqref{evp1} and \eqref{evp2}, and then show that all eigenvalues are real.

%

\begin{deff}\label{deff1}
We say that $\lambda\in\mathbb{C}$ is an eigenvalue of \eqref{evp1} if there exists a  function $\mathbf{u}\in H^1_{\mathbb{C}}\setminus\{0\}$ such that
\begin{equation}\label{eq11}
a_\gamma(\mathbf{u},\mathbf{v})=\lambda(\mathbf{u},\mathbf{v})_\gamma\qquad\forall\;\mathbf{v}\in H^1_{\mathbb{C}}.	
\end{equation}
Such a function $\mathbf{u}$ is called an eigenfunction associated with $\lambda$.
 \end{deff}



A standard weak formulation for
\eqref{evp1} uses the bilinear form \begin{equation}
\tilde a(\mathbf{u},\mathbf{v})=\sum_{i=1}^{2}\int_{\Omega_i}\left[\nabla u_i\cdot\nabla v_i+c_i(x)u_iv_i\right]\,\mathrm{d}x+\int_{\Sigma}(u_2-u_1)(\gamma_2v_2-\gamma_1v_1)\,\mathrm{d}s,
\end{equation}
which is nonsymmetric when $\gamma_1 \neq \gamma_2$ (see, e.g., \cite{2023Some}). In contrast, our bilinear form $a_\gamma(\cdot,\cdot)$ defined in \eqref{bf-1} is symmetric. This symmetry  immediately yields the reality of the eigenvalues.

Let $(\lambda,\mathbf{u})$ be an eigenpair of \eqref{evp1}. Taking $\mathbf{v}=\overline{\mathbf{u}}$ in \eqref{eq11} gives
\[
a_\gamma(\mathbf{u},\overline{\mathbf{u}})
=\lambda(\mathbf{u},\overline{\mathbf{u}})_\gamma.
\]
Since the coefficients of $a_\gamma$ are real, $a_\gamma(\mathbf{u},\overline{\mathbf{u}})\in\mathbb{R}$, whereas $(\mathbf{u},\overline{\mathbf{u}})_\gamma>0$ for $\mathbf{u}\neq 0$. Hence $\lambda\in\mathbb{R}$. Writing $\mathbf{u}=\operatorname{Re}\mathbf{u}+\sqrt{-1}\operatorname{Im}\mathbf{u}$, and using the fact that $\lambda$ is real, we further see from \eqref{eq11} that each of the nonzero functions  $\operatorname{Re}\mathbf{u}$ and $\operatorname{Im}\mathbf{u}$ is an eigenfunction associated with $\lambda$. Consequently, every eigenvalue of \eqref{evp1} is real, and it suffices to consider real-valued eigenfunctions. 


The definition for \eqref{evp2} is analogous.
\begin{deff}\label{deff2}
	We say that $\mu\in\mathbb{C}$ is an eigenvalue of \eqref{evp2} if there exists a  function $\mathbf{u}\in (H_\Sigma^1)_\mathbb{C}\setminus\{0\}$  such that
	\begin{equation}\label{eq19}
		b_\alpha(\mathbf{u},\mathbf{v})=\mu\langle\mathbf{u},\mathbf{v}\rangle_{\alpha} \quad \forall\; \mathbf{v}\in (H_\Sigma^1)_\mathbb{C}.
	\end{equation}
	Such a function $\mathbf{u}$ is called an eigenfunction associated with $\mu$. 
\end{deff}


The symmetry of $\langle\cdot,\cdot\rangle_\alpha$ and $b_\alpha(\cdot,\cdot)$ similarly implies
that every eigenvalue of \eqref{evp2} is real and we need only consider  real-valued eigenfunctions.


We now turn to the variational characterizations of the eigenvalues for problem \eqref{evp1}. The proof of Theorem \ref{thm1} proceeds through a sequence of lemmas: Lemma \ref{lem:regularityofeigenfunction} establishes part (\ref{thm1:regularity}), Lemma \ref{lem1} establishes part (\ref{thm1:principal}), Lemma \ref{lem2} establishes parts (\ref{thm1:spectrum}) and (\ref{thm1:eigenbasis}), and Lemmas \ref{lem3} and \ref{lem4} establish parts (\ref{thm1:minmax})(\ref{thm1:courant}) and (\ref{thm1:minmax})(\ref{thm1:poincare}), respectively.



We begin with a regularity result for Robin boundary value problems. 
Applying this result to each component of an eigenfunction of \eqref{evp1} gives continuity up to the interface. 


\begin{prop}\label{prop:regularityforrobin}
Let $N\ge 2$, and let $D\subset\mathbb{R}^N$ be a bounded Lipschitz domain.
Assume that $c\in L^\infty(D)$ and $\beta\in L^\infty(\partial D)$. If
$u\in H^1(D)$ is a weak solution of 
\begin{equation*}
	\left\{
	\begin{array}{rcl}
		-\Delta u + c(x)u &=&f \quad \text{in}\quad D,\\[0.2cm]
		\displaystyle\frac{\partial u}{\partial\nu}+\beta(x)u &=&g  \quad \text{on}\quad \partial D,
	\end{array}
	\right.
\end{equation*}
where $\nu$ is the outer unit normal to $\partial D$, then the following assertions hold.
\begin{enumerate}[label=\textup{(\roman*)}]
\item For every $q\in[2N/(N+2),N/2)$ and $\varepsilon>0$, if
\begin{equation*}
	f\in L^{q+\varepsilon}(D),\qquad
g\in L^{\frac{(N-1)q}{N-q}+\varepsilon}(\partial D),
\end{equation*}
then 
\begin{equation*}
	u\in L^{\frac{Nq}{N-2q}}(D),\qquad
u|_{\partial D}\in L^{\frac{(N-1)q}{N-2q}}(\partial D).
\end{equation*}
\item For every $\varepsilon>0$, there exists $\theta\in(0,1)$ 
such that, if 
\begin{equation*}
	f\in L^{\frac N2+\varepsilon}(D),\qquad
g\in L^{N-1+\varepsilon}(\partial D),
\end{equation*}
then $u\in C^{\theta}(\overline D)$.
\end{enumerate}	
\end{prop}
\begin{proof}[{\bf Proof.}]
It follows immediately by applying Theorem 3.14 in \cite{Nittka2011Regularity}  with $\Omega$ replaced by $D$, $a_{ij}=\delta_{ij}$,
$b_j=0$, $c_j=0$, $d=c(x)$, $f_0=f$, and $f_j=0$ for $j=1,\ldots,N$.	
\end{proof}

\begin{lem}\label{lem:regularityofeigenfunction}
Let $\mathbf{u}=(u_1,u_2)\in H^1$ be an
eigenfunction of \eqref{evp1} associated with an eigenvalue $\lambda$. Then
\begin{equation*}
	u_i\in C(\overline{\Omega_i}),\qquad i=1,2.
\end{equation*}
\end{lem}

\begin{proof}[{\bf Proof.}]

If $N=1$, the Sobolev embedding theorem yields  $u_i\in C(\overline{\Omega_i})$ for $i=1,2$. In what follows,
we fix $N\ge 2$. For $i=1,2$, let $\nu_i$ denote the outer unit normal to $\Omega_i$. The
 equations in \eqref{evp1}  can be written as the Robin
 problem
 \begin{equation}\label{eq5.1}
 	\left\{
 	\begin{array}{rcl}
 		-\Delta u_i+\left(c_i(x)-\lambda\right)u_i &=&0\quad\,\, \text{in}\quad \Omega_i,\\[0.2cm]
 		\displaystyle \frac{\partial u_i}{\partial\nu_i}+\beta_i(x) u_i &=&g_i \quad \text{on}\quad \partial\Omega_i,
 	\end{array}
 	\right.
 \end{equation}
 where 
 \begin{equation*}
 	\beta_i(x) =\begin{cases}
 		\gamma_i, & x\in \Sigma,\\
 		0, & x\in\Gamma_i,
 	\end{cases}
 \end{equation*}
 and 
 \begin{equation*}
 	g_i(x) = \begin{cases}
 		\gamma_i u_j(x) ,& x\in\Sigma,\\
 		0,& x\in \Gamma_i, 
 	\end{cases}
 \end{equation*}
 with $j\in \{1,2\}$, $j\neq i$. 
 Thus
 $\beta_i\in L^\infty(\partial\Omega_i)$, and Proposition
 \ref{prop:regularityforrobin} applies to each component.

Suppose first that $N\geq3$, and set
\begin{equation*}
	p_0=\frac{2(N-1)}{N-2}.
\end{equation*}
By the trace theorem, $u_i|_{\Sigma}\in L^{p_0}(\Sigma)$ for $i=1,2$.
If $p_0>N-1$, the desired conclusion already holds. Otherwise, we claim that
this integrability can be improved until the trace belongs to $L^p(\Sigma)$
for some $p>N-1$. Indeed, assume that both traces belong to $L^p(\Sigma)$ for
some $p_0\leq p\leq N-1$. Choose
\[
\frac{2N}{N+2}\leq q<\frac N2
\]
so that
\[
\frac{(N-1)q}{N-q}<p<\frac{(N-1)q}{N-2q}
\quad\text{and}\quad
\frac{N-2q}{(N-1)q}<\frac1p-\frac{1}{2(N-1)}.
\]
Such a choice is possible by taking $q$ sufficiently close to
$Np/(N-1+p)$. Choose $\varepsilon>0$ such that
\[
\frac{(N-1)q}{N-q}+\varepsilon\leq p.
\]
Since $g_i\in L^p(\partial\Omega_i)$ and the interior right-hand side is
zero,  Proposition \ref{prop:regularityforrobin}(i) yields
\[
u_i|_{\Sigma}\in L^{p'}(\Sigma),
\qquad
p'=\frac{(N-1)q}{N-2q},
\]
for $i=1,2$. Moreover, 
\begin{equation}\label{eq:iodsfdhkuifghf456fdgh5t}
	p'>p,\quad\text{and}\quad \frac1{p'}<\frac1p-\frac{1}{2(N-1)}.
\end{equation}
Let $p_1=p_0'$, $p_2=p_1'$, $\cdots$. Then \eqref{eq:iodsfdhkuifghf456fdgh5t} implies that, after finitely many iterations, $u_i|_{\Sigma}$ belongs to
$L^{p_j}(\Sigma)$ for some $p_j>N-1$.

For $N=2$, the trace theorem gives $u_i|_{\Sigma}\in L^p(\Sigma)$ for
every finite $p$, so we may directly choose $p>N-1$. Consequently, for $N\ge 2$,
 there exist $p>N-1$ and $\varepsilon>0$ such that
\[
g_i\in L^{N-1+\varepsilon}(\partial\Omega_i),
\qquad i=1,2.
\]
Applying  Proposition \ref{prop:regularityforrobin}(ii) to the
Robin problem above gives $u_i\in C^{\theta_i}(\overline{\Omega_i})$
for some $\theta_i\in(0,1)$. In particular,
$u_i\in C(\overline{\Omega_i})$ for $i=1,2$.
\end{proof}

\begin{lem}\label{lem1}
	Let 
	\begin{equation}\label{eq:dfiosf8dfhfdsyu}
		\lambda_1=\inf_{\mathbf{u}\in H^1,\;\mathbf{u}\neq 0}R_\gamma(\mathbf{u}),
	\end{equation}
	where $R_\gamma(\mathbf{u})$ is as in \eqref{R(u)}. Then the following statements hold. 
\begin{enumerate}[(i)]
\item There exists  $\mathbf{e}_1\in H^1$, normalized by $(\mathbf{e}_1,\mathbf{e}_1)_{\gamma}=1$,
such that $\lambda_1=R_{\gamma}(\mathbf{e}_1)$. 
	\item $\lambda_1$ is the smallest eigenvalue of \eqref{evp1}, and $\mathbf{e}_1$  is an eigenfunction associated with $\lambda_1$. 
	\item Every real-valued eigenfunction $\mathbf u=(u_1,u_2)$ associated with $\lambda_1$ satisfies either 
\begin{equation*}
	u_i>0\quad \text{on}\quad \overline{\Omega_i},\quad i=1,2,
\end{equation*}
or
\begin{equation*}
	u_i<0\quad \text{on}\quad \overline{\Omega_i},\quad i=1,2.
\end{equation*}
%
	\item The eigenspace of $\lambda_1$ is one-dimensional. 
\end{enumerate}
\end{lem}

\begin{proof}[{\bf Proof.}] (i) We first show $\lambda_1 >-\infty$.  For each nonzero $\mathbf{u}\in H^1$,   we have 
\begin{equation}
	R_{\gamma} (\mathbf{u})\;\ge\; \frac{\displaystyle \sum_{i=1}^{2}\frac{1}{\gamma_i}\int_{\Omega_i}c_i(x)|u_i|^2 \,\mathrm{d}x}{\displaystyle \sum_{i=1}^{2}\frac{1}{\gamma_i}\int_{\Omega_i}|u_i|^2\, \mathrm{d}x}  \;\ge\; -\|\mathbf{c}\|_{L^\infty},
\end{equation}
where \begin{equation*}
	\|\mathbf{c}\|_{L^\infty}=\max\{\|c_1\|_{L^\infty(\Omega_1)},\|c_2\|_{L^\infty(\Omega_2)}\}.
\end{equation*}
Hence,
\begin{equation}\label{eq14}
	\lambda_1\ge -\|\mathbf{c}\|_{L^\infty}>-\infty. 
\end{equation}
%

We next show that the infimum in \eqref{eq:dfiosf8dfhfdsyu} can be attained by some $\mathbf{e}_1$ satisfying $(\mathbf{e}_1,\mathbf{e}_1)_{\gamma}=1$.
Let $\{\mathbf{u}_n=(u_{n1},u_{n2})\}_{n=1}^\infty \subset H^1\setminus\{0\}$ be a minimizing sequence for $R_{\gamma}$, so that $R_\gamma(\mathbf{u}_n)\rightarrow \lambda_1$ as $n\rightarrow\infty$.  Since $R_{\gamma}$ is homogeneous of degree zero, we may normalize the minimizing sequence by 
\begin{equation} \label{eq:sdfoi0d76f8967dfskuif}
(\mathbf{u}_n,\mathbf{u}_n)_\gamma=1.	
\end{equation}
Then as $n\to\infty$, 
\begin{equation}\label{eq:dsfoi0fg76544dsfggg908}
R_\gamma(\mathbf{u}_n)=\sum_{i=1}^{2}\frac{1}{\gamma_i}\int_{\Omega_i}\left(|\nabla u_{ni}|^2+c_i(x)|u_{ni}|^2 \right)\,\mathrm{d}x+\int_{\Sigma}|u_{n2}-u_{n1}|^2 \,\mathrm{d}s\rightarrow \lambda_1. 	
\end{equation}
From \eqref{eq:sdfoi0d76f8967dfskuif}, we see that 
\begin{equation*}
	\left|\sum_{i=1}^{2}\frac{1}{\gamma_i}\int_{\Omega_i}c_i(x)|u_{ni}|^2 \mathrm{d}x\right|\leq\|\mathbf{c}\|_{L^\infty}(\mathbf{u}_n,\mathbf{u}_n)_\gamma=\|\mathbf{c}\|_{L^\infty}.
\end{equation*}
This, together with \eqref{eq:sdfoi0d76f8967dfskuif} and \eqref{eq:dsfoi0fg76544dsfggg908}, implies
that $\{\mathbf{u}_n\}_{n=1}^\infty$ is uniformly bounded in $H^1$. Therefore, by the reflexivity of $H^1$ and the compactness of the embedding $H^1\hookrightarrow L^2$, there exist
 $\mathbf{e}_1\in H^1$ and a subsequence, still denoted by $\{\mathbf{u}_n\}$, such that
\begin{equation}\label{weak-conv}
\mathbf{u}_n\rightharpoonup \mathbf{e}_1\;\;\mbox{in}\;\;H^1,\quad \mathbf{u}_n\rightarrow\mathbf{e}_1\;\;\mbox{in}\;\;L^2 \;\;\mbox{as}\;n\rightarrow\infty,\quad \mbox{and}\quad(\mathbf{e}_1,\mathbf{e}_1)_{\gamma}=1.
\end{equation}
Here and in what follows, `$\rightharpoonup$' denotes weak convergence and `$\rightarrow$' denotes strong convergence. 
We now introduce the following equivalent inner product and its induced norm on $H^1$:
\begin{equation}\label{eq22}
	(\mathbf{u},\mathbf{v})_*=a_\gamma(\mathbf{u},\mathbf{v})+(\|\mathbf{c}\|_{L^\infty}+1)(\mathbf{u},\mathbf{v})_{\gamma},\quad \|\mathbf{u}\|_*=\sqrt{(\mathbf{u},\mathbf{u})_*}.
\end{equation}
By the weak lower semicontinuity of the norm and \eqref{weak-conv}, we have 
\begin{equation*}
	R_{\gamma} (\mathbf{e}_1) = (\mathbf{e}_1,\mathbf{e}_1)_*-(\|\mathbf{c}\|_{L^\infty}+1)(\mathbf{e}_1,\mathbf{e}_1)_{\gamma} \le \liminf_{n\to\infty}  R_{\gamma} (\mathbf{u}_n) =\lambda_1. 
\end{equation*}
On the other hand, by  \eqref{eq:dfiosf8dfhfdsyu}, 
\begin{equation*}
	\lambda_1 \le R_{\gamma} (\mathbf{e}_1). 
\end{equation*}
So, $\lambda_1= R_{\gamma}(\mathbf{e}_1)$. 

(ii) We first show that $(\lambda_1,\mathbf{e}_1)$ is an eigenpair of \eqref{evp1}. 
For any $\mathbf{v}=(v_1,v_2)\in H^1$, define
\begin{equation*}
	F(\varepsilon)=R_\gamma(\mathbf{e}_1+\varepsilon \mathbf{v}).
\end{equation*}
Since $\mathbf{e}_1$ minimizes $R_\gamma$, $F(\varepsilon)$ attains a local minimum at $\varepsilon=0$, and hence  $F'(0)=0$. Setting $N(\varepsilon):=a_\gamma(\mathbf{e}_1+\varepsilon \mathbf{v},\mathbf{e}_1+\varepsilon \mathbf{v})$ and $D(\varepsilon):=(\mathbf{e}_1+\varepsilon \mathbf{v},\mathbf{e}_1+\varepsilon \mathbf{v})_{\gamma}$, we have
\begin{equation*}
	\frac{N'(0)D(0)-N(0)D'(0)}{(D(0))^2}=0.
\end{equation*}
Because $D(0)=(\mathbf{e}_1,\mathbf{e}_1)_\gamma=1$ (recall \eqref{weak-conv}) and $N(0)=a_\gamma(\mathbf{e}_1,\mathbf{e}_1)=\lambda_1$, this simplifies to
$N'(0) =\lambda_1 D'(0)$. Using the symmetry of $a_\gamma$ and $(\cdot,\cdot)_\gamma$, we compute
\begin{equation*}
	N'(0) = 2a_\gamma(\mathbf{e}_1,\mathbf{v}),\quad
D'(0) = 2(\mathbf{e}_1,\mathbf{v})_\gamma.
\end{equation*}
Thus, $a_\gamma(\mathbf{e}_1,\mathbf{v}) = \lambda_1(\mathbf{e}_1,\mathbf{v})_\gamma$. So,  by Definition \ref{deff1}, $(\lambda_1,\mathbf{e}_1)$ is an eigenpair of \eqref{evp1}. 

We then show that $\lambda_1$ is the smallest eigenvalue of \eqref{evp1}. Let $(\lambda,\mathbf{w})$ be any eigenpair of \eqref{evp1}. Then $a_\gamma(\mathbf{w},\mathbf{w}) = \lambda (\mathbf{w},\mathbf{w})_\gamma$, which gives $\lambda = R_\gamma(\mathbf{w}) \ge \lambda_1$ by definition of $\lambda_1$.

(iii) Let $\mathbf{u}=(u_1,u_2)$ be any real-valued eigenfunction associated with $\lambda_1$. Since
$|\mathbf{u}|=(|u_1|,|u_2|)\in H^1$, $|\nabla|u_i||=|\nabla u_i|$ in $\Omega_i$
 ($i=1,\,2$) almost everywhere, and
$(|u_2|-|u_1|)^2\leq (u_2-u_1)^2$ on $\Sigma$ almost everywhere, we have
\begin{equation*}
	\lambda_1\leq R_\gamma(|\mathbf{u}|)\leq R_\gamma(\mathbf{u})=\lambda_1.
\end{equation*}
Thus $|\mathbf{u}|$ is also a minimizer of $R_\gamma$, and the argument in (ii) shows that it is also an eigenfunction associated with $\lambda_1$.
Therefore, by Lemma \ref{lem:regularityofeigenfunction} we have $|u_i|\in C(\overline{\Omega_i})$. 
The Harnack inequality (see, e.g., \cite[Corollary 8.21]{2001EllipticTrudinger} and \cite[Theorem 5.10]{2011EllipticLin}) implies that, for each $i=1,2$, $|u_i|$ is either strictly positive in $\Omega_i$ or identically zero there.


We next show that $|u_i|>0$ throughout $\Omega_i$ for $i=1,2$.
If not, without loss of generality, assume $|u_1|\equiv 0$ in $\Omega_1$. Take $\mathbf{v}=(v_1,0)\in H^1$ with  
$v_1=|u_2|$ on $\Sigma$  in the trace sense (the existence of such $v_1$ is guaranteed  by \cite[Theorem 18.40]{Leoni2009}). Testing the eigen-equation for $|\mathbf{u}|$ with $\mathbf{v}$ yields 
\begin{equation*}
\int_\Sigma |u_2|^2\,\mathrm{d}s=0. 	
\end{equation*}
Now define a new function 
\begin{equation*}
	\tilde{u} =  \begin{cases}
0, & x\in \Omega_1,\\
|u_{2}(x)|, & x\in \Omega_2.
\end{cases}
\end{equation*}
Then $\tilde{u}\in H^1(\Omega)$ and  is a weak solution of
\begin{equation*}
	\begin{cases}
-\Delta \tilde{u}+\tilde{c}(x)\tilde{u}=\lambda_1 \tilde{u}, & x\in \Omega,\\
\partial_\nu \tilde{u}=0, & x\in \partial\Omega,
\end{cases}
\end{equation*}
where 
\begin{equation*}
	\tilde{c} = \begin{cases}
0, & x\in \Omega_1,\\
c_{2}(x), & x\in \Omega_2.
\end{cases}
\end{equation*}
Again, by the De Giorgi regularity and the Harnack inequality, we deduce $\tilde{u}\equiv 0$ in $\Omega$, which is a contradiction. 

Testing the eigen-equation for $\mathbf{u}$ with $\mathbf{v}=(u_1,0)$ gives 
\begin{equation*}
	\frac{1}{\frac{1}{\gamma_1}\int_{\Omega_1}|u_1|^2\,\mathrm{d}x}
\int_\Sigma u_1u_2\,\mathrm{d}s
=R_\gamma(\mathbf{v})-\lambda_1.
\end{equation*} 
Since $(u_1,0)$ cannot be  an  eigenfunction corresponding to $\lambda_1$ by the preceding paragraph, the right-hand side is strictly positive. Therefore
\begin{equation*}
	\int_\Sigma u_1u_2\,\mathrm{d}s>0.
\end{equation*}
Since each $u_i$ is either strictly positive or strictly negative in $\Omega_i$, the preceding integral shows that the two components have the same sign.

Choose the common sign so that $u_i\ge 0$ for $i=1,2$. We now prove that $u_i>0$ on $\overline{\Omega_i}$. 
 Consider the following auxiliary equation for $w_i$:
\begin{equation}\label{eq5.2}
 	\left\{
 	\begin{array}{rcl}
 		-\Delta w_i+\left(c_i(x)-\lambda_1+K\right)w_i &=&K u_i\quad\,\, \text{in}\quad \Omega_i,\\[0.2cm]
 		\displaystyle \frac{\partial w_i}{\partial\nu_i}+\beta_i(x) w_i &=&0 \quad \text{on}\quad \partial\Omega_i,
 	\end{array}
 	\right.
 \end{equation}
 where 
 \begin{equation*}
 	\beta_i(x) =\begin{cases}
 		\gamma_i, & x\in \Sigma,\\
 		0, & x\in\Gamma_i.
 	\end{cases}
 \end{equation*}
 Take $K>0$ large enough that $c_i(x)-\lambda_1+K\geq1$ almost everywhere in $\Omega_i$ for $i=1,2$. Since $\beta_i\geq0$, the bilinear form associated with the homogeneous Robin operator in \eqref{eq5.2} is coercive, and its first eigenvalue is at least $1$. Thus \eqref{eq5.2} admits a unique weak solution $w_i$. Moreover, $Ku_i$ is non-negative and nonzero, and the continuity of $u_i$ on $\overline{\Omega_i}$ implies that $Ku_i\in L^p(\Omega_i)$ for every finite $p$. Applying \cite[Corollary 4.6]{Arendt2020} with spectral parameter $0$, which is smaller than the first eigenvalue of this auxiliary operator, and any $p>N/2$, we obtain
 \begin{align}\label{eq5.3}
w_i>0\quad\text{in }\overline{\Omega_i}.
 \end{align}
 Since $u_i$ is non-negative and satisfies \eqref{eq5.1} with $\lambda=\lambda_1$, the function $z_i=u_i-w_i$ is a weak solution of
 \begin{equation*}
 	\left\{
 	\begin{array}{rcl}
 		-\Delta z_i+\left(c_i(x)-\lambda_1+K\right)z_i &=&0\quad\,\, \text{in}\quad \Omega_i,\\[0.2cm]
 		\displaystyle \frac{\partial z_i}{\partial\nu_i}+\beta_i(x) z_i &=& g_i \quad \text{on}\quad \partial\Omega_i,
 	\end{array}
 	\right.
 \end{equation*}
where 
 \begin{equation*}
 	g_i(x) = \begin{cases}
 		\gamma_i u_j(x) ,& x\in\Sigma,\\
 		0,& x\in \Gamma_i, 
 	\end{cases}
 \end{equation*}
 with $j\in \{1,2\}$, $j\neq i$. 
Then for every $\phi\in H^1(\Omega_i)$, we have
\begin{align*}
\int_{\Omega_i}\left[\nabla z_i \cdot \nabla \phi+\left(c_i(x)-\lambda_1+K\right)z_i \phi\right]\, \mathrm{d}x+\int_{\partial\Omega_i}\beta_i z_i \phi\, \mathrm{d}s=\int_{\Sigma}\gamma_iu_j \phi \,\mathrm{d}s.
\end{align*}
Taking $\phi=z_i^-:=\max\{-z_i,0\}$ as the test function and using $u_j\geq 0$, we obtain 
\begin{align*}
-\int_{\Omega_i}\left[|\nabla z_i^-|^2+\left(c_i(x)-\lambda_1+K\right)|z_i^-|^2\right] \mathrm{d}x-\int_{\partial\Omega_i}\beta_i |z_i^-|^2\, \mathrm{d}s=\int_{\Sigma}\gamma_iu_j z_i^- \,\mathrm{d}s\geq  0.
\end{align*}
Hence $z_i^-=0$, i.e., $u_i\geq w_i$ in $\Omega_i$. So from \eqref{eq5.3}, we have $u_i>0$ in $\overline{\Omega_i}$.
This proves (iii).

%
%



(iv) Suppose that $\mathbf{u}_1$ and $\mathbf{u}_2$ are linearly independent eigenfunctions associated with $\lambda_1$. Set
\[
\tau=\frac{(\mathbf{u}_1,\mathbf{u}_2)_\gamma}{(\mathbf{u}_1,\mathbf{u}_1)_\gamma}.
\]
Then $\mathbf{u}_2-\tau\mathbf{u}_1$ is a nonzero eigenfunction associated with $\lambda_1$ and is orthogonal to $\mathbf{u}_1$ with respect to $(\cdot,\cdot)_\gamma$. This is impossible by (iii).  Thus the eigenspace of $\lambda_1$ is one-dimensional.
\end{proof}

Recall that $\mathbf{u}\perp_\gamma \mathbf{v}$ denotes $(\mathbf{u},\mathbf{v})_{\gamma}=0$. The following lemma gives a variational characterization of the full sequence of eigenvalues of \eqref{evp1}.



\begin{lem}\label{lem2}
For problem \eqref{evp1}, all the eigenvalues are real and, counted with multiplicity,
can be enumerated as
\begin{align*}
\lambda_1<\lambda_2\leq\lambda_3\leq\cdots\leq \lambda_k\leq \cdots,\qquad \lim_{k\rightarrow\infty}\lambda_k=\infty.
\end{align*}
Moreover, there exists a sequence of eigenfunctions $\{\mathbf{e}_k\}_{k=1}^{\infty}$,
with $\mathbf{e}_k$ corresponding to $\lambda_k$ for every $k\geq1$, such that
\[
(\mathbf{e}_j,\mathbf{e}_\ell)_\gamma=\delta_{j\ell}
\qquad\text{for all }j,\ell\geq1,
\]
\[
\lambda_k=R_{\gamma}(\mathbf{e}_k)
=\min\{R_{\gamma}(\mathbf{u}):\;\mathbf{u}\in H^1,\;\mathbf{u}\neq0,\;\mathbf{u}\perp_\gamma E_{k-1}\}
\qquad\text{for every }k\geq1,
\]
where $E_{k-1}:=\operatorname{span}\{\mathbf{e}_1,\ldots,\mathbf{e}_{k-1}\}$ for $k\geq2$, and $E_0:=\{0\}$.
The sequence $\{\mathbf{e}_k\}_{k=1}^{\infty}$ forms a basis for both $L^2$ and $H^1$.
\end{lem}

\begin{proof}[{\bf Proof.}]

\smallskip
\noindent {\bf Step 1. Construction of the eigenpairs.}

Lemma \ref{lem1} provides a normalized first eigenpair
$(\lambda_1,\mathbf{e}_1)$. Suppose, for some $k\geq2$, that the normalized
eigenpairs $(\lambda_j,\mathbf{e}_j)$ have been constructed for
$1\leq j\leq k-1$, with
$(\mathbf{e}_j,\mathbf{e}_\ell)_\gamma=\delta_{j\ell}$. Define
\begin{equation*}
	E_{k-1}:=\operatorname{span}\{\mathbf{e}_1,\ldots,\mathbf{e}_{k-1}\},
\end{equation*}
\begin{equation*}
	X^k:= \{\mathbf{u}\in H^1:\, \mathbf{u} \perp_\gamma E_{k-1}\}, 
\end{equation*}
and set
\begin{equation*}
	\lambda_k:=\inf\{R_\gamma(\mathbf{u}):\;\mathbf{u}\in X^k,\;\mathbf{u}\neq0\}.
\end{equation*}
Arguing as in the proof of Lemma \ref{lem1}, we conclude that 
there exists a minimizer   $\mathbf{e}_k\in X^k$ such that 
\begin{equation*}
	(\mathbf{e}_k,\mathbf{e}_k)_\gamma=1,
\qquad \lambda_k=R_\gamma(\mathbf{e}_k),
\end{equation*}
and that 
\begin{equation*}
	a_\gamma(\mathbf{e}_k,\mathbf{v})
=\lambda_k(\mathbf{e}_k,\mathbf{v})_\gamma
\qquad\text{for all }\mathbf{v}\in X^k.
\end{equation*}
For arbitrary $\mathbf{v}\in H^1$, write 
\begin{equation*}
	\mathbf{v}=\sum_{j=1}^{k-1}(\mathbf{v},\mathbf{e}_j)_\gamma\mathbf{e}_j
+\mathbf{v}_\perp,
\qquad \mathbf{v}_\perp\in X^k.
\end{equation*}
Since each $\mathbf{e}_j$ is an eigenfunction and
$(\mathbf{e}_j,\mathbf{e}_k)_\gamma=0$ for $j<k$, symmetry of
$a_\gamma$ gives
\[
a_\gamma(\mathbf{e}_k,\mathbf{e}_j)
=a_\gamma(\mathbf{e}_j,\mathbf{e}_k)
=\lambda_j(\mathbf{e}_j,\mathbf{e}_k)_\gamma=0.
\]
Consequently,
\[
a_\gamma(\mathbf{e}_k,\mathbf{v})
=a_\gamma(\mathbf{e}_k,\mathbf{v}_\perp)
=\lambda_k(\mathbf{e}_k,\mathbf{v}_\perp)_\gamma
=\lambda_k(\mathbf{e}_k,\mathbf{v})_\gamma.
\]
Thus $(\lambda_k,\mathbf{e}_k)$ is an eigenpair of \eqref{evp1}. By
induction, this constructs the required sequence of normalized,
$\gamma$-orthogonal eigenfunctions and proves its variational
characterization.

\smallskip
\noindent {\bf Step 2. Ordering and multiplicities.}

Since $X^{k+1}\subset X^k$, the variational characterization yields
\[
\lambda_1\leq\lambda_2\leq\cdots\leq\lambda_k\leq\cdots.
\]
Moreover, Lemma \ref{lem1} gives
$E_{\lambda_1}=\operatorname{span}\{\mathbf{e}_1\}$. If
$\lambda_1=\lambda_2$, then $\mathbf{e}_2\in E_{\lambda_1}$, which
contradicts $(\mathbf{e}_1,\mathbf{e}_2)_\gamma=0$. Hence
$\lambda_1<\lambda_2$.


Suppose that an eigenvalue $\lambda$ first appears as $\lambda_\ell$ and
that its eigenspace has dimension $m$. Eigenfunctions corresponding to distinct eigenvalues are
$\gamma$-orthogonal by the symmetry of $a_\gamma$. Hence, whenever fewer than
$m$ $\gamma$-orthonormal eigenfunctions associated with $\lambda$ have been
selected, one can choose an eigenfunction associated with $\lambda$ that is
$\gamma$-orthogonal to all the previously selected eigenfunctions. It is an
admissible competitor in the next variational problem, so the next
eigenvalue is again $\lambda$. Conversely, no more than $m$
$\gamma$-orthonormal eigenfunctions can belong to this eigenspace. Thus, once
an eigenvalue appears in the sequence, it appears precisely according to the
dimension of its eigenspace.

\smallskip
\noindent {\bf Step 3. Divergence of the eigenvalues.}

Suppose that $\{\lambda_k\}_{k=1}^\infty$ is bounded. Since
$(\mathbf{e}_k,\mathbf{e}_k)_\gamma=1$, we have
\begin{equation*}
\sum_{i=1}^{2}\frac{1}{\gamma_i}\int_{\Omega_i}|\nabla e_{ki}|^2\,\mathrm{d}x
\leq \lambda_k+ \|\mathbf{c}\|_{L^\infty}.
\end{equation*}
Thus $\{\mathbf{e}_k\}_{k=1}^\infty$ is bounded in $H^1$, and the compact
embedding $H^1\hookrightarrow L^2$ yields a subsequence converging strongly
in $L^2$. This is impossible because, for $j\neq\ell$,
\[
\|\mathbf{e}_j-\mathbf{e}_\ell\|_\gamma^2=2.
\]
Hence $\lim\limits_{k\to\infty}\lambda_k=\infty$.

\smallskip
\noindent {\bf Step 4. Completeness in $H^1$.}

For every $\mathbf{u}\in H^1$ and each $k\geq1$,
\begin{equation*}
	(\mathbf{u},\mathbf{e}_k)_*
=(\lambda_k+\|\mathbf{c}\|_{L^\infty}+1)(\mathbf{u},\mathbf{e}_k)_\gamma,
\end{equation*}
where $(\cdot,\cdot)_*$ is the inner product introduced in \eqref{eq22}. 
Using \eqref{eq14}, we obtain  
\begin{equation}
	\lambda_k+\|\mathbf{c}\|_{L^\infty}+1\geq\lambda_1+\|\mathbf{c}\|_{L^\infty}+1>0,
\end{equation}
and hence the sequence
$\{\mathbf{e}_k\}_{k=1}^\infty$ is orthogonal with respect to
$(\cdot,\cdot)_*$. Let $\mathbf{v}\in H^1$ be orthogonal to every
$\mathbf{e}_k$ with respect to $(\cdot,\cdot)_*$. Then
$(\mathbf{v},\mathbf{e}_k)_\gamma=0$ for every $k$, and hence
$\mathbf{v}\in X^k$ for all $k$. If $\mathbf{v}\neq0$, then the variational
characterization gives
\begin{equation*}
	R_\gamma(\mathbf{v})\geq\lambda_k\qquad\text{for all }k \geq 1.
\end{equation*}
This contradicts $\lambda_k\to\infty$, since $R_{\gamma}(\mathbf{v})$ is finite.
   Therefore, the orthogonal set $\{\mathbf{e}_k\}_{k=1}^\infty$ is complete in $H^1$ under $(\cdot,\cdot)_*$, and thus forms an orthogonal basis for $H^1$ with respect to $(\cdot,\cdot)_*$.


\smallskip
\noindent {\bf Step 5. Completeness in $L^2$.}

Let $\mathbf{u}\in L^2$ satisfy $\mathbf{u}\perp_{\gamma}\mathbf{e}_k$ for all $k\ge 1$.  Since  $\{\mathbf{e}_k\}_{k=1}^\infty$ is complete in $H^1$ and $H^1$ is dense in $L^2$, $(\mathbf{u},\mathbf{v})_{\gamma}=0$ for all $\mathbf{v}\in L^2$. Hence $\mathbf{u}=0$. This implies that  $\{\mathbf{e}_k\}_{k=1}^\infty$  is   complete in $L^2$.



\smallskip
\noindent {\bf Step 6. Exhaustion of the  eigenvalues.}

Let $\lambda$ be an arbitrary eigenvalue of \eqref{evp1}, and let
$\mathbf{u}_\lambda\neq0$ be an  eigenfunction associated with
$\lambda$. For every $k\geq1$, the symmetry of $a_{\gamma}$ gives 
\[
\lambda(\mathbf{u}_\lambda,\mathbf{e}_k)_\gamma
=a_\gamma(\mathbf{u}_\lambda,\mathbf{e}_k)
=a_\gamma(\mathbf{e}_k,\mathbf{u}_\lambda)
=\lambda_k(\mathbf{u}_\lambda,\mathbf{e}_k)_\gamma.
\]
If $\lambda\neq\lambda_k$ for every $k$, then
$(\mathbf{u}_\lambda,\mathbf{e}_k)_\gamma=0$ for all $k$, contradicting
the completeness in $L^2$. Hence every eigenvalue is one of the
$\lambda_k$, while Step 1 shows that every $\lambda_k$ is an eigenvalue.
Together with Step 2, this proves that the sequence enumerates all
eigenvalues counted with multiplicity. This completes the proof.
\end{proof}

\begin{lem}[Courant type]\label{lem3} 
Let $\{\lambda_k\}_{k=1}^\infty$ be the eigenvalues of \eqref{evp1} given by 
Lemma \ref{lem2}.  
Then, for every $k\ge 1$, 
\begin{equation*}
	\lambda_k=\max_{V_{k-1}\subset L^2} \min\{R_\gamma(\mathbf{u}):\;\mathbf{u}\in H^1,\;\mathbf{u}\neq 0,\;\mathbf{u}\perp_\gamma V_{k-1}\}.
\end{equation*}
\end{lem}
\begin{proof}[{\bf Proof.}]
	Define \begin{equation}\label{eq44}
		\hat{\lambda}_k=\sup_{V_{k-1}\subset L^2} \inf\, \{R_\gamma(\mathbf{u}):\;\mathbf{u}\in H^1,\;\mathbf{u}\neq 0,\;\mathbf{u}\perp_\gamma V_{k-1}\}.
	\end{equation}
On the one hand,  by Lemma \ref{lem2},
\begin{equation}\label{eq1.18}
	\lambda_k= \min\{R_\gamma(\mathbf{u}):\;\mathbf{u}\in H^1,\;\mathbf{u}\neq 0,\;\mathbf{u}\perp_\gamma \operatorname{span}\{\mathbf{e}_1,\mathbf{e}_2,\cdots,\mathbf{e}_{k-1}\}\}\leq \hat{\lambda}_k\,.
\end{equation}
On the other hand, take an arbitrary subspace $V_{k-1} \subset L^2$. Since
\begin{equation*}
	\dim (\operatorname{span}\{\mathbf{e}_1,\mathbf{e}_2,\cdots, \mathbf{e}_k\})=k>\dim V_{k-1},
\end{equation*}
there exists a nonzero function $\mathbf{w}=c_1\mathbf{e}_1+c_2\mathbf{e}_2+\cdots + c_k\mathbf{e}_k$ such that $\mathbf{w}\perp_\gamma V_{k-1}$. 
Using Lemma \ref{lem2}, we obtain 
\begin{equation*}
	\inf\,\{R_\gamma(\mathbf{u}):\mathbf{u}\in H^1,\mathbf{u}\neq 0,\mathbf{u}\perp_\gamma V_{k-1}\}\leq R_\gamma(\mathbf{w})=\frac{a_\gamma(\mathbf{w},\mathbf{w})}{(\mathbf{w},\mathbf{w})_{\gamma}}=\frac{\sum_{j=1}^k c_j^2\lambda_j(\mathbf{e}_j,\mathbf{e}_j)_{\gamma}}{\sum_{j=1}^k c_j^2(\mathbf{e}_j,\mathbf{e}_j)_{\gamma}}\leq \lambda_k.
\end{equation*}
Hence, by the arbitrariness of $V_{k-1}$, $\hat{\lambda}_k\le \lambda_k$. Combining this with \eqref{eq1.18} gives $\hat{\lambda}_k=\lambda_k$.

For every $V_{k-1}\subset L^2$, 
the infimum in  \eqref{eq44} is attained by the  argument used in Lemma \ref{lem2}.
Moreover, the supremum in \eqref{eq44} is attained at $V_{k-1}=\operatorname{span}\{\mathbf{e}_1,\mathbf{e}_2,\cdots, \mathbf{e}_{k-1}\}$. Therefore, the supremum and infimum in \eqref{eq44} may be replaced by
a maximum and a minimum, respectively.
This completes the proof.  
\end{proof}

\begin{lem}[Poincar\'{e} type]\label{lem4}
Let $\{\lambda_k\}_{k=1}^\infty$ be the eigenvalues of \eqref{evp1} given by 
Lemma \ref{lem2}.  
Then, for every $k\ge 1$,
\begin{align*}
\lambda_k=\min_{V_k\subset H^1} \max\{R_\gamma(\mathbf{u}):\;\mathbf{u}\in V_k,\;\mathbf{u}\neq 0\}.
\end{align*}
\end{lem}
\begin{proof}[{\bf Proof.}]
Define
\begin{equation}\label{eq45}
\tilde{\lambda}_k=\inf_{V_k\subset H^1} \sup\{R_\gamma(\mathbf{u}):\;\mathbf{u}\in V_k,\;\mathbf{u}\neq 0\},
\end{equation}
and set $U_k=\operatorname{span}\{\mathbf{e}_1,\mathbf{e}_2,\cdots,\mathbf{e}_k\}$ for $k\geq 1$, with $U_0=\{0\}$.

On the one hand, let $V_k\subset H^1$ be arbitrary with $k\geq 1$. Since $\dim V_k>\dim U_{k-1}$,
there exists a nonzero function $\mathbf{v}\in V_k$ such that $\mathbf{v}\perp_\gamma U_{k-1}$. By Lemma \ref{lem2},
\begin{align*}
\sup\{R_\gamma(\mathbf{u}):\mathbf{u}\in V_k,\mathbf{u}\neq 0\}\geq R_\gamma(\mathbf{v})\geq \min \{R_\gamma(\mathbf{u}):\mathbf{u}\in H^1,\mathbf{u}\neq 0,\mathbf{u}\perp_\gamma U_{k-1}\}=\lambda_k.
\end{align*}
Since $V_k$ is arbitrary, it follows that $\tilde{\lambda}_k\geq \lambda_k$.

On the other hand, let $V_k=U_k$. For any nonzero function  $\mathbf{u}=c_1\mathbf{e}_1+c_2\mathbf{e}_2+\cdots+c_k\mathbf{e}_k\in U_k$, the ordering $\lambda_1<\lambda_2\leq\cdots \leq\lambda_k$ from Lemma \ref{lem2} gives
\begin{align*}
R_\gamma(\mathbf{u})=\frac{a_\gamma(\mathbf{u},\mathbf{u})}{(\mathbf{u},\mathbf{u})_{\gamma}}=\frac{\sum_{j=1}^k c_j^2\lambda_j(\mathbf{e}_j,\mathbf{e}_j)_{\gamma}}{\sum_{j=1}^k c_j^2(\mathbf{e}_j,\mathbf{e}_j)_{\gamma}}\leq \lambda_k\,.
\end{align*}
Hence, from \eqref{eq45},
\begin{align*}
\tilde{\lambda}_k\leq\sup\{R_\gamma(\mathbf{u}):\mathbf{u}\in U_k,\mathbf{u}\neq 0\}\leq \lambda_k\,.
\end{align*}
Combining the two inequalities yields $\lambda_k=\tilde{\lambda}_k$.

For every  $V_k \subset H^1$, the function
$R_\gamma$ is continuous on the compact set
\[
\{\mathbf u\in V_k:\, (\mathbf u,\mathbf u)_\gamma=1\}.
\]
Hence the supremum in \eqref{eq45} is attained. Moreover,  the infimum  in \eqref{eq45} is attained at $U_k$.  Therefore, the   infimum and  supremum in \eqref{eq45} may be replaced by
 a minimum and a maximum, respectively. This completes the proof.
\end{proof}


We now address the eigenvalue problem \eqref{evp2}. The proof of Theorem \ref{thm2} proceeds analogously to that of Theorem \ref{thm1}. We therefore give only a sketch, organized into a corresponding sequence of lemmas.



\begin{lem}\label{lem:regularityofeigenfunction2}
Let $\hat{\mathbf{u}}=(\hat u_1,\hat u_2)\in H^1_\Sigma\, (= H^1(\Omega))$ be an
eigenfunction of \eqref{evp2} associated with an eigenvalue $\mu$. Then  
\begin{equation*}
	\hat{\mathbf{u}}\in C(\overline{\Omega}).
\end{equation*}
\end{lem}

\begin{proof}[{\bf Proof.}]
If $N=1$, the conclusion follows from the Sobolev embedding theorem. 
Assume
in what follows that $N\geq2$.   
Set
\[
a(x)=
\begin{cases}
1, & x\in\Omega_1,\\
\alpha, & x\in\Omega_2,
\end{cases}
\qquad
\hat{c}(x)=
\begin{cases}
c_1(x)-\mu, & x\in\Omega_1,\\
\alpha(c_2(x)-\mu), & x\in\Omega_2.
\end{cases}
\]
By the weak formulation of \eqref{evp2}, $\hat{\mathbf{u}}\in H^1(\Omega)$ satisfies
\begin{equation}
	\int_\Omega \bigl(a\nabla\hat {\mathbf{u}}\cdot\nabla v
+\hat{c}\hat{\mathbf{u}} v\bigr)\,\mathrm{d}x=0
\qquad\text{for all }v\in H^1(\Omega).
\end{equation}
Equivalently, $\hat{\mathbf{u}}$ is a weak solution of 
\begin{equation}\label{eq:sdjfoidfugfhoiuyhgfklioy55ttgfdg}
	\begin{cases}
-\displaystyle\sum_{i=1}^N(a(x)u_{x_i})_{x_i}+\hat c(x)u=0,
& x\in\Omega,\\
a(x)\partial_\nu u=0, & x\in\partial\Omega.
\end{cases}
\end{equation}
Since $a,\hat{c} \in L^\infty(\Omega) $ , 
Theorem 3.14 (ii) in \cite{Nittka2011Regularity} yields 
$\hat{\mathbf{u}} \in C^{\theta}(\overline{\Omega})$ for some
$\theta\in(0,1)$. Hence $\hat{\mathbf{u}}\in C(\overline{\Omega})$.
\end{proof}

\begin{lem}\label{lem5}
Let 
\begin{equation*}
	\mu_1=\inf_{\mathbf{u}\in H_\Sigma^1,\;\mathbf{u}\neq 0}\hat{R}_\alpha(\mathbf{u}),
\end{equation*}
where  $\hat{R}_\alpha(\mathbf{u})$ is as in \eqref{hatR(u)}. 
Then the following statements  hold.
\begin{enumerate}[(i)]
  \item There exists $\hat{\mathbf{e}}_1\in H_\Sigma^1$, normalized by $\langle\hat{\mathbf{e}}_1,\hat{\mathbf{e}}_1\rangle_\alpha=1$, such that $\mu_1=\hat{R}_\alpha(\hat{\mathbf{e}}_1)$.
  \item $\mu_1$ is the smallest eigenvalue of \eqref{evp2}, and $\hat{\mathbf{e}}_1$ is an eigenfunction associated with $\mu_1$.
  \item Every real-valued eigenfunction associated with $\mu_1$ has a definite sign in $\Omega$.
  \item The eigenspace of $\mu_1$ is one-dimensional.
\end{enumerate}
\end{lem}

\begin{proof}[{\bf Proof.}]
(i) As in the proof of Lemma \ref{lem1}, there exists
$\hat{\mathbf{e}}_1\in H_\Sigma^1$ 
such that 
\begin{equation*}
\langle\hat{\mathbf{e}}_1,\hat{\mathbf{e}}_1\rangle_\alpha=1\quad\text{and}\quad 	\hat{R}_\alpha(\hat{\mathbf{e}}_1)=\mu_1.
\end{equation*}

(ii) For any $\mathbf{v}=(v_1,v_2)\in H_\Sigma^1$, define
\[
G(\varepsilon)=\hat{R}_\alpha(\hat{\mathbf{e}}_1+\varepsilon\mathbf{v}).
\]
Since $\hat{\mathbf{e}}_1$ minimizes $\hat{R}_\alpha$, $G(\varepsilon)$
attains a local minimum at $\varepsilon=0$, and hence $G'(0)=0$. This is
equivalent to
\[
b_\alpha(\hat{\mathbf{e}}_1,\mathbf{v})
=\mu_1\langle\hat{\mathbf{e}}_1,\mathbf{v}\rangle_\alpha.
\]
Thus, by Definition \ref{deff2}, $(\mu_1,\hat{\mathbf{e}}_1)$ is an
eigenpair of \eqref{evp2}. As in the proof of Lemma \ref{lem1}, $\mu_1$
is the smallest eigenvalue of \eqref{evp2}.

(iii) Let $\mathbf{u}=(u_1,u_2)$ be a real-valued eigenfunction associated with
$\mu_1$. As in the proof of Lemma \ref{lem1}, $|\mathbf{u}|=(|u_1|,|u_2|)$
is also a minimizer of $\hat{R}_\alpha$, and the variational argument in
(ii) shows that it is an eigenfunction associated with $\mu_1$.  Using the identification of
$H_\Sigma^1$ with $H^1(\Omega)$, we may regard $|\mathbf{u}|$ as an element of $H^1(\Omega)$.
Moreover, Definition
\ref{deff2} implies that both $\mathbf{u}$ and $|\mathbf{u}|$ are weak solutions to \eqref{eq:sdjfoidfugfhoiuyhgfklioy55ttgfdg} with $\mu=\mu_1$.  
Applying Lemma \ref{lem:regularityofeigenfunction2}, we obtain  
$\mathbf{u}, |\mathbf{u}|\in C(\Omega)$. 
 Next, the Harnack inequality (see, e.g.,
\cite[Corollary 8.21]{2001EllipticTrudinger},
\cite[Theorem 5.10]{2011EllipticLin}) states that, for every
$\Omega'\subset\subset\Omega$, there exists a constant $C>0$ such that
\[
\sup_{\Omega'}|\mathbf{u}|\leq C\inf_{\Omega'}|\mathbf{u}|.
\]
Suppose that there exists $x_0\in\Omega$ with $|\mathbf{u}|(x_0)=0$. The
Harnack bound then forces $|\mathbf{u}|\equiv0$ on $\Omega$, a
contradiction. Hence $|\mathbf{u}|>0$ throughout $\Omega$, so
the continuity of $\mathbf{u}$ and the connectedness of $\Omega$ imply that $\mathbf{u}$ has a fixed sign in $\Omega$, which proves (iii).

(iv) The proof that the eigenspace corresponding to $\mu_1$ is
one-dimensional follows the same reasoning as in part (iv) of Lemma
\ref{lem1}.
\end{proof}

Recall that $\mathbf{u}\perp^\alpha \mathbf{v}$ if $\langle \mathbf{u},\mathbf{v} \rangle_{\alpha}=0$, and $\mathbf{u}\perp^\alpha V$ if $\langle \mathbf{u},\mathbf{v} \rangle_\alpha=0$ for all $\mathbf{v}\in V$.

\begin{lem}\label{lem6}
For problem \eqref{evp2}, all the eigenvalues are real and, counted with multiplicity,
can be enumerated as
\begin{align*}
\mu_1<\mu_2\leq\mu_3\leq\cdots\leq \mu_k\leq \cdots,\qquad \lim_{k\rightarrow\infty}\mu_k=\infty.
\end{align*}
Moreover, there exists a sequence of eigenfunctions $\{\hat{\mathbf{e}}_k\}_{k=1}^{\infty}$,
with $\hat{\mathbf{e}}_k$ corresponding to $\mu_k$ for every $k\geq1$, such that
\[
\langle\hat{\mathbf{e}}_j,\hat{\mathbf{e}}_\ell\rangle_\alpha=\delta_{j\ell}
\qquad\text{for all }j,\ell\geq1,
\]
\[
\mu_k=\hat{R}_\alpha(\hat{\mathbf{e}}_k)
=\min\{\hat{R}_\alpha(\mathbf{u}):\;\mathbf{u}\in H_\Sigma^1,\;\mathbf{u}\neq0,\;\mathbf{u}\perp^\alpha \hat{E}_{k-1}\}
\qquad\text{for every }k\geq1,
\]
where $\hat{E}_{k-1}:=\operatorname{span}\{\hat{\mathbf{e}}_1,\ldots,\hat{\mathbf{e}}_{k-1}\}$ for $k\geq2$, and $\hat{E}_0:=\{0\}$.
The sequence $\{\hat{\mathbf{e}}_k\}_{k=1}^{\infty}$ forms a basis for both $L^2$ and $H_\Sigma^1$.
\end{lem}
\begin{proof}[{\bf Proof.}]
The proof is analogous to that of Lemma \ref{lem2}.
\end{proof}

\begin{lem}[Courant type]\label{lem7}
Let $\{\mu_k\}_{k=1}^\infty$ be the eigenvalues of \eqref{evp2} given by
Lemma \ref{lem6}.
Then, for every $k\geq 1$,
\begin{equation*}
\mu_k=\max_{V_{k-1}\subset L^2} \min\{\hat{R}_\alpha(\mathbf{u}):\;\mathbf{u}\in H_\Sigma^1,\;\mathbf{u}\neq 0,\;\mathbf{u}\perp^\alpha V_{k-1}\}.
\end{equation*}
\end{lem}
\begin{proof}[{\bf Proof.}]
The proof is analogous to that of Lemma \ref{lem3}.
\end{proof}

\begin{lem}[Poincar\'{e} type]\label{lem8}
Let $\{\mu_k\}_{k=1}^\infty$ be the eigenvalues of \eqref{evp2} given by
Lemma \ref{lem6}.
Then, for every $k\geq 1$,
\begin{align*}
\mu_k=\min_{V_k\subset H_\Sigma^1} \max\{\hat{R}_\alpha(\mathbf{u}):\,\mathbf{u}\in V_k,\;\mathbf{u}\neq 0\}.
\end{align*}
\end{lem}
\begin{proof}[{\bf Proof.}]
The proof is analogous to that of Lemma \ref{lem4}.
\end{proof}

We now conclude the proof of Theorem \ref{thm2}.
\begin{proof}[{\bf Proof of Theorem \ref{thm2}.}]
Item (\ref{thm2:principal}) of Theorem \ref{thm2} is a direct consequence of Lemma \ref{lem5}. Item (\ref{thm2:regularity}) follows from Lemma \ref{lem:regularityofeigenfunction2}. Items (\ref{thm2:spectrum}) and (\ref{thm2:eigenbasis}) follow from Lemma \ref{lem6}, while parts (\ref{thm2:minmax})(\ref{thm2:courant}) and (\ref{thm2:minmax})(\ref{thm2:poincare}) are implied by Lemmas \ref{lem7} and \ref{lem8}, respectively.
\end{proof}

\begin{rem}\label{rem-generalization}
Consider the general divergence-form  interface eigenvalue problems
\begin{align}\label{eq6}
\begin{cases}
\displaystyle -\sum_{j,\,\ell=1}^N \partial_{\ell}(a_{j\ell}^{(i)}(x)\partial_{j}u_i)+c_i(x)u_i=\lambda u_i,  & x\in \Omega_i,\\
\partial_{\nu_{A_i}}u_i=\gamma_i(u_2-u_1), & x\in \Sigma,\\
\partial_{\nu_{A_i}} u_i=0, & x\in \Gamma_i,
\end{cases}
\qquad i=1,2,
\end{align}
and
\begin{align}\label{eq7}
\begin{cases}
\displaystyle -\sum_{j,\,\ell=1}^N \partial_{\ell}(a_{j\ell}^{(i)}(x)\partial_{j}u_i)+c_i(x)u_i=\mu u_i, & x\in \Omega_i,\\
\partial_{\nu_{A_1}} u_1=\alpha\partial_{\nu_{A_2}} u_2,\quad u_1=u_2, & x\in \Sigma,\\
\partial_{\nu_{A_i}} u_i=0, & x\in \Gamma_i,
\end{cases}
\qquad i=1,2.
\end{align}
 Here the  coefficients $a_{j\ell}^{(i)}\in L^\infty(\Omega_i)$ and $c_i\in L^\infty(\Omega_i)$ are real-valued. The principal coefficients satisfy  $a_{j\ell}^{(i)}(x)=a_{\ell j}^{(i)}(x)$ for almost every $x\in\Omega_i$, and  the strict ellipticity condition 
 \[\sum_{j,\,\ell=1}^{N}a_{j\ell}^{(i)}(x)\xi_j\xi_\ell\geq \lambda_0|\xi|^2\]
  holds for some $\lambda_0>0$, almost every $x\in\Omega_i$, and all $\xi=(\xi_1,\cdots,\xi_N)\in \mathbb{R}^N$. 
  The problems \eqref{eq6} and \eqref{eq7} are understood in the weak sense. For sufficiently regular coefficients and solutions, $\partial_{\nu_{A_i}}u_i = (A_i\nabla u_i)\cdot\nu$, where $A_i(x)=(a_{j\ell}^{(i)}(x))_{N\times N}$  and $\nu$ denotes the normal vector as in Section \ref{sec1}.
  

The conclusions of Theorem \ref{thm1} for problem \eqref{evp1} remain valid for problem \eqref{eq6}, and those of Theorem \ref{thm2} for problem \eqref{evp2} remain valid for problem \eqref{eq7}, provided that the bilinear forms
$a_\gamma(\cdot,\cdot)$ and $b_\alpha(\cdot,\cdot)$ are respectively replaced by $\tilde{a}_\gamma(\cdot,\cdot)$ and $\tilde{b}_\alpha(\cdot,\cdot)$, defined as
\begin{align*}
\tilde{a}_\gamma(\mathbf{u},\mathbf{v})=\sum_{i=1}^{2}\frac{1}{\gamma_i}\int_{\Omega_i}\left[\sum_{j,\,\ell=1}^{N}a_{j\ell}^{(i)}(x)\partial_{j} u_i\partial_{\ell} v_i+c_i(x)u_iv_i\right]\mathrm{d}x+\int_{\Sigma}(u_2-u_1)(v_2-v_1)\,\mathrm{d}s\,,
\end{align*}
\begin{align*}
\tilde{b}_\alpha(\mathbf{u},\mathbf{v})
&=\int_{\Omega_1}\left[\sum_{j,\,\ell=1}^{N} a_{j\ell}^{(1)}(x)\partial_{j} u_1\partial_{\ell} v_1+c_1(x)u_1v_1\right]\mathrm{d}x\\
&\quad+\alpha\int_{\Omega_2}\left[\sum_{j,\,\ell=1}^{N}a_{j\ell}^{(2)}(x)\partial_{j} u_2\partial_{\ell} v_2+ c_2(x) u_2v_2\right]\mathrm{d}x.
\end{align*}
\end{rem}

\section{Proofs of Theorems \ref{thm3}--\ref{thm6}}\label{sec4}

 This section is devoted to the dependence of the eigenvalues
$\lambda_k(\gamma)$ and $\mu_k(\alpha)$ on the interface parameters.
We first establish their continuity with respect to $\gamma$ and $\alpha$ and
then investigate their asymptotic behavior in the limiting regimes
$\gamma\to0$, $\gamma\to\infty$, $\alpha\to0$, and $\alpha\to\infty$.


For each fixed
$\mathbf{u}$, the Rayleigh quotients $R_\gamma(\mathbf{u})$ and
$\hat{R}_\alpha(\mathbf{u})$ 
in the variational characterizations of
Theorems \ref{thm1} and \ref{thm2} 
depend continuously on $\gamma$ and
$\alpha$, respectively. However, because these characterizations involve
minimization and maximization over various  admissible sets, the continuity
of the individual eigenvalues $\lambda_k(\gamma)$ and $\mu_k(\alpha)$ is
not immediate, particularly when $k>1$.

By replacing $c_i$ with $c_i+\|\mathbf{c}\|_{L^\infty}$ for $i=1,2$, all eigenvalues in the problems considered below are shifted by the same constant $\|\mathbf{c}\|_{L^\infty}$. Hence, the assertions of Theorems \ref{thm3}--\ref{thm6} are unchanged. We may therefore assume, without loss of generality, that $c_i\geq0$ in $\Omega_i$ for $i=1,2$ throughout this section.

\begin{proof}[{\bf Proof of Theorem \ref{thm3}.}]
	The proof of (ii) is analogous to that of (i), while (iii) follows from the monotonicity of
the Rayleigh quotient in $\gamma_2$ when
$\gamma_1=\alpha\gamma_2$ for a fixed $\alpha>0$. It therefore remains to
prove (i), which we do in detail.

Fix $k\geq1$ and $\gamma\in (0,\infty)^2$.  Let
$\{\gamma^{(n)}\}_{n=1}^{\infty}\subset(0,\infty)^2$ be arbitrary and satisfy
$\gamma^{(n)}\to\gamma$. To prove (i),  it suffices to prove 
\begin{equation}\label{eq:dsfio90gffdkghdfh}
\lambda_k(\gamma^{(n)})\to\lambda_k(\gamma)\quad\text{as}\quad n\to\infty. 
\end{equation}

For $\mathbf{u}\in H^1$, set
\begin{equation*}
	M_i(\mathbf{u}):=\int_{\Omega_i}\left(u_i\right)^2\,\mathrm{d}x,\qquad
A_i(\mathbf{u}):=\int_{\Omega_i}\bigl(|\nabla u_i|^2+c_i(x)\left(u_i\right)^2\bigr)\,\mathrm{d}x,
\end{equation*}
and 
\begin{equation*}
	J(\mathbf{u}):=\int_\Sigma (u_2-u_1)^2\,\mathrm{d}s.
\end{equation*}
Since $c_i\geq0$, all these quantities are nonnegative. Moreover, 
\begin{equation*}
	(\mathbf{u},\mathbf{u})_\gamma
=\sum_{i=1}^2\frac{M_i(\mathbf{u})}{\gamma_i},
\qquad
a_\gamma(\mathbf{u},\mathbf{u})
=\sum_{i=1}^2\frac{A_i(\mathbf{u})}{\gamma_i}+J(\mathbf{u}).
\end{equation*}
Define
\begin{equation*}
	\varepsilon_n:=\max_{i=1,2}\left|\frac{\gamma_i}{\gamma_i^{(n)}}-1\right|.
\end{equation*}
Then
\begin{equation*}
	(1-\varepsilon_n)(\mathbf{u},\mathbf{u})_\gamma
\leq(\mathbf{u},\mathbf{u})_{\gamma^{(n)}}
\leq(1+\varepsilon_n)(\mathbf{u},\mathbf{u})_\gamma
\end{equation*}
and
\begin{equation*}
	(1-\varepsilon_n)a_\gamma(\mathbf{u},\mathbf{u})
\leq a_{\gamma^{(n)}}(\mathbf{u},\mathbf{u})
\leq(1+\varepsilon_n)a_\gamma(\mathbf{u},\mathbf{u}).
\end{equation*}
Consequently, since $\varepsilon_n\to0$, for  sufficiently large $n$,
\begin{equation*}
	\frac{1-\varepsilon_n}{1+\varepsilon_n}R_\gamma(\mathbf{u})
\leq R_{\gamma^{(n)}}(\mathbf{u})
\leq\frac{1+\varepsilon_n}{1-\varepsilon_n}R_\gamma(\mathbf{u}),\quad\forall\, \mathbf{u}\in H^1\setminus\{0\}.
\end{equation*}
Applying the Poincar\'e-type characterization in Theorem \ref{thm1}(\ref{thm1:minmax})(\ref{thm1:poincare}), we obtain
\begin{equation*}
	\frac{1-\varepsilon_n}{1+\varepsilon_n}\lambda_k(\gamma)
\leq\lambda_k(\gamma^{(n)})
\leq\frac{1+\varepsilon_n}{1-\varepsilon_n}\lambda_k(\gamma).
\end{equation*}
Sending $n\to\infty$ yields \eqref{eq:dsfio90gffdkghdfh}, which proves (i).
\end{proof}

We now prove Theorem \ref{thm4}, which states that, as $\gamma_1,\gamma_2\rightarrow\infty$ with $\gamma_1/\gamma_2\rightarrow \alpha$, the $k$-th eigenvalue of \eqref{evp1} converges to the $k$-th eigenvalue of \eqref{evp2}.

\begin{proof}[{\bf Proof of Theorem \ref{thm4}.}]

\smallskip
\noindent {\bf Step 1. Boundedness of $\lambda_k(\gamma)$.}

Since $c_i\geq0$, we see that $R_\gamma(\mathbf{u})\ge 0$ for all $\mathbf{u}\in H^1\setminus\{0\}$, and hence by Theorem \ref{thm1}(\ref{thm1:eigenbasis}) we have
\begin{equation}\label{eq:6.14}
\lambda_k(\gamma)\geq0\qquad \forall\; k\ge 1, \; \gamma=(\gamma_1,\gamma_2)\in (0,\infty)^2.
\end{equation}
From \eqref{ip-1}-\eqref{R(u)} and  \eqref{ip-2}-\eqref{hatR(u)}, we observe that
\begin{equation*}\label{eq51}
R_\gamma(\mathbf{u})=\hat{R}_{\gamma_1/\gamma_2}(\mathbf{u})\qquad  \forall\; \mathbf{u}\in H_\Sigma^1\setminus \{0\}.
\end{equation*}
 Applying the Poincar\'{e}-type characterizations in Theorems \ref{thm1}(\ref{thm1:minmax})(\ref{thm1:poincare}) and \ref{thm2}(\ref{thm2:minmax})(\ref{thm2:poincare}) yields
\begin{align}\label{eq10}
\lambda_k(\gamma)=&\min_{V_{k}\subset H^1} \max\{R_\gamma(\mathbf{u}):\,\mathbf{u}\in V_k,\;\mathbf{u}\neq 0\}\nonumber\\
\leq & \min_{V_{k}\subset H_\Sigma^1} \max\{R_\gamma(\mathbf{u}):\,\mathbf{u}\in V_k,\;\mathbf{u}\neq 0\}\nonumber\\
=& \min_{V_{k}\subset H_\Sigma^1} \max\{\hat{R}_{\gamma_1/\gamma_2}(\mathbf{u}):\,\mathbf{u}\in V_k,\;\mathbf{u}\neq 0\}\nonumber\\
=&\mu_k(\gamma_1/\gamma_2).
\end{align}
Together with the continuity of $\mu_k$ from Theorem \ref{thm3}(ii), this gives
\begin{equation}\label{eq:6.13}
\limsup_{\gamma\to (\infty,\infty),\,\gamma_1/\gamma_2\to\alpha}\lambda_k(\gamma)\le\mu_k(\alpha)\,.
\end{equation}
This shows the uniform boundedness of $\lambda_k(\gamma)$ as $\gamma\to (\infty,\infty)$ with
$\gamma_1/\gamma_2\rightarrow\alpha$ for every fixed $k\ge 1$.


\smallskip
\noindent {\bf Step 2. Boundedness of the scaled eigenfunctions.}

Let $\mathbf{e}_j^{(\gamma)}$ ($j\ge 1$)
be the eigenfunction corresponding to
$\lambda_j(\gamma)$ given by
Theorem~\ref{thm1}(\ref{thm1:eigenbasis}). Set
\begin{equation*}
\tilde{\mathbf{e}}_j^{(\gamma)}
:=\frac{1}{\sqrt{\gamma_1}}\mathbf{e}_j^{(\gamma)}.
\end{equation*}
Then
\begin{equation}\label{eq17} \left\langle\tilde{\mathbf{e}}_j^{(\gamma)},\,\tilde{\mathbf{e}}_\ell^{(\gamma)}\right\rangle_{\gamma_1/\gamma_2}=\left(\mathbf{e}_j^{(\gamma)},\,\mathbf{e}_\ell^{(\gamma)}\right)_{\gamma}
=\delta_{j\ell},\qquad j\ge 1,\;\ell\ge 1\,.
\end{equation}




In what follows, we use $b_\beta$ also for its natural extension to
$H^1\times H^1$, defined by the same integral formula as in \eqref{bf-2}.
Then, for every $\mathbf{v}=(v_1,v_2)\in H^1$,
\begin{equation}\label{eq3}
b_{\gamma_1/\gamma_2}\bigl(\tilde{\mathbf{e}}_k^{(\gamma)},\mathbf{v}\bigr)
+\gamma_1\int_{\Sigma}
\bigl(\tilde e_{k2}^{(\gamma)}-\tilde e_{k1}^{(\gamma)}\bigr)(v_2-v_1)\,\mathrm{d}s
=\lambda_k(\gamma)
\left\langle\tilde{\mathbf{e}}_k^{(\gamma)},\mathbf{v}\right\rangle_{\gamma_1/\gamma_2}.
\end{equation}
Taking $\mathbf{v}=\tilde{\mathbf{e}}_k^{(\gamma)}$ in \eqref{eq3} gives
\begin{equation}\label{eq9}
b_{\gamma_1/\gamma_2}\bigl(\tilde{\mathbf{e}}_k^{(\gamma)},
\tilde{\mathbf{e}}_k^{(\gamma)}\bigr)
+\gamma_1\|\tilde e_{k2}^{(\gamma)}-\tilde e_{k1}^{(\gamma)}\|_{L^2(\Sigma)}^2
=\lambda_k(\gamma).
\end{equation}
Since $c_i\geq0$,  we conclude from \eqref{eq17}, \eqref{eq9},  and Step 1 that
\begin{equation} \label{eq:sdodfsgpogfj}
	\|\tilde{\mathbf{e}}_k^{(\gamma)}\|_{H^1}\leq M,\qquad
\|\tilde e_{k2}^{(\gamma)}-\tilde e_{k1}^{(\gamma)}\|_{L^2(\Sigma)}^2
\leq\frac{M}{\gamma_1}\,,
\end{equation}
for some constant $M>0$, which is independent of $\gamma$ as $\gamma\to (\infty,\infty)$ with
$\gamma_1/\gamma_2\rightarrow\alpha$.


\smallskip
\noindent {\bf Step 3. Identification of the limiting problem.}

Fix $k\ge 1$. From \eqref{eq:6.13} and \eqref{eq:sdodfsgpogfj} we may assume, after passing to a subsequence of $\gamma$, that for all $j$, $\ell=1,\dots,k$,
\begin{align}\label{eq40}
\lambda_j(\gamma)\rightarrow \lambda_j^*,\qquad \lambda_j^*\leq \mu_j(\alpha), 
\end{align}
and
\begin{equation}\label{eq4}
\tilde{\mathbf{e}}_j^{(\gamma)}\rightharpoonup \mathbf{w}_j \;\text{ in } H^1,
\qquad \tilde{\mathbf{e}}_j^{(\gamma)}\rightarrow \mathbf{w}_j \;\text{ in } L^2,
\end{equation}
\begin{equation}\label{eq52}
\langle \mathbf{w}_j,\mathbf{w}_\ell \rangle_\alpha=\delta_{j\ell}\,.
\end{equation}
Moreover, it follows from \eqref{eq:sdodfsgpogfj} and \eqref{eq4} that
\begin{equation*}
w_{j1}=w_{j2}\quad\text{on }\Sigma,\; \mbox{~so~} \mathbf{w}_j\in H_\Sigma^1\,.
\end{equation*}
For
$\mathbf{v}\in H_\Sigma^1$, the interface integral in \eqref{eq3} vanishes
identically. Therefore, passing to the limit, and  using \eqref{eq40} and \eqref{eq4}, we
obtain
\begin{equation}\label{eq55}
b_\alpha(\mathbf{w}_j,\mathbf{v})
=\lambda_j^*\langle\mathbf{w}_j,\mathbf{v}\rangle_{\alpha}
\qquad\forall\; \mathbf{v}\in H_\Sigma^1,\quad j=1,\dots,k.
\end{equation}
Hence $(\lambda_j^*,\mathbf{w}_j)$ is an eigenpair of \eqref{evp2} with
parameter $\alpha$ for every $j=1,\dots,k$.

\smallskip
\noindent {\bf Step 4. Identification of the limiting eigenvalue.}

Since the eigenvalues $\lambda_j(\gamma)$ are nondecreasing in $j$,
\eqref{eq40} yields
\[
\lambda_1^*\leq\lambda_2^*\leq\cdots\leq\lambda_k^*.
\]
Let
\[
W:=\operatorname{span}\{\mathbf w_1,\ldots,\mathbf w_k\}.
\]
By \eqref{eq52}, $\dim W=k$. Moreover, \eqref{eq55} and
$\lambda_j^*\leq\lambda_k^*$ show that $W$ is contained in the spectral
subspace of \eqref{evp2} associated with eigenvalues not exceeding
$\lambda_k^*$. Thus this spectral subspace has dimension at least $k$.
Since the eigenvalues of \eqref{evp2} are enumerated in nondecreasing order,
counted with multiplicity, by Theorem~\ref{thm2}(\ref{thm2:spectrum}), it follows directly that
\[
\mu_k(\alpha)\leq\lambda_k^*.
\]
Combining this with \eqref{eq40} yields $\lambda_k^*=\mu_k(\alpha)$.

Since the preceding argument applies to every subsequence of  $\gamma\to(\infty,\infty)$ with $\gamma_1/\gamma_2\to\alpha$,
\eqref{gamma-infty} follows.
\end{proof}

We now proceed to prove Theorem \ref{thm5}, which provides the asymptotic behavior of the eigenvalues of problem \eqref{evp1} as $\gamma\rightarrow (0,0)$.
In the sequel, we will sometimes revert to the standard $L^2$-inner product given in \eqref{L2}.


\begin{proof}[\bf Proof of Theorem \ref{thm5}.]

\noindent {\bf Step 1. Boundedness of $\lambda_k(\gamma)$}


\smallskip
For $i=1,2$, let $\sigma_k^{RN}(\Omega_i;\gamma_i)$ denote the $k$-th eigenvalue  of the Robin-Neumann problem
\begin{align}\label{eqlambdaR}
\begin{cases}
-\Delta u+c_i(x)u=\sigma u, & x\in \Omega_i,\\
(-1)^{i+1}\partial_\nu u+\gamma_i u=0, & x\in \Sigma,\\
\partial_\nu u=0, & x\in \Gamma_i.
\end{cases}
\end{align}
When $\gamma_i=0$, this reduces to the Neumann problem  \eqref{eqlambdaN}, so $\sigma_k^{RN}(\Omega_i;0)=\sigma_k^N(\Omega_i)$. Recall that $\{\sigma_k^N\}_{k=1}^\infty$ is obtained by merging $\{\sigma_k^N(\Omega_1)\}_{k=1}^\infty$ and $\{\sigma_k^N(\Omega_2)\}_{k=1}^\infty$ into a nondecreasing sequence, counting multiplicities. For every $k\ge 1$, let $\rho_k$ be the number of eigenvalues among $\sigma_1^N,\dots,\sigma_k^N$ that originate from $\{\sigma_j^N(\Omega_1)\}_{j=1}^\infty$. Then $k-\rho_k$ is the number originating from $\{\sigma_j^N(\Omega_2)\}_{j=1}^\infty$. In what follows, we assume $1\le \rho_k\le k-1$; the cases $\rho_k=0$ and $\rho_k=k$ are analogous and simpler.

Given $\gamma=(\gamma_1,\gamma_2)$, for each $i=1,2$, choose an $L^2(\Omega_i)$-orthonormal sequence $\{\phi_j^{(i)}\}_{j=1}^\infty$ of eigenfunctions of \eqref{eqlambdaR}, with $\gamma_i$ replaced by $2\gamma_i$, such that $\phi_j^{(i)}$ corresponds to $\sigma_j^{RN}(\Omega_i;2\gamma_i)$.
Regard these eigenfunctions as elements of $H^1$ by setting
 \[\Phi_j^{(1)}=(\phi^{(1)}_j,0),\qquad \Phi_j^{(2)}=(0,\phi^{(2)}_j).\]
Define 
\begin{align*}
W_k=\operatorname{span}\left\{\Phi_1^{(1)},\cdots,\Phi_{\rho_k}^{(1)},\,\Phi_1^{(2)},\cdots,\Phi_{k-\rho_k}^{(2)}\right\}.
\end{align*}
Clearly $\dim W_k=k$. By the Poincar\'{e}-type characterization in Theorem \ref{thm1}(\ref{thm1:minmax})(\ref{thm1:poincare}),
\begin{align}\label{eq1.16}
\lambda_k(\gamma)
\leq \max\{R_\gamma(\mathbf{u}):\;\mathbf{u}\in W_k,\;\mathbf{u}\neq 0\}.
\end{align}

To estimate $R_\gamma(\mathbf{u})$ for $\mathbf{u}=(u_1,u_2)\in W_k\setminus\{0\}$, we introduce the Robin-Neumann Rayleigh quotient on $\Omega_i$:
\begin{equation*}
R_{\gamma_i}^{(i)}(u_i)=\frac{\displaystyle \int_{\Omega_i}\left(|\nabla u_i|^2+c_i(x)|u_i|^2 \right)\, \mathrm{d}x+\gamma_i\int_{\Sigma}|u_i|^2 \,\mathrm{d}s}{\displaystyle\int_{\Omega_i}|u_i|^2 \,\mathrm{d}x}\,,\quad u_i\in H^1(\Omega_i)\setminus\{0\}.
\end{equation*}
If $u_2=0$, then $u_1\neq 0$ and
\begin{align*}
R_\gamma(\mathbf{u})=R_\gamma((u_1,0))=R_{\gamma_1}^{(1)}(u_1)\leq R_{2\gamma_1}^{(1)}(u_1)\,.
\end{align*}
If $u_1=0$, then $u_2\neq 0$ and
\begin{align*}
R_\gamma(\mathbf{u})=R_\gamma((0,u_2))=R_{\gamma_2}^{(2)}(u_2)\leq R_{2\gamma_2}^{(2)}(u_2)\,.
\end{align*}
If both $u_1\neq 0$ and $u_2\neq 0$, using the elementary inequality $|u_2-u_1|^2\le 2|u_1|^2+2|u_2|^2$ and the fact that $(a+b)/(c+d)\le \max\{a/c,b/d\}$ for positive denominators, we obtain
\begin{align*}
R_\gamma(\mathbf{u})
\leq \max_{i=1,\,2}\left\{R_{2\gamma_i}^{(i)}(u_i)\right\}\,.
\end{align*}
Since $u_1\in \operatorname{span}\left\{\phi^{(1)}_1,\cdots, \phi^{(1)}_{\rho_k}\right\}\setminus\{0\}$ and $u_2\in \operatorname{span}\left\{\phi^{(2)}_1,\cdots, \phi^{(2)}_{k-\rho_k}\right\}\setminus\{0\}$, it follows from the variational characterization of \eqref{eqlambdaR} that
\begin{equation*}
R_{2\gamma_1}^{(1)}(u_1)\le \sigma_{\rho_k}^{RN}(\Omega_1;2\gamma_1)\,,\qquad R_{2\gamma_2}^{(2)}(u_2)\le \sigma_{k-\rho_k}^{RN}(\Omega_2;2\gamma_2)\,.
\end{equation*}
Therefore, from \eqref{eq1.16} we have
\begin{align}\label{eq60}
\lambda_k(\gamma)\leq  \max\left\{\sigma_{\rho_k}^{RN}(\Omega_1;2\gamma_1),\;\sigma_{k-\rho_k}^{RN}(\Omega_2;2\gamma_2)\right\}\,.
\end{align}

For each fixed $i$ and $k$, the Robin-Neumann eigenvalue $\sigma_k^{RN}(\Omega_i;\gamma_i)$ depends continuously on $\gamma_i\in [0,\infty)$. This is a standard result for a single domain and follows from arguments similar
to those used in the proof of Theorem \ref{thm3}.
Letting $\gamma\to (0,0)$ in \eqref{eq60}, we get
\begin{align}\label{eq1.1}
\limsup_{\gamma\rightarrow (0,0)}\lambda_k(\gamma)
\leq  \max\left\{\sigma_{\rho_k}^N(\Omega_1),\;\sigma_{k-\rho_k}^N(\Omega_2)\right\}= \sigma_k^N.
\end{align}
Combining this with $\lambda_k(\gamma)\geq 0$ (recall \eqref{eq:6.14}), we have that $\lambda_k(\gamma)$ is uniformly bounded as $\gamma\rightarrow (0,0)$.

\smallskip

\noindent {\bf Step 2. Convergence of $\lambda_k(\gamma)$.}


For each $\gamma\in (0,\infty)^2$ and each $j\ge 1$, let $\mathbf{e}_j^{(\gamma)}$ 
be the eigenfunction corresponding to
$\lambda_j(\gamma)$ given by
Theorem~\ref{thm1}(\ref{thm1:eigenbasis}). Set
\begin{equation*}
	\mathbf{\tilde{e}}_j^{(\gamma)}
:=\frac{1}{\sqrt{\gamma_1}}\mathbf{e}_j^{(\gamma)}
=\left(\tilde{e}_{j1}^{(\gamma)},\tilde{e}_{j2}^{(\gamma)}\right).
\end{equation*}
Then 
\begin{equation}\label{eq62}
\int_{\Omega_1}\tilde{e}_{j1}^{(\gamma)}\tilde{e}_{\ell1}^{(\gamma)}\,\mathrm{d}x
+\frac{\gamma_1}{\gamma_2}\int_{\Omega_2}\tilde{e}_{j2}^{(\gamma)}\tilde{e}_{\ell2}^{(\gamma)}\,\mathrm{d}x
=\delta_{j\ell},
\qquad j,\;\ell\ge 1.
\end{equation}
Since $\lambda_j(\gamma)\leq\lambda_k(\gamma)$ for $1\leq j\leq k$, the bound from Step 1 and the eigen-equations of $\left(\lambda_j(\gamma),\mathbf{\tilde{e}}_j^{(\gamma)}\right)$ yield a constant $M>0$, independent of $\gamma$ and $j$, such that
\begin{align}\label{eq84}
\int_{\Omega_1}\left|\nabla \tilde{e}_{j1}^{(\gamma)}\right|^2\mathrm{d}x+\int_{\Omega_2}\left|\nabla \left(\sqrt{\tfrac{\gamma_1}{\gamma_2}}\tilde{e}_{j2}^{(\gamma)}\right)\right|^2\mathrm{d}x\leq M,
\qquad j=1,\ldots,k.
\end{align}
 Thus \eqref{eq62} with $\ell=j$ and \eqref{eq84} imply that $\left(\tilde{e}_{j1}^{(\gamma)},\sqrt{\tfrac{\gamma_1}{\gamma_2}}\tilde{e}_{j2}^{(\gamma)}\right)$ is uniformly bounded in $H^1$ as $\gamma\rightarrow (0,0)$ for every $j=1,\ldots,k$.
 Hence, after passing to a common subsequence, there exist $\mathbf{w}_j=(w_{j1},w_{j2})\in H^1$ and $\lambda_j^*$ such that, for each $j=1,\ldots,k$, as $\gamma\to (0,0)$,
\begin{equation}\label{eq58}
\lambda_j(\gamma)\rightarrow \lambda_j^*,
\end{equation}
\begin{equation}\label{eq:psh489rug74gur34eu}
	\left(\tilde{e}_{j1}^{(\gamma)},\sqrt{\tfrac{\gamma_1}{\gamma_2}}\tilde{e}_{j2}^{(\gamma)}\right)
\rightharpoonup (w_{j1},w_{j2}) \quad\text{in } H^1,
\end{equation}
and
\begin{equation}\label{eq59}
	\left(\tilde{e}_{j1}^{(\gamma)},\sqrt{\tfrac{\gamma_1}{\gamma_2}}\tilde{e}_{j2}^{(\gamma)}\right)
\to  (w_{j1},w_{j2}) \quad\text{in } L^2.
\end{equation}
From \eqref{eq62} and  \eqref{eq59}, we have
\begin{equation}
	\left(\mathbf{w}_j,\mathbf{w}_\ell\right)_{L^2}=\delta_{j\ell},\qquad 1\le j,\;\ell\le k.
\end{equation}

We next identify the limiting equations. Fix $j\in\{1,\ldots,k\}$. For
any $v_1\in H^1(\Omega_1)$, take $\mathbf{v}=(v_1,0)$ in the
eigen-equation for $\tilde{\mathbf e}_j^{(\gamma)}$ and multiply by
$\gamma_1$. This gives
\begin{equation*}
\int_{\Omega_1}
\left(\nabla \tilde{e}_{j1}^{(\gamma)}\cdot\nabla v_1
+c_1(x)\tilde{e}_{j1}^{(\gamma)}v_1\right)\,\mathrm{d}x
+\gamma_1\int_{\Sigma}
\left(\tilde{e}_{j1}^{(\gamma)}-\tilde{e}_{j2}^{(\gamma)}\right)v_1\,\mathrm{d}s
=\lambda_j(\gamma)\int_{\Omega_1}\tilde{e}_{j1}^{(\gamma)}v_1\,\mathrm{d}x.
\end{equation*}
The trace theorem and the bounds above show that the interface term tends to
zero: its $\tilde e_{j1}^{(\gamma)}$ part is $O(\gamma_1)$, while its
$\tilde e_{j2}^{(\gamma)}$ part is $O(\sqrt{\gamma_1\gamma_2})$.
Consequently,
\begin{equation*}
\int_{\Omega_1}\left(\nabla w_{j1}\cdot\nabla v_1+c_1(x)w_{j1}v_1\right)\,\mathrm{d}x
=\lambda_j^*\int_{\Omega_1}w_{j1}v_1\,\mathrm{d}x.
\end{equation*}

Similarly, for any $v_2\in H^1(\Omega_2)$, take
$\mathbf{v}=(0,v_2)$ in the eigen-equation for
$\tilde{\mathbf e}_j^{(\gamma)}$ and multiply by
$\sqrt{\gamma_1\gamma_2}$. We obtain
\begin{align*}
\int_{\Omega_2}
\left(
\nabla\left(\sqrt{\tfrac{\gamma_1}{\gamma_2}}\tilde{e}_{j2}^{(\gamma)}\right)
\cdot\nabla v_2
+c_2(x)\sqrt{\tfrac{\gamma_1}{\gamma_2}}\tilde{e}_{j2}^{(\gamma)}v_2
\right)\,\mathrm{d}x
&+\sqrt{\gamma_1\gamma_2}\int_{\Sigma}
\left(\tilde{e}_{j2}^{(\gamma)}-\tilde{e}_{j1}^{(\gamma)}\right)v_2\,\mathrm{d}s \\
&\qquad\qquad=\lambda_j(\gamma)\int_{\Omega_2}
\sqrt{\tfrac{\gamma_1}{\gamma_2}}\tilde{e}_{j2}^{(\gamma)}v_2\,\mathrm{d}x.
\end{align*}
Here the interface term tends to zero because its two parts are respectively
$O(\gamma_2)$ and $O(\sqrt{\gamma_1\gamma_2})$. Passing to the limit gives
\begin{equation*}
\int_{\Omega_2}\left(\nabla w_{j2}\cdot\nabla v_2+c_2(x)w_{j2}v_2\right)\,\mathrm{d}x
=\lambda_j^*\int_{\Omega_2}w_{j2}v_2\,\mathrm{d}x.
\end{equation*}
Therefore,
\begin{align}\label{eq1.20}
(\lambda_j^*,w_{ji}) \text{ is an eigenpair of the Neumann problem \eqref{eqlambdaN} on }\Omega_i \text{ if }w_{ji}\neq 0.
\end{align}
Since $\|(w_{j1},w_{j2})\|_{L^2}=1$, the two components cannot vanish
simultaneously, and hence
\[
\lambda_j^*\in \{\sigma_r^N\}_{r=1}^\infty,
\qquad j=1,\ldots,k.
\]

We claim that
\begin{equation}\label{eq63}
\lambda_j^*=\sigma_j^N,\qquad j=1,\ldots,k.
\end{equation}
We prove this by induction on $j$. Since
$\lambda_1^*\in\{\sigma_r^N\}_{r=1}^\infty$ and, by \eqref{eq1.1},
$\lambda_1^*\leq\sigma_1^N$, we have
$\lambda_1^*=\sigma_1^N$.

Suppose that
\[
\lambda_\ell^*=\sigma_\ell^N,\qquad \ell=1,\ldots,j-1,
\]
for some $2\leq j\leq k$. We show that
$\lambda_j^*=\sigma_j^N$. Assume otherwise. Since
\[
\sigma_{j-1}^N=\lambda_{j-1}^*
\leq\lambda_j^*<\sigma_j^N
\]
and $\lambda_j^*\in\{\sigma_r^N\}_{r=1}^\infty$, it follows that
\[
\lambda_j^*=\sigma_{j-1}^N=:\eta.
\]

Let $E_\eta^N(\Omega_i)$ denote the Neumann eigenspace associated with $\eta$
on $\Omega_i$, and set
\[
E_\eta:=E_\eta^N(\Omega_1)\oplus E_\eta^N(\Omega_2),
\qquad m:=\dim E_\eta.
\]
Since
$\eta=\sigma_{j-1}^N<\sigma_j^N$, the eigenvalue $\eta$ occurs precisely at
the indices $j-m,\ldots,j-1$ in the merged sequence
$\{\sigma_r^N\}_{r=1}^\infty$. By the induction hypothesis,
\[
\lambda_{j-m}^*=\cdots=\lambda_{j-1}^*=\eta.
\]
Hence, by \eqref{eq1.20},
$\mathbf{w}_{j-m},\ldots,\mathbf{w}_{j-1}\in E_\eta$. These $m$ functions
are $L^2$-orthonormal; therefore, they form an orthonormal basis of $E_\eta$.
On the other hand, $\lambda_j^*=\eta$ and \eqref{eq1.20} imply that
$\mathbf{w}_j\in E_\eta\setminus\{0\}$. This is impossible because
$\mathbf{w}_j$ is $L^2$-orthogonal to
$\mathbf{w}_{j-m},\ldots,\mathbf{w}_{j-1}$. Therefore,
$\lambda_j^*=\sigma_j^N$, completing the induction.

Since every sequence $\gamma\to (0,0)$ has a subsequence for which $\lambda_k(\gamma)\to\sigma_k^N$, the full convergence  \eqref{gamma-0} follows.
\end{proof}

We now prove Theorem \ref{thm6}, which gives the asymptotic behavior of the eigenvalues of problem \eqref{evp2} with respect to $\alpha$.

\begin{proof}[{\bf Proof of Theorem \ref{thm6}.}] It suffices to prove the result for $\alpha\rightarrow 0$; the case $\alpha\rightarrow \infty$ follows by interchanging the indices $1$ and $2$, which yields the limit $\sigma_k^{DN,N}$  from the already established $\alpha\rightarrow 0$ result.

\smallskip
\noindent {\bf Step 1. Boundedness of $\mu_k(\alpha)$}


We establish an upper bound for $\mu_k(\alpha)$ by comparing it with the
eigenvalues of the Neumann problem \eqref{eqlambdaN} on $\Omega_1$ and the
Dirichlet-Neumann problem \eqref{eqlambdaDN} on $\Omega_2$. Fix a
nondecreasing merging of $\{\sigma_j^N(\Omega_1)\}_{j=1}^\infty$ and
$\{\sigma_{j}^{DN}(\Omega_2)\}_{j=1}^\infty$, preserving the order within
each sequence and counting multiplicities, and denote the resulting sequence
by $\{\sigma_k^{N,DN}\}_{k=1}^\infty$.

For each $k\ge 1$, let $\rho_k$ be the number of the first $k$ terms of this
merged sequence that originate from
$\{\sigma_j^N(\Omega_1)\}_{j=1}^\infty$. Then $k-\rho_k$ is the number
originating from $\{\sigma_j^{DN}(\Omega_2)\}_{j=1}^\infty$. In what
follows, we assume $1\le \rho_k\le k-1$; the cases $\rho_k=0$ and
$\rho_k=k$ are analogous and simpler.

Introduce a subspace
\[
H_{\Sigma,0}^1(\Omega_2):=\{u\in H^1(\Omega_2):\, u=0 \; \mbox{on}~ \Sigma\}\,.
\]
Let $\{\phi_{\ell}^{DN}\}_{\ell=1}^\infty\subset H_{\Sigma,0}^1(\Omega_2)$
be an $L^2(\Omega_2)$-orthonormal sequence of eigenfunctions of
\eqref{eqlambdaDN}, where $\phi_{\ell}^{DN}$ corresponds to
$\sigma_{\ell}^{DN}(\Omega_2)$, and set
\[
\Phi_{\ell}^{DN}:=\left(0,\phi_{\ell}^{DN}\right)\in H_{\Sigma}^1.\]
Let $\{\phi_j^N\}_{j=1}^\infty\subset H^1(\Omega_1)$ be an
$L^2(\Omega_1)$-orthonormal sequence of eigenfunctions of \eqref{eqlambdaN},
where $\phi_j^N$ corresponds to $\sigma_j^N(\Omega_1)$. Choose an extension
$\phi_{j2}\in H^1(\Omega_2)$ such that
$\phi_{j2}=\phi_{j}^{N}$ on $\Sigma$  in the trace sense (the existence of such an extension is guaranteed by \cite[Theorem 18.40]{Leoni2009}), and define
\begin{align*}
\hat{\phi}_{j2}:=\phi_{j2}-\sum_{\ell=1}^{k-\rho_k}\left(\phi_{j2},\phi_{\ell}^{DN}\right)_{L^2(\Omega_2)}\phi_{\ell}^{DN}\in H^1(\Omega_2)\,.
\end{align*}
Then
\begin{equation}\label{eq79}
(\hat{\phi}_{j2},\phi_{\ell}^{DN})_{L^2(\Omega_2)}=0 \quad\mbox{for}\; j=1,\dots,\rho_k,\;\ell=1,\dots,k-\rho_k,
\end{equation}
 and
$\hat{\phi}_{j2}=\phi_{j}^{N}$ on $\Sigma$. Hence
\[\Phi_j^N:=\left(\phi_{j}^N,\hat{\phi}_{j2}\right)\in H_\Sigma^1.\]
Define
\begin{align*}
W_k=\operatorname{span}\left\{\Phi_1^N,\cdots,\Phi_{\rho_k}^N,\,\Phi_1^{DN},\cdots,\Phi_{k-\rho_k}^{DN}\right\}\,.
\end{align*}
Then $W_k \subset H_\Sigma^1$ and $\dim W_k=k$.
By the Poincar\'{e}-type characterization in Theorem \ref{thm2}(\ref{thm2:minmax})(\ref{thm2:poincare}),
\begin{align}\label{eq1.8}
\mu_k(\alpha)
\leq \max\left\{\hat{R}_\alpha(\mathbf{u}):\,\mathbf{u}\in W_k,\;\mathbf{u}\neq 0\right\}.
\end{align}

To estimate $\hat{R}_\alpha(\mathbf{u})$ for $\mathbf{u}\in W_k\setminus\{0\}$, introduce the Neumann Rayleigh quotient on $\Omega_1$ and the Dirichlet-Neumann Rayleigh quotient on $\Omega_2$:
\begin{equation*}
R_{N}(w):=\frac{\int_{\Omega_1}\left(|\nabla w|^2+c_1(x)|w|^2 \right)\, \mathrm{d}x}{\int_{\Omega_1}|w|^2 \,\mathrm{d}x}\,,\quad w\in H^1(\Omega_1)\setminus\{0\}\,;
\end{equation*}
\begin{equation*}
R_{DN}(v):=\frac{\int_{\Omega_2}\left(|\nabla v|^2+c_2(x)|v|^2 \right)\, \mathrm{d}x}{\int_{\Omega_2}|v|^2 \,\mathrm{d}x}\,,\quad v\in H_{\Sigma,0}^1(\Omega_2)\setminus\{0\}\,.
\end{equation*}
For every $\mathbf{u}\in W_k\setminus\{0\}$, write $\mathbf{u}=\mathbf{w}+\mathbf{v}$ with $\mathbf{w}=(w_1,w_2)$ and $\mathbf{v}=(0,v_2)$, where
\begin{eqnarray}
&&(w_1,w_2)=\left(\sum_{j=1}^{\rho_k}d_j\phi_{j}^N, \,\sum_{j=1}^{\rho_k}d_j\hat{\phi}_{j2}\right),\quad d_1,\ldots,d_{\rho_k}\in \mathbb{R},\label{eq65}
\\ && v_2\in \operatorname{span}\left\{\phi_{1}^{DN},\cdots,\phi_{k-\rho_k}^{DN} \right\}.\label{eq65-2}
\end{eqnarray}
Since the $\phi_{j}^N$'s are linearly independent, $w_1=0$ implies $d_j=0$ for all $j$, and hence $w_2=0$. Thus, for $\mathbf{u}\neq 0$, exactly one of the following holds:

 \noindent (i) $\mathbf{u}=\mathbf{v}$ with $v_2\neq 0$; \; (ii) $\mathbf{u}=\mathbf{w}$ with $w_1\neq 0$; \; (iii) $\mathbf{u}=\mathbf{w}+\mathbf{v}$ with $w_1\neq 0$, $v_2\neq 0$.

In case (i),  the variational characterization of the eigenvalues $\sigma_{\ell}^{DN}(\Omega_2)$ gives
\begin{align}\label{eq1.9}
\hat{R}_\alpha(\mathbf{v})=R_{DN}(v_2)\le \sigma_{k-\rho_k}^{DN}(\Omega_2)\,.
\end{align}

In case (ii), by homogeneity of $\hat{R}_\alpha$, we may assume $\sum_{j=1}^{\rho_k}(d_j)^2=1$.
Together with the orthonormality relation
$$(\phi_j^N,\phi_{q}^N)_{L^2(\Omega_1)}=\delta_{jq},$$
this yields
\begin{align}\label{eq83}
\int_{\Omega_1}|w_1|^2\,\mathrm{d}x=1.
\end{align}
Moreover, \eqref{eq65} implies that
\begin{equation}\label{eq:dsoihfio987fds6dfgf}
	\int_{\Omega_2} \left( |\nabla w_2|^2 +c_2(x)|w_2|^2 \right)\, \mathrm{d}x \le M,
\end{equation}
where $M>0$ is a constant independent of $\mathbf{u}$ and $\alpha$.

By a direct calculation,
using $c_i\ge 0$,
we get
\begin{equation}\label{eq:ashdio9fdshoihdfyfdsjh}
\begin{array}{rcl}
\displaystyle \hat{R}_\alpha(\mathbf{w})&=&
\displaystyle  \frac{\displaystyle \int_{\Omega_1}\left(|\nabla w_1|^2+c_1(x)|w_1|^2\right)\,\mathrm{d}x}
	{\displaystyle \int_{\Omega_1} |w_1|^2\,\mathrm{d}x + \alpha\int_{\Omega_2}|w_2|^2\,\mathrm{d}x}
	+  \frac{\displaystyle \alpha \int_{\Omega_2}\left(|\nabla w_2|^2+c_2(x) |w_2|^2\right)\,\mathrm{d}x}
	{\displaystyle \int_{\Omega_1} |w_1|^2\,\mathrm{d}x + \alpha\int_{\Omega_2}|w_2|^2\,\mathrm{d}x}\\[0.8cm]
	&\le &  \displaystyle \frac{\displaystyle \int_{\Omega_1}\left[|\nabla w_1|^2+\left(c_1(x) \right)|w_1|^2\right]\,\mathrm{d}x}
	{\displaystyle \int_{\Omega_1} |w_1|^2\,\mathrm{d}x }
	 +  \frac{\displaystyle \alpha \int_{\Omega_2}\left[|\nabla w_2|^2+\left(c_2(x) \right)|w_2|^2\right]\,\mathrm{d}x}
	{\displaystyle \int_{\Omega_1} |w_1|^2\,\mathrm{d}x }\\[0.8cm]
	&\le&\displaystyle   R_{N}(w_1)  + M\alpha,
\end{array}
\end{equation}
where we have used \eqref{eq83} and \eqref{eq:dsoihfio987fds6dfgf} to obtain the last inequality.
Since $R_{N}(w_1)\le \sigma_{\rho_k}^{N}(\Omega_1)$ by the variational characterization of the eigenvalues $\sigma_{j}^{N}(\Omega_1)$, we deduce from \eqref{eq:ashdio9fdshoihdfyfdsjh}
that
\begin{equation}\label{eq66}
	\hat{R}_\alpha(\mathbf{w}) \le \sigma_{\rho_k}^{N}(\Omega_1)+M\alpha.
\end{equation}

In case (iii),
\begin{align}\label{eq80}
\hat{R}_\alpha(\mathbf{u})=\hat{R}_\alpha(\mathbf{w}+\mathbf{v})=\frac{b_\alpha\left( \mathbf{w}, \mathbf{w}\right)+b_\alpha\left( \mathbf{v}, \mathbf{v}\right) +2b_\alpha\left( \mathbf{w}, \mathbf{v}\right)}{\langle \mathbf{w}, \mathbf{w}\rangle_{\alpha}+\langle \mathbf{v}, \mathbf{v}\rangle_{\alpha}+2\langle \mathbf{w}, \mathbf{v}\rangle_{\alpha}}\,.
\end{align}
From \eqref{eq79},  \eqref{eq65}, and \eqref{eq65-2}, we have
\begin{equation}\label{eq81}
\langle \mathbf{w}, \mathbf{v}\rangle_{\alpha}=\alpha (w_2,v_2)_{L^2(\Omega_2)}=0\,.
\end{equation}
Since $\mathbf{v}=(0,v_2)$, it follows that
\begin{equation}\label{eq:dsioffg04g5345jjl}
	\left|b_\alpha\left( \mathbf{w}, \mathbf{v}\right) \right|
	= \alpha \left|\int_{\Omega_2}\left(\nabla w_2\cdot\nabla v_2+ c_2(x) w_2v_2\right)\,\mathrm{d}x \right| \le C_1 \alpha \|w_2\|_{H^1(\Omega_2)}\|v_2\|_{H^1(\Omega_2)}.
\end{equation}
The spaces spanned by $\{\hat{\phi}_{j2}\}_{j=1}^{\rho_k}$ and  $\{\phi_j^{DN}\}_{j=1}^{k-\rho_k}$ are finite-dimensional and independent of $\alpha$, so $\|\cdot\|_{H^1(\Omega_2)}$ and $\|\cdot\|_{L^2(\Omega_2)}$ are equivalent on them. Consequently,
\begin{equation}\label{eq:dsoif09foigh}
	\|w_2\|_{H^1(\Omega_2)} \le C_2 \|w_2\|_{L^2(\Omega_2)}, \quad \|v_2\|_{H^1(\Omega_2)} \le C_2 \|v_2\|_{L^2(\Omega_2)}.
\end{equation}
%
Using  \eqref{eq65}, we obtain
\begin{equation}\label{eq:sdfiojhdfio}
	\|w_2\|_{L^2(\Omega_2)}\le C_3 \|w_1\|_{L^2(\Omega_1)}.
\end{equation}
Plugging \eqref{eq:dsoif09foigh} and \eqref{eq:sdfiojhdfio} into \eqref{eq:dsioffg04g5345jjl} and
applying the Cauchy-Schwarz inequality yields
\begin{equation}\label{eq82}
	\left|b_\alpha\left( \mathbf{w}, \mathbf{v}\right) \right|\le \frac{1}{2}C_4\sqrt{\alpha}
	\left(\int_{\Omega_1} |w_1|^2  \,\mathrm{d}x + \alpha \int_{\Omega_2} |v_2|^2\,\mathrm{d}x\right).
\end{equation}
Therefore,
\begin{equation}\label{eq:sdjfopig897gfdiog}
	\left|
	\frac{2b_{\alpha}\left( \mathbf{w},\mathbf{v}\right) }{\langle \mathbf{w},\mathbf{w}\rangle_{\alpha}
	+ \langle \mathbf{v},\mathbf{v}\rangle_{\alpha} }  \right|
	\le  \frac{\displaystyle C_4\sqrt{\alpha} \left(\int_{\Omega_1} |w_1|^2  \,\mathrm{d}x + \alpha \int_{\Omega_2} |v_2|^2\,\mathrm{d}x\right) }
	{\displaystyle \int_{\Omega_1}|w_1|^2\,\mathrm{d}x + \alpha \int_{\Omega_2} |w_2|^2\,\mathrm{d}x + \alpha \int_{\Omega_2} |v_2|^2\,\mathrm{d}x}\le C_4\sqrt{\alpha}.
\end{equation}
The  constants $C_i \, (i=1,\, 2,\,3,\,4)$ above are independent of $\alpha$, $\mathbf{w}$ and $\mathbf{v}$.

Using \eqref{eq80}, \eqref{eq81}, \eqref{eq:sdjfopig897gfdiog} and the elementary inequality  $(a+b)/(c+d)\le \max\{a/c,b/d\}$ for positive denominators, we have
\begin{equation*}
\hat{R}_\alpha(\mathbf{u})\le \frac{b_\alpha\left( \mathbf{w}, \mathbf{w}\right)+b_\alpha\left( \mathbf{v}, \mathbf{v}\right)}{\langle \mathbf{w}, \mathbf{w}\rangle_{\alpha}+\langle \mathbf{v}, \mathbf{v}\rangle_{\alpha}}+C_4\sqrt{\alpha}\le \max\left\{\frac{b_\alpha\left( \mathbf{w}, \mathbf{w}\right)}{\langle \mathbf{w}, \mathbf{w}\rangle_{\alpha}}, \frac{b_\alpha\left( \mathbf{v}, \mathbf{v}\right)}{\langle \mathbf{v}, \mathbf{v}\rangle_{\alpha}}\right\}+C_4\sqrt{\alpha}\,.
\end{equation*}
Combining this  with \eqref{eq1.9} and  \eqref{eq66}, we obtain
\begin{align}\label{eq70}
\hat{R}_\alpha(\mathbf{u})\le \max\left\{\sigma_{\rho_k}^{N}(\Omega_1)+M\alpha ,\;\sigma_{k-\rho_k}^{DN}(\Omega_2)\right\}+C_4\sqrt{\alpha}\,.
\end{align}
Combining the estimates in cases (i)--(iii), namely \eqref{eq1.9},
\eqref{eq66}, and \eqref{eq70}, with \eqref{eq1.8}, we obtain
\begin{align*}
\mu_k(\alpha) \leq \max\left\{\sigma_{\rho_k}^{N}(\Omega_1)+M\alpha,\;\sigma_{k-\rho_k}^{DN}(\Omega_2)\right\}+C_4\sqrt{\alpha}\,.
\end{align*}
Letting $\alpha\rightarrow 0$, we conclude that
\begin{align}\label{eq1.11}
\limsup_{\alpha\rightarrow 0}\mu_k(\alpha) \leq \max\left\{\sigma_{\rho_k}^{N}(\Omega_1),\;\sigma_{k-\rho_k}^{DN}(\Omega_2)\right\}=\sigma_k^{N,DN}\,.
\end{align}
On the other hand, since $c_i\ge 0$ for $i=1,2$,   $\mu_k(\alpha)\geq 0$.


\smallskip

\noindent {\bf Step 2. Convergence of $\mu_k(\alpha)$.}

In Step 1, we have already shown that
\begin{equation}
	\limsup_{\alpha\to 0}\mu_k(\alpha) \le \sigma_k^{N,DN}\,.
\end{equation}
Therefore, to complete the proof, we need only show that
\begin{equation}
	\liminf_{\alpha\to 0}\mu_k(\alpha) \ge  \sigma_k^{N,DN}\,.
\end{equation}

Choose a sequence $\{\alpha_n\}_{n=1}^\infty\subset(0,\infty)$  such that,  as $n\to\infty$,
\begin{equation}
	\alpha_n\to 0,\quad \text{and}\quad \mu_k(\alpha_n)\to \liminf_{\alpha\to 0}\mu_k(\alpha).
\end{equation}
For each $n$, choose, as in Theorem~\ref{thm2}(\ref{thm2:eigenbasis}), an
$\alpha_n$-orthonormal family of eigenfunctions
$\{\mathbf{e}_j^{(\alpha_n)}\}_{j=1}^\infty$, where
$\mathbf{e}_j^{(\alpha_n)}$
corresponds to $\mu_j(\alpha_n)$.  Thus,
\begin{equation}\label{eq72}
\left\langle \mathbf{e}_j^{(\alpha_n)},
\mathbf{e}_\ell^{(\alpha_n)}\right\rangle_{\alpha_n}
=\int_{\Omega_1}e_{j1}^{(\alpha_n)}e_{\ell1}^{(\alpha_n)}\,\mathrm{d}x
+\alpha_n\int_{\Omega_2}e_{j2}^{(\alpha_n)}e_{\ell2}^{(\alpha_n)}\,\mathrm{d}x
=\delta_{j\ell},
\qquad  j,\;\ell\ge 1.
\end{equation}



The argument in Step 1 also applies to all $1\le j\le k$. As a result,
$\{\mu_j(\alpha_n)\}_{n=1}^\infty$ is uniformly bounded for $1\leq j\leq k$.
Taking $\mathbf{v}=\mathbf{e}_j^{(\alpha_n)}$ in the eigen-equation associated with $\mu_j(\alpha_n)$ and using  \eqref{eq72}  gives
\begin{equation*}
	\int_{\Omega_1}\left|\nabla e_{j1}^{(\alpha_n)}\right|^2\,\mathrm{d}x
+\int_{\Omega_2}\left|\nabla
\left(\sqrt{\alpha_n}e_{j2}^{(\alpha_n)}\right)\right|^2\,\mathrm{d}x
\leq M_1,
\qquad 1\leq j\leq k,
\end{equation*}
for some constant $M_1>0$ independent of $n$ and $j$.


Together with
\eqref{eq72}, this shows that $\{\left(e_{j1}^{(\alpha_n)},\sqrt{\alpha_n}e_{j2}^{(\alpha_n)}\right)
\}_{n=1}^\infty$ is  uniformly bounded in $H^1$  for  $j=1,\ldots,k$.  Consequently,
there exist $\mu_1^*\le \mu_2^*\le \ldots\le \mu_k^*$ and $\mathbf{w}_1,\mathbf{w}_2,\ldots,\mathbf{w}_k\in H^1$ such that, after passing to a subsequence,  as $n\to\infty$,
\begin{equation*}
	\mu_j(\alpha_n)\to\mu_j^*,
\end{equation*}
\begin{equation*}
	\left(e_{j1}^{(\alpha_n)},\sqrt{\alpha_n}e_{j2}^{(\alpha_n)}\right)
\rightharpoonup (w_{j1},w_{j2}) \quad\text{in }H^1,
\end{equation*}
\begin{equation*}
	\left(e_{j1}^{(\alpha_n)},\sqrt{\alpha_n}e_{j2}^{(\alpha_n)}\right)
\to (w_{j1},w_{j2}) \quad\text{in }L^2,
\end{equation*}
\begin{equation*}
\left(e_{j1}^{(\alpha_n)}\big|_\Sigma,\sqrt{\alpha_n}e_{j2}^{(\alpha_n)}\big|_\Sigma\right)
\to \left(w_{j1}\big|_\Sigma,w_{j2}\big|_\Sigma\right)
\quad\text{in }L^2(\Sigma)\times L^2(\Sigma),	
\end{equation*}
for $j=1,\ldots,k$.	Sending $n\to\infty$  in \eqref{eq72} yields
\begin{equation}\label{eq73}
(\mathbf{w}_j,\mathbf{w}_\ell)_{L^2}=\delta_{j\ell},
\qquad 1\leq j,\;\ell\leq k.
\end{equation}



We next identify the limiting problems. Fix $j\in\{1,\ldots,k\}$. Taking
an arbitrary $\mathbf{v}=(v_1,v_2)\in H_\Sigma^1$ in the eigen-equation for
$\mathbf{e}_j^{(\alpha_n)}$ and letting $n\to\infty$, we obtain
\[
\int_{\Omega_1}\left(\nabla w_{j1}\cdot\nabla v_1+c_1w_{j1}v_1\right)
\,\mathrm{d}x
=\mu_j^*\int_{\Omega_1}w_{j1}v_1\,\mathrm{d}x.
\]
Since $\Omega_1$ is a Lipschitz domain, the Sobolev extension theorem provides, for every $v_1\in H^1(\Omega_1)$, an extension $\widetilde v_1\in H^1(\mathbb R^N)$. Setting $v_2=\widetilde v_1|_{\Omega_2}$ gives $(v_1,v_2)\in H_\Sigma^1$, since the two components have the same trace on $\Sigma$. Thus the preceding identity holds for every $v_1\in H^1(\Omega_1)$, and $w_{j1}$ is either zero or a Neumann eigenfunction on $\Omega_1$ associated with $\mu_j^*$. Moreover, the trace condition for
$\mathbf{e}_j^{(\alpha_n)}$ gives
\[
\left(\sqrt{\alpha_n}e_{j2}^{(\alpha_n)}\right)\big|_\Sigma
=\sqrt{\alpha_n}\left(e_{j1}^{(\alpha_n)}\right)\big|_\Sigma\to0
\quad\text{in }L^2(\Sigma).
\]
Hence $w_{j2}\in H_{\Sigma,0}^1(\Omega_2)$. Taking
$\mathbf{v}=(0,v_2)$ with $v_2\in H_{\Sigma,0}^1(\Omega_2)$ in the
eigen-equation, dividing by $\sqrt{\alpha_n}$, and passing to the limit
gives
\[
\int_{\Omega_2}\left(\nabla w_{j2}\cdot\nabla v_2+c_2w_{j2}v_2\right)
\,\mathrm{d}x
=\mu_j^*\int_{\Omega_2}w_{j2}v_2\,\mathrm{d}x.
\]
Thus $w_{j2}$ is either zero or a Dirichlet-Neumann eigenfunction on
$\Omega_2$ associated with $\mu_j^*$. Since \eqref{eq73} gives
$\|\mathbf{w}_j\|_{L^2}=1$, the two components cannot vanish
simultaneously. Consequently,
\[
\mu_j^*\in\{\sigma_r^{N,DN}\}_{r=1}^\infty,
\qquad j=1,\ldots,k.
\]

We claim that
\[
\mu_j^*=\sigma_j^{N,DN},
\qquad j=1,\ldots,k.
\]
By \eqref{eq1.11}, $\mu_j^*\leq\sigma_j^{N,DN}$ for every
$j=1,\ldots,k$. Since the eigenvalues $\mu_j(\alpha_n)$ are nondecreasing
in $j$, the limits $\mu_j^*$ are also nondecreasing. Since $\mu_1^*$ belongs to the merged spectrum, this gives
$\mu_1^*=\sigma_1^{N,DN}$. Suppose that
\[
\mu_\ell^*=\sigma_\ell^{N,DN},
\qquad \ell=1,\ldots,j-1,
\]
for some $2\leq j\leq k$. If $\mu_j^*\neq\sigma_j^{N,DN}$, then
\[
\sigma_{j-1}^{N,DN}=\mu_{j-1}^*
\leq\mu_j^*<\sigma_j^{N,DN}.
\]
Since $\mu_j^*$ belongs to the merged spectrum, it follows that
$\mu_j^*=\sigma_{j-1}^{N,DN}=:\eta$. Let
\[
E_\eta:=E_\eta^N(\Omega_1)\oplus E_\eta^{DN}(\Omega_2),
\qquad m:=\dim E_\eta,
\]
where $E_\eta^N(\Omega_1)$ and $E_\eta^{DN}(\Omega_2)$ denote the
corresponding single-domain eigenspaces. Since
$\eta=\sigma_{j-1}^{N,DN}<\sigma_j^{N,DN}$, the eigenvalue $\eta$ occurs
precisely at the indices $j-m,\ldots,j-1$ in the merged sequence. The
induction hypothesis and the limiting problems show that
$\mathbf{w}_{j-m},\ldots,\mathbf{w}_{j-1}\in E_\eta$. By \eqref{eq73},
these $m$ functions are $L^2$-orthonormal and hence form an orthonormal basis
of $E_\eta$. On the other hand, the limiting problems and
$\mu_j^*=\eta$ imply that $\mathbf{w}_j\in E_\eta\setminus\{0\}$, which
contradicts its $L^2$-orthogonality to
$\mathbf{w}_{j-m},\ldots,\mathbf{w}_{j-1}$. This proves the claim.

Since $\mu_k(\alpha_n)\to\liminf\limits_{\alpha\to 0}\mu_k(\alpha)$ and the
subsequence considered above satisfies
$\mu_k(\alpha_n)\to\mu_k^*=\sigma_k^{N,DN}$, we obtain
\[
\liminf_{\alpha\to0}\mu_k(\alpha)=\sigma_k^{N,DN}.
\]
Combining this with \eqref{eq1.11}, we conclude that
\[
\mu_k(\alpha)\to\sigma_k^{N,DN}\qquad\text{as }\alpha\to0,
\]
which is the first convergence in \eqref{alpha-0,infty}.
\end{proof}

\begin{rem}\label{rem-g2}
Under the coefficient assumptions in Remark \ref{rem-generalization},
Theorems \ref{thm3}--\ref{thm6} extend to the divergence-form interface
problems \eqref{eq6} and \eqref{eq7}, with the principal coefficients
and potentials held fixed as the interface parameters vary.
The limiting problems \eqref{eqlambdaN} and \eqref{eqlambdaDN}
are replaced by their corresponding divergence-form versions.
\end{rem}

\section{Proof of Theorem \ref{thmnodal}}\label{sec5}

In this section, we prove the nodal-domain estimates in Theorem \ref{thmnodal}. We first record a lemma on restricting Sobolev functions to nodal domains. The
continuity results in Lemmas \ref{lem:regularityofeigenfunction} and \ref{lem:regularityofeigenfunction2}  ensure that Definition \ref{def-nodal} applies to the
eigenfunctions under consideration.
We then prove the estimates for \eqref{evp1} and \eqref{evp2} separately, taking account of their different interface structures. Finally, we discuss the corresponding results for general divergence-form interface problems.

\begin{lem}\label{lem:nodal-restriction}
	Let $\mathbf u=(u_1,u_2)$ satisfy $u_i\in H^1(\Omega_i)\cap C(\overline{\Omega_i})$ for $i=1,2$. Let $D$ be  a connected component
	of $\Omega\setminus Z(\mathbf{u})$, set
$E_i=D\cap\Omega_i$, and define
$v_i=u_i\mathbf1_{E_i}$ in $\Omega_i$.
Then
\begin{equation}\label{eq:w5.3}
	v_i\in H^1(\Omega_i)\cap C(\overline{\Omega_i}),
\qquad
\nabla v_i=\mathbf1_{E_i}\nabla u_i
\quad\text{a.e. in }\Omega_i.
\end{equation}
%
%
%
%
\end{lem}
\begin{proof}[{\bf Proof. }]
	If $E_i = \emptyset$, the conclusion is trivial. Therefore, from now on, we assume $E_i\neq \emptyset$.	
Every connected component of $\{u_i\ne0\}\subset\Omega_i$
that meets $E_i$ is contained in $E_i$.
Indeed, it is a connected subset of
$\Omega\setminus Z(\mathbf u)$ meeting $D$. Without loss of generality, suppose $u_i>0$ in $E_i$. 

For $\varepsilon>0$, define  
\[
v_{i,\varepsilon}=\mathbf 1_{E_i}(u_i-\varepsilon)_+\quad \text{in}\quad\Omega_i. 
\]
Since $u_i=0$ on $\partial E_i\cap \Omega_i$, continuity implies that, in a neighborhood of $\partial E_i\cap \Omega_i$, 
$v_{i,\varepsilon}=0$.  Hence $v_{i,\varepsilon}\in H^1(\Omega_i)$,
and the Sobolev chain rule gives
\[
\nabla v_{i,\varepsilon}
=
\mathbf1_{E_i}\mathbf1_{\{u_i>\varepsilon\}}\nabla u_i
\quad\text{a.e. in }\Omega_i.
\]
Moreover,  $|v_{i,\varepsilon}|\le |u_i|$  and $|\nabla v_{i,\varepsilon}|\le|\nabla u_i|$. The dominated convergence theorem therefore yields
\[
v_{i,\varepsilon}\longrightarrow u_i\mathbf1_{E_i}
\quad\text{in }L^2(\Omega_i),
\qquad
\nabla v_{i,\varepsilon}
\longrightarrow \mathbf1_{E_i}\nabla u_i
\quad\text{in }L^2(\Omega_i;\mathbb R^N).
\]
This proves \eqref{eq:w5.3}.	
\end{proof}

By Lemmas \ref{lem:regularityofeigenfunction} and
\ref{lem:regularityofeigenfunction2}, Definition \ref{def-nodal} is
applicable to eigenfunctions of \eqref{evp1} and \eqref{evp2}. Next, we
prove Theorem \ref{thmnodal} separately for \eqref{evp1} and \eqref{evp2}.


\begin{proof}[\bf Proof of  $\omega(\mathbf{u})\le k$.] 
%
 
 For $k=1$,  Lemma \ref{lem1}(iii) gives $Z(\mathbf{u})=\emptyset$, and hence $\omega(\mathbf u)=1$. For $k\ge 2$, we argue by contradiction.
 

%
 Suppose that $\mathbf{u}=(u_{1},u_{2})$ possesses at least $k+1$
 distinct nodal domains, and choose $D_1,D_2,\ldots,D_{k+1}$ among them.
 For each $j=1,\ldots,k$, define $\mathbf{v}_j=(v_{j1},v_{j2})$ by
\begin{equation}\label{eq78}
v_{ji}(x)=
\begin{cases}
u_{i}(x), & x\in\Omega_i\cap D_j,\\
0, & x\in\Omega_i\setminus D_j,\\
\end{cases}
\quad i=1,2. 
\end{equation}
By Lemma \ref{lem:nodal-restriction},
$v_{ji}\in H^1(\Omega_i)$ and
\[
\nabla v_{ji}
=
\mathbf1_{D_j\cap\Omega_i}\nabla u_i
\quad\text{a.e. in }\Omega_i.
\]
Hence $\mathbf v_j\in H^1$.
All interface values of these restrictions below are understood
as Sobolev traces from the corresponding subdomains.


For $j\ne\ell$, the disjointness of the supports in each
subdomain, together with the gradient formula in
Lemma \ref{lem:nodal-restriction}, gives
\begin{align}\label{eq1.17}
a_\gamma(\mathbf v_j,\mathbf v_\ell)
&=\int_\Sigma
(v_{j2}-v_{j1})(v_{\ell2}-v_{\ell1})\,\mathrm ds
=: \int_\Sigma I(x)\,\mathrm ds.
\end{align}
We now show that $I\ge0$ almost everywhere on $\Sigma$.
By \eqref{eq78}, 
each $v_{ji}(x)$ is either $u_i(x)$ or zero.
Moreover, if $u_i(x)\ne0$, a sufficiently small connected
neighborhood on the $\Omega_i$ side belongs to a unique
nodal domain. Hence, for $j\ne\ell$,
\[
v_{ji}v_{\ell i}=0
\quad\text{a.e. on }\Sigma,\qquad i=1,2.
\]
Consequently,
\[
I=-v_{j2}v_{\ell1}-v_{j1}v_{\ell2}
\quad\text{a.e. on }\Sigma.
\]
If $u_1u_2=0$, both products on the right-hand side vanish.
If $u_1u_2>0$, continuity implies that the local pieces on
both sides of the interface belong to the same nodal domain.
Thus neither product can be nonzero for distinct $D_j$ and
$D_\ell$, and again $I=0$.
Finally, if $u_1u_2<0$, each product on the right-hand side
is either zero or $u_1u_2<0$, so $I\ge0$.
It follows that
\begin{align}\label{eq76}
a_\gamma(\mathbf v_j,\mathbf v_\ell)\ge0,
\qquad j\ne\ell.
\end{align}

Now set $V=\operatorname{span}\{\mathbf{v}_1,\mathbf{v}_2,\cdots,\mathbf{v}_k\}$ and
$\mathbf r=\mathbf u-\sum_{j=1}^k\mathbf v_j$. Since this is a finite
sum, $\mathbf r\in H^1$, and $\dim V=k$. From the disjoint supports of
the $\mathbf{v}_j$'s, we have
\begin{equation}\label{eq77}
(\mathbf{v}_j,\mathbf{v}_{\ell})_{\gamma}=0 \;\;(j\neq \ell)\,,\quad
(\mathbf r,\mathbf v_j)_\gamma=0\,,\quad
(\mathbf{u},\mathbf{v}_j)_{\gamma}=(\mathbf{v}_j,\mathbf{v}_j)_{\gamma}\,.
\end{equation}
The gradient formula in Lemma \ref{lem:nodal-restriction}
also shows that the bulk terms in
$a_\gamma(\mathbf v_j,\mathbf r)$ vanish.
On each side of the interface, at most one of the selected
nodal restrictions has a nonzero trace at a given point.
Thus $r_i$ is either $u_i$ or zero, and
$v_{ji}r_i=0$ almost everywhere on $\Sigma$.
Therefore,
\[
a_\gamma(\mathbf v_j,\mathbf r)
=-\int_\Sigma(v_{j2}r_1+v_{j1}r_2)\,\mathrm ds.
\]
Where $u_1u_2>0$, the local pieces on both sides belong to
the same nodal domain, so the two products in this integral
vanish. Where $u_1u_2=0$, they also vanish.
Where $u_1u_2<0$, each product is either zero or $u_1u_2$.
Consequently,
\begin{equation}\label{eq:remainder-sign}
a_\gamma(\mathbf v_j,\mathbf r)\geq0,
\qquad j=1,\ldots,k.
\end{equation}
Since $\dim V=k$, we may choose a nonzero $\mathbf{w}=\sum\limits_{j=1}^k c_j\mathbf{v}_j\in V$ satisfying
\begin{align*}
\mathbf{w}\perp_\gamma \operatorname{span}\{\mathbf{e}_1,\mathbf{e}_2,\cdots,\mathbf{e}_{k-1}\}\,.
\end{align*}
Then by the variational characterization of $\lambda_k$,
\begin{align*}
R_\gamma(\mathbf{w})\geq \min\{R_\gamma(\mathbf{u}):\,\mathbf{u}\in H^1,\;\mathbf{u}\neq 0,\;\mathbf{u}\perp_\gamma \operatorname{span}\{\mathbf{e}_1,\mathbf{e}_2,\cdots,\mathbf{e}_{k-1}\}\}=\lambda_k\,.
\end{align*}
On the other hand, 
\begin{equation*}
	R_{\gamma}(\mathbf{w})
= \frac{\sum\limits_{j=1}^k \sum\limits_{\ell=1}^k c_j c_\ell a_{\gamma}(\mathbf{v}_j, \mathbf{v}_\ell)}
        {\sum\limits_{j=1}^k (c_j)^2 (\mathbf{v}_j, \mathbf{v}_j)_\gamma}=: \frac{D}{S}\,.
\end{equation*}
Using \eqref{eq76}, \eqref{eq77}, \eqref{eq:remainder-sign}, the symmetry of $a_{\gamma}$, and the eigen-equation of $(\lambda_k,\mathbf{u})$, we estimate
\begin{eqnarray*}
D&\le& \sum_{j=1}^k \sum_{\ell=1}^k \frac{1}{2} \left[(c_j)^2 + (c_\ell)^2\right] a_{\gamma}(\mathbf{v}_j, \mathbf{v}_\ell)=\sum_{j=1}^k \sum_{\ell=1}^k  (c_j)^2 a_{\gamma}(\mathbf{v}_j, \mathbf{v}_\ell)\\
&\le &\sum_{j=1}^k (c_j)^2 a_{\gamma}\left(\mathbf{v}_j, \sum_{\ell=1}^k \mathbf{v}_\ell+\mathbf r\right)=\sum_{j=1}^k (c_j)^2 a_{\gamma}(\mathbf{v}_j,\mathbf{u})=\sum_{j=1}^{k}(c_j)^2\lambda_k(\mathbf{v}_j,\mathbf{u})_\gamma=\lambda_kS\,.
\end{eqnarray*}
Thus $R_{\gamma}(\mathbf{w}) = \lambda_k$.
Let
\[
E_{k-1}=\operatorname{span}
\{\mathbf e_1,\ldots,\mathbf e_{k-1}\},
\qquad
X=\{\mathbf v\in H^1:
(\mathbf v,\mathbf e_j)_\gamma=0,\ j=1,\ldots,k-1\}.
\]
By Theorem \ref{thm1}(\ref{thm1:eigenbasis}), $\mathbf w$ minimizes $R_\gamma$
over $X\setminus\{0\}$. Differentiating
$R_\gamma(\mathbf w+t\mathbf v)$ at $t=0$ gives
\[
a_\gamma(\mathbf w,\mathbf v)
=\lambda_k(\mathbf w,\mathbf v)_\gamma,
\qquad \mathbf v\in X.
\]
For $j=1,\ldots,k-1$, symmetry and the weak
eigen-equation for $\mathbf e_j$ give
\[
a_\gamma(\mathbf w,\mathbf e_j)
=a_\gamma(\mathbf e_j,\mathbf w)
=\lambda_j(\mathbf e_j,\mathbf w)_\gamma=0
=\lambda_k(\mathbf w,\mathbf e_j)_\gamma.
\]
Every $\mathbf v\in H^1$ is the sum of an element of $X$
and an element of $E_{k-1}$. Hence
\[
a_\gamma(\mathbf w,\mathbf v)
=\lambda_k(\mathbf w,\mathbf v)_\gamma,
\qquad \mathbf v\in H^1.
\]
Thus $(\lambda_k,\mathbf w)$ is a weak eigenpair of
\eqref{evp1}. In particular, each component
$w_i\in H^1(\Omega_i)$ satisfies
\[
-\Delta w_i+c_i(x)w_i=\lambda_k w_i
\qquad\text{weakly in }\Omega_i,\quad i=1,2.
\]
Now $\mathbf{w}\equiv 0$ in $D_{k+1}$. If $D_{k+1}\cap \Omega_1\neq\emptyset$, the unique continuation property  (see, e.g., \cite{garofalo1986monotonicity} and \cite{garofalo1987unique}) implies $w_1\equiv 0$ in $\Omega_1$. Similarly, if $D_{k+1}\cap \Omega_2\neq\emptyset$, then $w_2\equiv 0$ in $\Omega_2$. Assume without loss of generality that $w_1\equiv 0$, so $\mathbf{w}=(0,w_2)$. Choose $\mathbf{v}=(v_1,0)\in H^1\setminus\{0\}$ with $v_1=w_2$ on $\Sigma$ (existence of such $v_1$ follows from \cite[Theorem 18.40]{Leoni2009}). Substituting this test function into the eigen-equation for $(\lambda_k,\mathbf{w})$ gives
\begin{align}\label{eq:w5.8}
\int_{\Sigma}(w_2)^2\,\mathrm{d}s=0. 
\end{align}
So $w_2=0$  on $\Sigma$. Identify the pair $\mathbf{w}$ with the scalar 
function $\widetilde{w}$ defined by $\widetilde{w}|_{\Omega_i}=w_i$ for $i=1,2$.  Then $\widetilde{w}\in H^1(\Omega)$ and is 
a weak solution of 
\begin{align*}
\begin{cases}
-\Delta \widetilde{w}+c(x)\widetilde{w}=\lambda_k \widetilde{w}, & x\in\Omega,\\
\partial_\nu \widetilde{w}=0, & x\in \partial\Omega,
\end{cases}
\end{align*}
where $c(x)=c_i(x)$ for $x\in \Omega_i$. By unique continuation  again, $\widetilde{w} \equiv 0$ in $\Omega$, contradicting $\mathbf{w}=(0,w_2)\neq 0$. Therefore $\mathbf u$ cannot have $k+1$ distinct nodal domains, and hence $\omega(\mathbf{u})\le k$. 
\end{proof}


We next consider \eqref{evp2}. 
\begin{proof}[\bf Proof of $\omega(\hat{\mathbf{u}})\leq k$.]
For $k=1$, Lemma \ref{lem5}(iii) shows that
$\hat{\mathbf u}$, identified with a continuous function on
$\Omega$, has a strict sign throughout $\Omega$.
Hence $Z(\hat{\mathbf u})=\emptyset$, and the connectedness
of $\Omega$ gives $\omega(\hat{\mathbf u})=1$. Let $k\geq2$
and suppose, to the contrary, that $\hat{\mathbf{u}}$ has
at least $k+1$ distinct nodal domains. Choose $D_1,\ldots,D_{k+1}$ among
them.

For $j=1,\ldots,k$, let $\hat{\mathbf{v}}_j$ be the restriction of
$\hat{\mathbf{u}}$ to $D_j$, extended by zero outside $D_j$. The standard
nodal-domain restriction lemma gives $\hat{\mathbf{v}}_j\in H^1(\Omega)$.
In contrast to the situation for \eqref{evp1}, the two traces of
$\hat{\mathbf{v}}_j$ agree on $\Sigma$; hence
$\hat{\mathbf{v}}_j\in H^1_\Sigma$. Set
$\hat{\mathbf r}=\hat{\mathbf u}-\sum_{j=1}^k\hat{\mathbf v}_j\in
H^1_\Sigma$. Moreover, for $j\ne\ell$,
\[
\langle\hat{\mathbf{v}}_j,\hat{\mathbf{v}}_\ell\rangle_\alpha=0,
\qquad
b_\alpha(\hat{\mathbf{v}}_j,\hat{\mathbf{v}}_\ell)=0.
\]
The above identities hold because both the functions and their weak gradients
have disjoint supports. The same observation gives
$\langle\hat{\mathbf v}_j,\hat{\mathbf r}\rangle_\alpha=0$ and
$b_\alpha(\hat{\mathbf v}_j,\hat{\mathbf r})=0$. In particular,
\[
V:=\operatorname{span}\{\hat{\mathbf{v}}_1,\ldots,
\hat{\mathbf{v}}_k\}
\quad\text{satisfies}\quad \dim V=k.
\]
Choose a nonzero
\[
\mathbf{w}=\sum_{j=1}^k c_j\hat{\mathbf{v}}_j\in V,
\qquad
\mathbf{w}\perp^\alpha
\operatorname{span}\{\hat{\mathbf{e}}_1,\ldots,\hat{\mathbf{e}}_{k-1}\}.
\]
Using the eigen-equation for $(\mu_k,\hat{\mathbf{u}})$ and the preceding
orthogonality relations, we obtain
\begin{align*}
b_\alpha(\mathbf{w},\mathbf{w})
&=\sum_{j=1}^k (c_j)^2b_\alpha
(\hat{\mathbf{v}}_j,\hat{\mathbf{v}}_j)\\
&=\sum_{j=1}^k (c_j)^2b_\alpha
(\hat{\mathbf{v}}_j,\hat{\mathbf{u}})
=\mu_k\sum_{j=1}^k (c_j)^2
\langle\hat{\mathbf{v}}_j,\hat{\mathbf{v}}_j\rangle_\alpha\\
&=\mu_k\langle\mathbf{w},\mathbf{w}\rangle_\alpha.
\end{align*}
Thus $(\mu_k,\mathbf w)$ is a weak eigenpair of
\eqref{evp2}. Since $\mathbf w\equiv0$ in $D_{k+1}$,
this yields a contradiction as follows.

The open set $D_{k+1}$ meets at least one of $\Omega_1$ and $\Omega_2$.
Suppose, for instance, that $D_{k+1}\cap\Omega_1\ne\emptyset$. Unique
continuation for
\[
-\Delta w_1+c_1(x)w_1=\mu_k w_1\qquad\text{in }\Omega_1
\]
gives $w_1\equiv0$ in $\Omega_1$. Thus $w_2=0$ on $\Sigma$. Identify the pair $\mathbf{w}$ with the scalar 
function $\widetilde{w}$ defined by $\widetilde{w}|_{\Omega_i}=w_i$ for $i=1,2$.  Then $\widetilde{w}\in H^1(\Omega)$ and is 
a weak solution of 
\[
-\Delta\widetilde w+c(x)\widetilde w=\mu_k\widetilde w
\qquad\text{in }\Omega,
\]
where $c(x)=c_i(x)$ in $\Omega_i$. By unique continuation again, 
$\widetilde w\equiv0$ in $\Omega$, contradicting $\mathbf{w}\ne0$. The case
$D_{k+1}\cap\Omega_2\ne\emptyset$ is identical. Therefore
$\hat{\mathbf u}$ cannot have $k+1$ distinct nodal domains, and
$\omega(\hat{\mathbf{u}})\leq k$.
\end{proof}


We close the section by discussing  the divergence-form extension of the
nodal-domain results.
Consider the symmetric divergence-form interface eigenvalue problems
\eqref{eq6} and \eqref{eq7}, with coefficients satisfying the hypotheses
of Remark \ref{rem-generalization}. The conclusions of Lemmas
\ref{lem:regularityofeigenfunction} and
\ref{lem:regularityofeigenfunction2} remain valid.
Thus Definition \ref{def-nodal}
applies to their real-valued eigenfunctions.
\begin{prop}\label{rem-nodal-generalization}
Assume  $a_{j\ell}^{(i)}$ are Lipschitz continuous over $\overline{\Omega_i}$  
for $1\le j,\ell\le N$ and $i=1,2$.  Then every real-valued eigenfunction
$\mathbf u$ of \eqref{eq6} corresponding to $\lambda_k$ and every
real-valued eigenfunction $\hat{\mathbf u}$ of \eqref{eq7}
corresponding to $\mu_k$ satisfy
\[
\omega(\mathbf u)\le k,
\qquad
\omega(\hat{\mathbf u})\le k.
\]
\end{prop}
\begin{proof}[{\bf Proof.}]
The variational part of the proof of Theorem \ref{thmnodal}
is unchanged when $a_\gamma$ and $b_\alpha$ are replaced
by $\tilde a_\gamma$ and $\tilde b_\alpha$, respectively.
We explain the unique-continuation step.
Write $\eta=\lambda_k$ for \eqref{eq6} and
$\eta=\mu_k$ for \eqref{eq7}.
The auxiliary weak eigenfunction $\mathbf w=(w_1,w_2)$
vanishes on a nonempty open subset of at least one
subdomain. Interior unique continuation for Lipschitz
principal coefficients
\cite{garofalo1986monotonicity,garofalo1987unique}
therefore implies that one component vanishes identically.
It suffices to consider the case $w_1\equiv0$ in $\Omega_1$;
the other case is analogous.

For \eqref{eq6}, testing with $(v_1,0)$, where
$v_1|_\Sigma=w_2|_\Sigma$, gives
$\int_\Sigma w_2^2\,\mathrm ds=0$, exactly as in
\eqref{eq:w5.8}. For \eqref{eq7}, the condition
$\mathbf w\in H^1_\Sigma$ directly gives
$w_2|_\Sigma=w_1|_\Sigma=0$.
Thus, in either case, $\mathbf{w}=(0,w_2)\in H^1(\Omega)$.


Choose constants $0<\theta\le\Theta$ such that
$\theta I\le A_2(x)\le\Theta I$ on
$\overline{\Omega_2}$.
Extend the entries of $A_2$ from $\overline{\Omega_2}$
to Lipschitz functions on $\mathbb R^N$, preserving
symmetry, and denote the resulting matrix field by $B$.
In the space of real symmetric matrices equipped with
the Frobenius norm, let $P_{\mathcal K}$ be the metric
projection onto the closed convex set
\[
\mathcal K=\{M=M^{\mathsf T}:\theta I\le M\le\Theta I\}.
\]
Since $P_{\mathcal K}$ is nonexpansive and fixes every
element of $\mathcal K$, the field
$\widetilde A_2=P_{\mathcal K}\circ B$ is Lipschitz,
symmetric, and uniformly elliptic, and agrees with
$A_2$ on $\overline{\Omega_2}$.

%
It is easy to see that 
\begin{equation*}
	-\operatorname{div}
(\widetilde A_2\nabla\mathbf{w})
+q \mathbf{w}=0
\qquad\text{weakly in }\Omega,
\end{equation*}
where $q(x)=c(x)-\eta$.
Unique continuation, applied to this equation on the
connected domain $\Omega$, gives
$\mathbf{w}\equiv0$, which is a contradiction and hence completes the proof.
\end{proof}

When  the coefficients $a_{j\ell}$ are merely bounded and measurable, unique
continuation may fail.  Thus, the argument used to obtain the usual Courant
bound is no longer available.  The Poincar\'{e}-type characterization  nevertheless yields the
following multiplicity-dependent bound.

Since the spectral conclusions in Theorems \ref{thm1} and \ref{thm2} extend
to \eqref{eq6} and \eqref{eq7} by Remark \ref{rem-generalization}, every
eigenvalue has finite multiplicity.  For $k\geq1$, define
\[
\Phi(k):=\max\{j:\lambda_j=\lambda_k\},
\qquad
\Psi(k):=\max\{j:\mu_j=\mu_k\}.
\]
Equivalently,
\[
\lambda_k=\lambda_{\Phi(k)}<\lambda_{\Phi(k)+1},
\qquad
\mu_k=\mu_{\Psi(k)}<\mu_{\Psi(k)+1}.
\]

\begin{prop}\label{prop1}
Assume that the coefficients in \eqref{eq6} and \eqref{eq7} satisfy
the hypotheses of Remark \ref{rem-generalization}.
Let $\mathbf{u}$ and $\hat{\mathbf{u}}$ be  real-valued eigenfunctions
corresponding to $\lambda_k$ and $\mu_k$, respectively. Then
\[
\omega(\mathbf{u})\leq \Phi(k),
\qquad
\omega(\hat{\mathbf{u}})\leq \Psi(k).
\]
\end{prop}

\begin{proof}[{\bf Proof.}]
We give the proof for $\mathbf{u}$; the argument for $\hat{\mathbf{u}}$  is similar.
 Suppose, to the contrary, that $\mathbf u$
has more than $\Phi(k)$ nodal domains, set $n=\Phi(k)+1$, and choose
distinct nodal domains $D_1,\ldots,D_n$. Let $\mathbf{v}_j$ be the
restriction of $\mathbf{u}$ to $D_j$, extended by zero outside $D_j$, as
in \eqref{eq78}.  The standard nodal-domain
restriction lemma gives $\mathbf{v}_j\in H^1$.  Moreover, the
$\mathbf{v}_j$ are nonzero and have pairwise disjoint supports; hence
\[
W:=\operatorname{span}\{\mathbf{v}_1,\ldots,\mathbf{v}_n\}
\quad\text{satisfies}\quad \dim W=n.
\]
Set $\mathbf r=\mathbf u-\sum_{j=1}^n\mathbf v_j\in H^1$.

Let
\[
\widetilde R_\gamma(\mathbf{v})
:=\frac{\tilde a_\gamma(\mathbf{v},\mathbf{v})}
{(\mathbf{v},\mathbf{v})_\gamma}
\qquad (\mathbf{v}\ne0).
\]
The calculation in the proof of Theorem \ref{thmnodal} is unchanged for
$\tilde a_\gamma$:
for $j\ne\ell$,
\[
(\mathbf{v}_j,\mathbf{v}_\ell)_\gamma=0,
\qquad
\tilde a_\gamma(\mathbf{v}_j,\mathbf{v}_\ell)\geq0.
\]
The same calculation with the complementary remainder gives
$\tilde a_\gamma(\mathbf v_j,\mathbf r)\geq0$. Thus, if
$\mathbf{w}=\sum\limits_{j=1}^n c_j\mathbf{v}_j\in W$, then
\begin{eqnarray*}
 \tilde a_\gamma(\mathbf{w},\mathbf{w}) &\le& \sum_{j=1}^n \sum_{\ell=1}^n \frac{1}{2} \left[(c_j)^2 + (c_\ell)^2\right] \tilde a_{\gamma}(\mathbf{v}_j, \mathbf{v}_\ell)=\sum_{j=1}^n \sum_{\ell=1}^n  (c_j)^2 \tilde a_{\gamma}(\mathbf{v}_j, \mathbf{v}_\ell)\\
&\le &\sum_{j=1}^n (c_j)^2 \tilde a_{\gamma}\left(\mathbf{v}_j, \sum_{\ell=1}^n \mathbf{v}_\ell+\mathbf r\right)=\sum_{j=1}^n (c_j)^2 \tilde a_{\gamma}(\mathbf{v}_j,\mathbf{u})=\sum_{j=1}^{n}(c_j)^2\lambda_k(\mathbf{v}_j,\mathbf{u})_\gamma=\lambda_k (\mathbf{w},\mathbf{w})_\gamma\,.
\end{eqnarray*}
Consequently, $\widetilde R_\gamma(\mathbf{w})\leq\lambda_k$ for every
$\mathbf{w}\in W\setminus\{0\}$.  The Poincar\'e-type characterization therefore gives
\[
\lambda_n\leq
\max_{\mathbf{w}\in W\setminus\{0\}}\widetilde R_\gamma(\mathbf{w})
\leq\lambda_k.
\]
This contradicts $n=\Phi(k)+1$ and the definition of $\Phi(k)$.

For \eqref{eq7}, if $\hat{\mathbf u}$ had more than $\Psi(k)$ nodal
domains, we would choose $n=\Psi(k)+1$ of them. Repeating the above argument
with $\tilde b_\alpha$ and its Rayleigh quotient would give
$\mu_n\leq\mu_k$, contradicting the definition of $\Psi(k)$.
\end{proof}

After a constant shift if necessary, our weighted forms fall within the positivity-preserving framework of Keller and Schwarz \cite{KellerSchwarz2020}. Their multiplicity-dependent Courant bounds concern nodal domains defined through restricted forms; relating these to Definition \ref {def-nodal} requires further verification. Proposition \ref{prop1} instead gives a direct proof for geometric nodal domains. For the stronger bound in Theorem \ref{thmnodal}, a further distinction arises: the form associated with \eqref{evp2} is local, whereas the jump term in \eqref{evp1} gives a nonlocal coupling between the two sides of the interface. Thus, the strong Courant theorem for local forms in \cite{KellerSchwarz2020} does not apply directly to \eqref{evp1}; our proof combines the sign of the interface cross terms with unique continuation.

\section{Illustrative Examples}\label{sec6}
We demonstrate, via one‑dimensional examples, our results on the dependence of eigenvalues on interface parameters (Theorems \ref{thm3}-\ref{thm6}) and on the number of nodal domains of eigenfunctions (Theorem \ref{thmnodal}).  
Set $\Omega_1=(-1,0)$, $\Omega_2=(0,2)$, $\Sigma=\{0\}$, $\Gamma_1=\{-1\}$, $\Gamma_2=\{2\}$, and assume $c_i(x)\equiv 0$ on $\Omega_i$.
Then Problem \eqref{evp1} reduces to
\begin{equation}\label{eq:6.1}
\begin{cases}
-u_i''=\lambda u_i,\quad x\in\Omega_i,\\[2pt]
u_i'(0)=\gamma_i\big(u_2(0)-u_1(0)\big), \\[2pt]
u_1'(-1)=u_2'(2)=0,
\end{cases}i=1,2.
\end{equation}
The spectrum consists of two families of eigenvalues.
First,
\begin{equation}\label{eq:6.2}
\lambda_k^{(1)}=(k\pi)^2,\qquad k=0,1,2,\dots,
\end{equation}
with corresponding eigenfunctions
\[
u_k^{(1)}(x)=\cos(k\pi x),\quad x\in[-1,2].
\]
Second, $\lambda_k^{(2)}=(\omega_k)^2$ for $k=0,1,2,\dots$, where $\omega_k$ is the unique root of
\begin{equation}\label{eq:6.3}
\gamma_1\cot(\omega)+\gamma_2\cot(2\omega)=\omega,\qquad \omega\in\bigl(\tfrac{k\pi}{2},\tfrac{(k+1)\pi}{2}\bigr).
\end{equation}
The associated eigenfunctions read
\[
u_k^{(2)}(x)=
\begin{cases}
\big[\omega_k\sin(2\omega_k)-\gamma_2\cos(2\omega_k)\big]\cos\big(\omega_k(x+1)\big),&x\in[-1,0),\\[2pt]
-\gamma_2\cos(\omega_k)\cos\big(\omega_k(x-2)\big),&x\in(0,2].
\end{cases}
\]
Collecting all eigenvalues in increasing order yields
\begin{equation}\label{eq:6.4}
0,\;(\omega_0)^2,\;(\omega_1)^2,\;\pi^2,\;(\omega_2)^2,\;(\omega_3)^2,\;(2\pi)^2,\dots.
\end{equation}

\bigskip

Problem \eqref{evp2} takes the form
\begin{equation}\label{eq:6.5}
\begin{cases}
-u_i''=\mu u_i,\quad x\in\Omega_i,\\[2pt]
u_1'(0)=\alpha\,u_2'(0),\quad u_1(0)=u_2(0),\\[2pt]
u_1'(-1)=u_2'(2)=0,
\end{cases} i=1,2.
\end{equation}
Again there are two families of eigenvalues. The first family is
\begin{equation}\label{eq:6.6}
\mu_k^{(1)}=(k\pi)^2,\qquad k=0,1,2,\dots,
\end{equation}
with eigenfunctions
\[
v_k^{(1)}(x)=\cos(k\pi x),\quad x\in[-1,2].
\]
The second family is $\mu_k^{(2)}=(\eta_k)^2$, $k=0,1,2,\dots$, where $\eta_k$ is the unique solution of
\begin{equation}\label{eq:6.7}
\cos^2(\eta)=\frac1{2(1+\alpha)},\qquad \eta\in\bigl(\tfrac{k\pi}{2},\tfrac{(k+1)\pi}{2}\bigr).
\end{equation}
Its eigenfunctions are
\[
v_k^{(2)}(x)=
\begin{cases}
\alpha\cos\big(\eta_k(x+1)\big),&x\in[-1,0),\\[2pt]
-(1+\alpha)\cos(\eta_k)\cos\big(\eta_k(x-2)\big),&x\in(0,2].
\end{cases}
\]
Arranging all eigenvalues in increasing order gives
\begin{equation}\label{eq:6.8}
0,\;(\eta_0)^2,\;(\eta_1)^2,\;\pi^2,\;(\eta_2)^2,\;(\eta_3)^2,\;(2\pi)^2,\dots.
\end{equation}

\bigskip

\paragraph{Observations}
\begin{enumerate}[label=(\roman*)]
\item For problem \eqref{eq:6.1}, every eigenvalue in \eqref{eq:6.2} is independent of $\gamma=(\gamma_1,\gamma_2)$.
Each eigenvalue from \eqref{eq:6.3} depends continuously on $\gamma$ and is increasing in $\gamma_2$ under the scaling $\gamma_1=\alpha\gamma_2$ for fixed $\alpha$.
These facts illustrate Theorem \ref{thm3}(i) and (iii).

\item For problem \eqref{eq:6.5}, eigenvalues in \eqref{eq:6.6} do not depend on $\alpha$, whereas each eigenvalue from \eqref{eq:6.7} depends continuously on $\alpha$.
This corresponds to Theorem \ref{thm3}(ii).

\item Multiply equation \eqref{eq:6.3} by $\frac1{\gamma_2}\sin(2\omega)$:
\[
\frac{\gamma_1}{\gamma_2}\,2\cos^2\omega
+2\cos^2\omega-1
=\frac{\omega}{\gamma_2}\sin(2\omega),
\qquad
\omega\in\Bigl(\tfrac{k\pi}{2},\tfrac{(k+1)\pi}{2}\Bigr).
\]
Let $\gamma_1,\gamma_2\to\infty$ with $\gamma_1/\gamma_2\to\alpha$. In the limit, $\omega$ satisfies
\[
2\alpha\cos^2\omega+2\cos^2\omega-1=0
\quad\Longleftrightarrow\quad
\cos^2\omega=\frac1{2(1+\alpha)},\qquad
\omega\in\Bigl[\tfrac{k\pi}{2},\tfrac{(k+1)\pi}{2}\Bigr].
\]
Comparing with \eqref{eq:6.7}, we obtain $\omega_k\to\eta_k$ as $\gamma_1,\gamma_2\to\infty$ and $\gamma_1/\gamma_2\to\alpha$. This illustrates the limit asserted in Theorem \ref{thm4}.

\item The decoupled Neumann problems on $\Omega_i$ read
\[
\begin{cases}
-u_1''=\sigma u_1,\quad x\in\Omega_1,\\
u_1'(-1)=u_1'(0)=0,
\end{cases}
\qquad
\begin{cases}
-u_2''=\sigma u_2,\quad x\in\Omega_2,\\
u_2'(0)=u_2'(2)=0.
\end{cases}
\]
Their eigenvalues are
\begin{equation}\label{eq:6.9}
\sigma_k^N(\Omega_1)=(k\pi)^2,\qquad k=0,1,2,\dots,
\end{equation}
with eigenfunctions $u_k(x)=\cos(k\pi x)$; and
\begin{equation}\label{eq:6.10}
\sigma_k^N(\Omega_2)=\bigl(\tfrac{k\pi}{2}\bigr)^2,\qquad k=0,1,2,\dots,
\end{equation}
with eigenfunctions $u_k(x)=\cos\bigl(\tfrac{k\pi}{2}x\bigr)$.

Now take $\gamma_1,\gamma_2\to 0$ in \eqref{eq:6.3}. One finds $\omega_k\to \frac{k\pi}{2}$.
Together with \eqref{eq:6.4}, \eqref{eq:6.9}, and \eqref{eq:6.10}, this illustrates the limit asserted in Theorem \ref{thm5}.

\item We illustrate the two limits in Theorem \ref{thm6} separately for the present example.

\smallskip
\noindent\emph{The limit $\alpha\to0$.} In this regime, the limiting
spectrum is obtained by merging 
the Neumann spectrum on $\Omega_1$ with  the
Dirichlet--Neumann spectrum on $\Omega_2$. The former is given in
\eqref{eq:6.9}. For the latter, the problem on $\Omega_2$ is
\[
\begin{cases}
-u_2''=\sigma u_2, & x\in\Omega_2,\\
u_2(0)=0,\quad u_2'(2)=0,
\end{cases}
\]
with eigenvalues and eigenfunctions
\begin{equation}\label{eq:6.12}
\sigma_m^{DN}(\Omega_2)=\Bigl(\tfrac{(2m+1)\pi}{4}\Bigr)^2,
\qquad
u_m(x)=\sin\bigl(\tfrac{(2m+1)\pi}{4}x\bigr),
\qquad m=0,1,2,\ldots.
\end{equation}
Letting $\alpha\to0$ in \eqref{eq:6.7} gives
\[
\cos^2\eta=\tfrac12,
\qquad
\eta\in\Bigl[\tfrac{m\pi}{2},\tfrac{(m+1)\pi}{2}\Bigr],
\]
and hence
\[
\eta_m\to\tfrac{(2m+1)\pi}{4}.
\]
Thus the fixed family \eqref{eq:6.6} yields the Neumann spectrum on
$\Omega_1$, while the second family yields the Dirichlet--Neumann spectrum
on $\Omega_2$. In view of the ordering \eqref{eq:6.8}, this illustrates the first limit asserted in Theorem \ref{thm6}.

\smallskip
\noindent\emph{The limit $\alpha\to\infty$.} Here the limiting spectrum is
the merge of the Dirichlet--Neumann spectrum on $\Omega_1$ and the Neumann
spectrum on $\Omega_2$, the latter being given in \eqref{eq:6.10}. The
Dirichlet--Neumann problem on $\Omega_1$ has eigenvalues and eigenfunctions
\begin{equation}\label{eq:6.11}
\sigma_m^{DN}(\Omega_1)=\Bigl(\tfrac{(2m+1)\pi}{2}\Bigr)^2,
\qquad
u_m(x)=\sin\bigl(\tfrac{(2m+1)\pi}{2}x\bigr),
\qquad m=0,1,2,\ldots.
\end{equation}
Since \eqref{eq:6.7} gives $\cos^2\eta_m\to0$, we have
\[
\eta_{2m},\ \eta_{2m+1}\ \longrightarrow\
\tfrac{(2m+1)\pi}{2},
\qquad m=0,1,2,\ldots.
\]
Therefore, the second family in \eqref{eq:6.8} supplies every odd
half-integer frequency twice: once for the Dirichlet--Neumann spectrum on
$\Omega_1$ and once for the odd modes of the Neumann spectrum on $\Omega_2$.
The fixed family \eqref{eq:6.6} supplies the remaining even modes of the
Neumann spectrum on $\Omega_2$. Consequently, 
the ordered eigenvalue sequence in \eqref{eq:6.8} converges to the merged spectrum asserted in
the second part of Theorem \ref{thm6}.
\item From \eqref{eq:6.4} and the associated eigenfunctions, we observe that the $k$-th eigenfunction of problem \eqref{eq:6.1} possesses exactly $k$ nodal domains, with the convention
    that a sign change across $x=0$ is counted when the two one-sided values have opposite signs.
For problem \eqref{eq:6.5}, it follows from \eqref{eq:6.8} and its corresponding eigenfunctions that its $k$-th eigenfunction has precisely $k$ nodal domains. Fig. \ref{fig:6.1} and Fig. \ref{fig:6.5}  visualize the nodal  domains of the first six eigenfunctions  of \eqref{eq:6.1} and \eqref{eq:6.5}, respectively. These observations are consistent with Theorem \ref{thmnodal}.

\begin{figure}[htbp]
  \centering
  \begin{subfigure}[b]{0.32\textwidth}
    \centering
    \includegraphics[width=\linewidth]{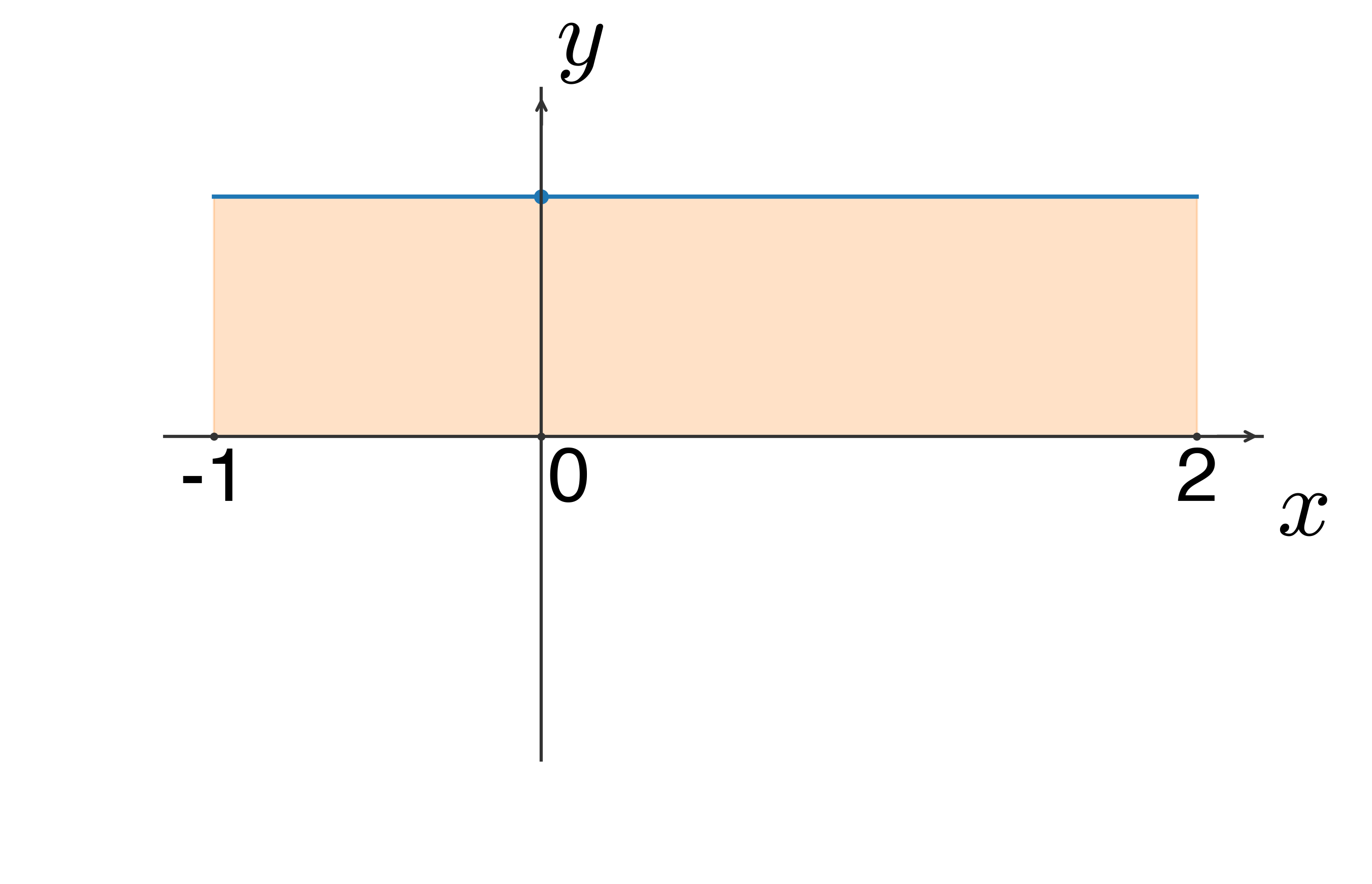} 
\caption{$\lambda_1=0,\, y\equiv 1$}
    \label{fig:sub_a}
  \end{subfigure}\hfill
  \begin{subfigure}[b]{0.32\textwidth}
    \centering
    \includegraphics[width=\linewidth]{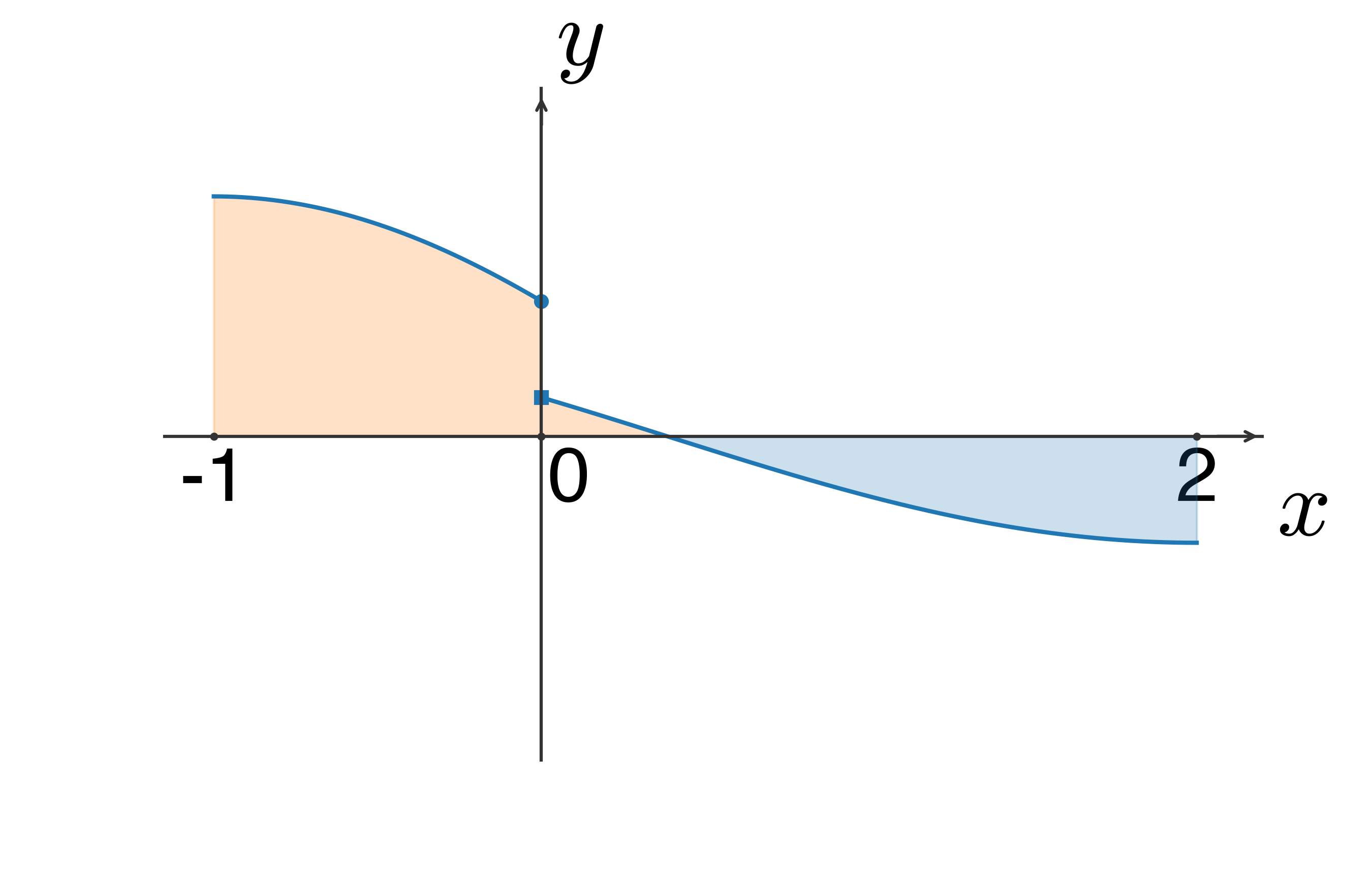}
\caption{$\lambda_2 = (\omega_0)^2,\, y=u_0^{(2)}(x)$ }
    \label{fig:sub_b}
  \end{subfigure}\hfill
  \begin{subfigure}[b]{0.32\textwidth}
    \centering
    \includegraphics[width=\linewidth]{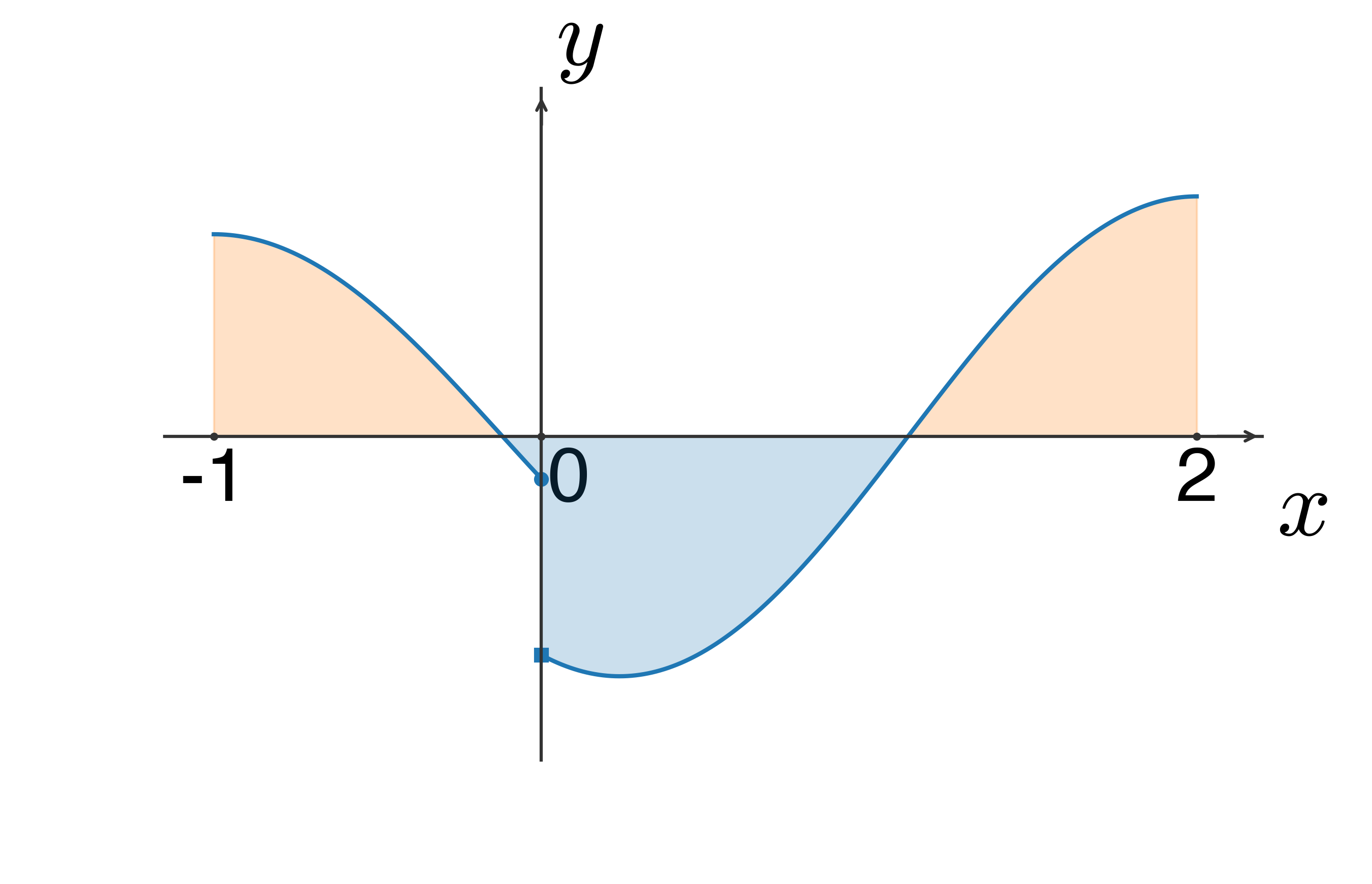}
\caption{$\lambda_3  = (\omega_1)^2,\, y=u_1^{(2)}(x)$ }
    \label{fig:sub_c}
  \end{subfigure}

  \vspace{1.5ex} 

  \begin{subfigure}[b]{0.32\textwidth}
    \centering
    \includegraphics[width=\linewidth]{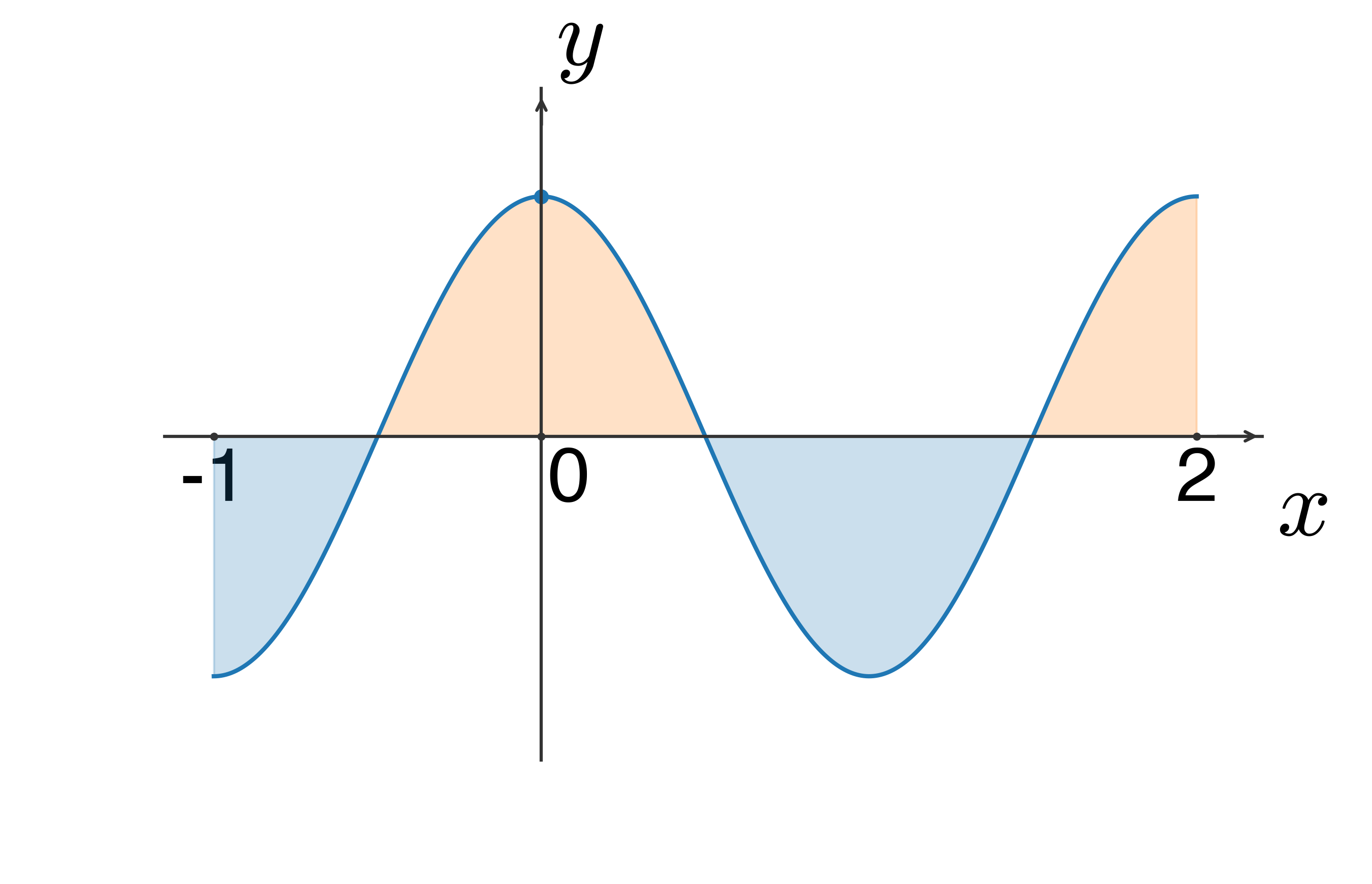}
\caption{$\lambda_4  = \pi^2,\, y=\cos (\pi x)$}
    \label{fig:sub_d}
  \end{subfigure}\hfill
  \begin{subfigure}[b]{0.32\textwidth}
    \centering
    \includegraphics[width=\linewidth]{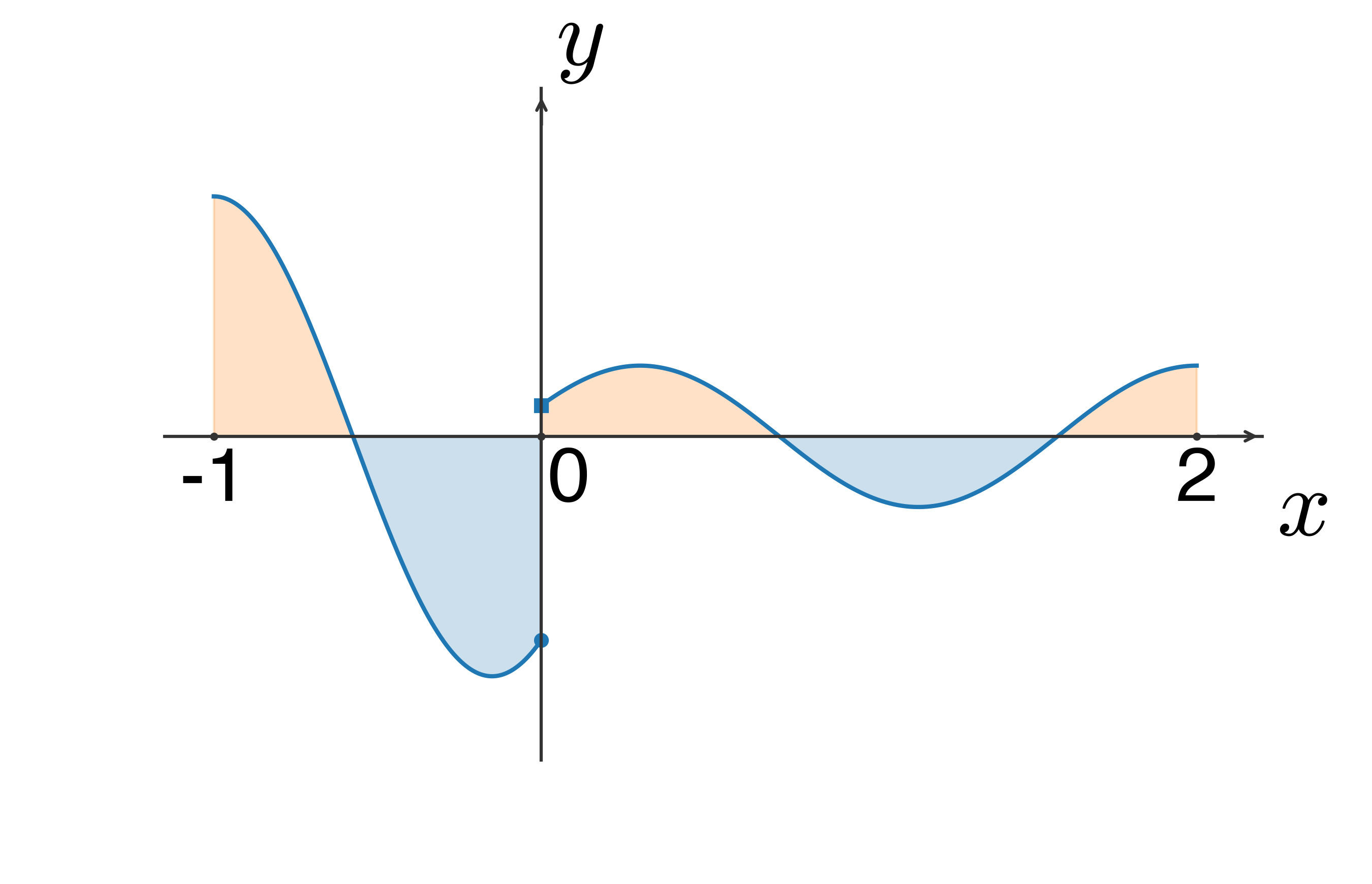}
\caption{$\lambda_5  = (\omega_2)^2,\, y=u_2^{(2)}(x)$}
    \label{fig:sub_e}
  \end{subfigure}\hfill
  \begin{subfigure}[b]{0.32\textwidth}
    \centering
    \includegraphics[width=\linewidth]{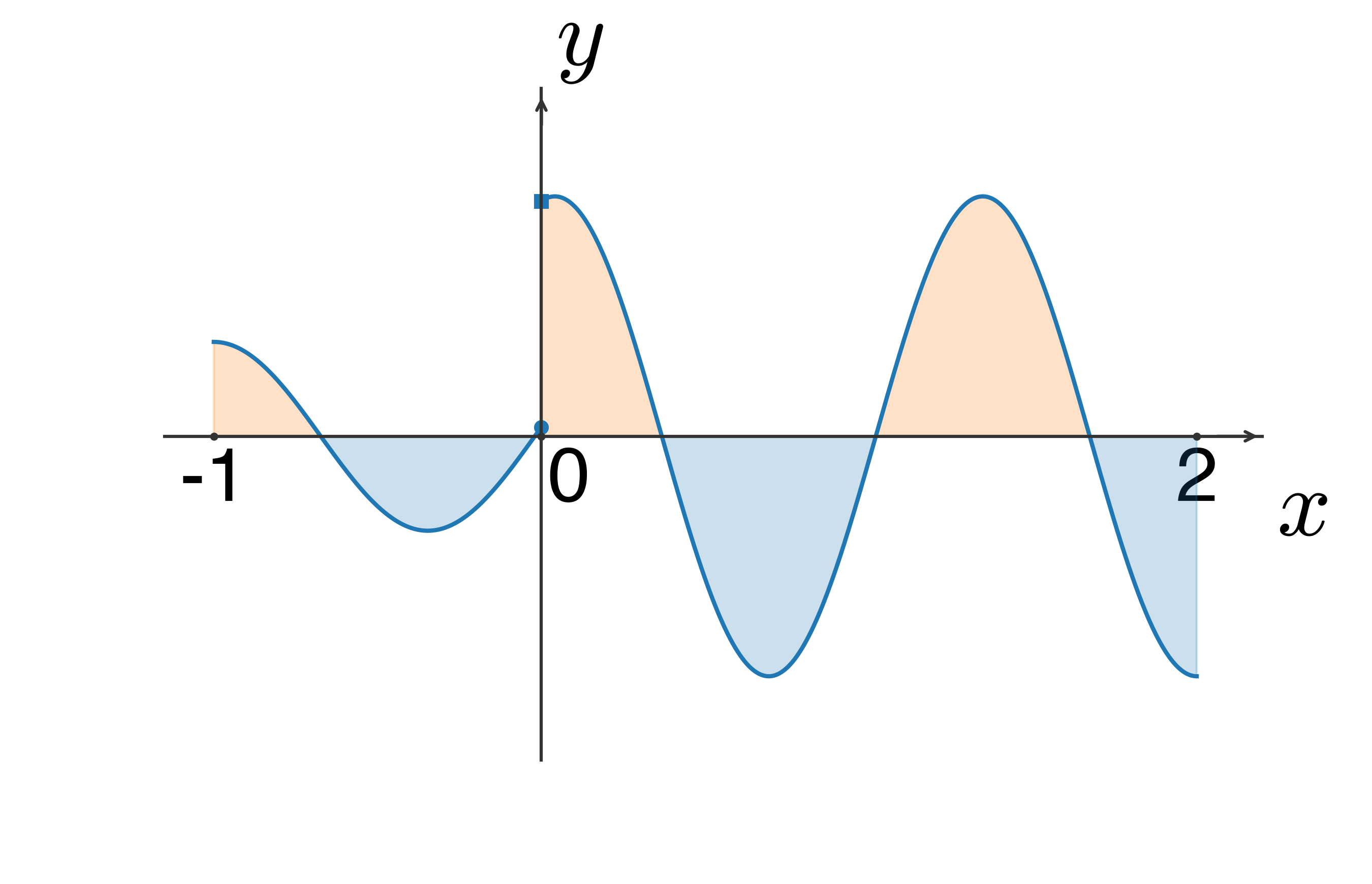}
\caption{$\lambda_6  = (\omega_3)^2,\, y= u_3^{(2)}(x)$ }
    \label{fig:sub_f}
  \end{subfigure}

  \caption{Plots of  the first six eigenfunctions of \eqref{eq:6.1} with $\gamma_1=2,\gamma_2=1$}
  \label{fig:6.1}
\end{figure}

\begin{figure}[h]
  \centering
  \begin{subfigure}[b]{0.32\textwidth}
    \centering
    \includegraphics[width=\linewidth]{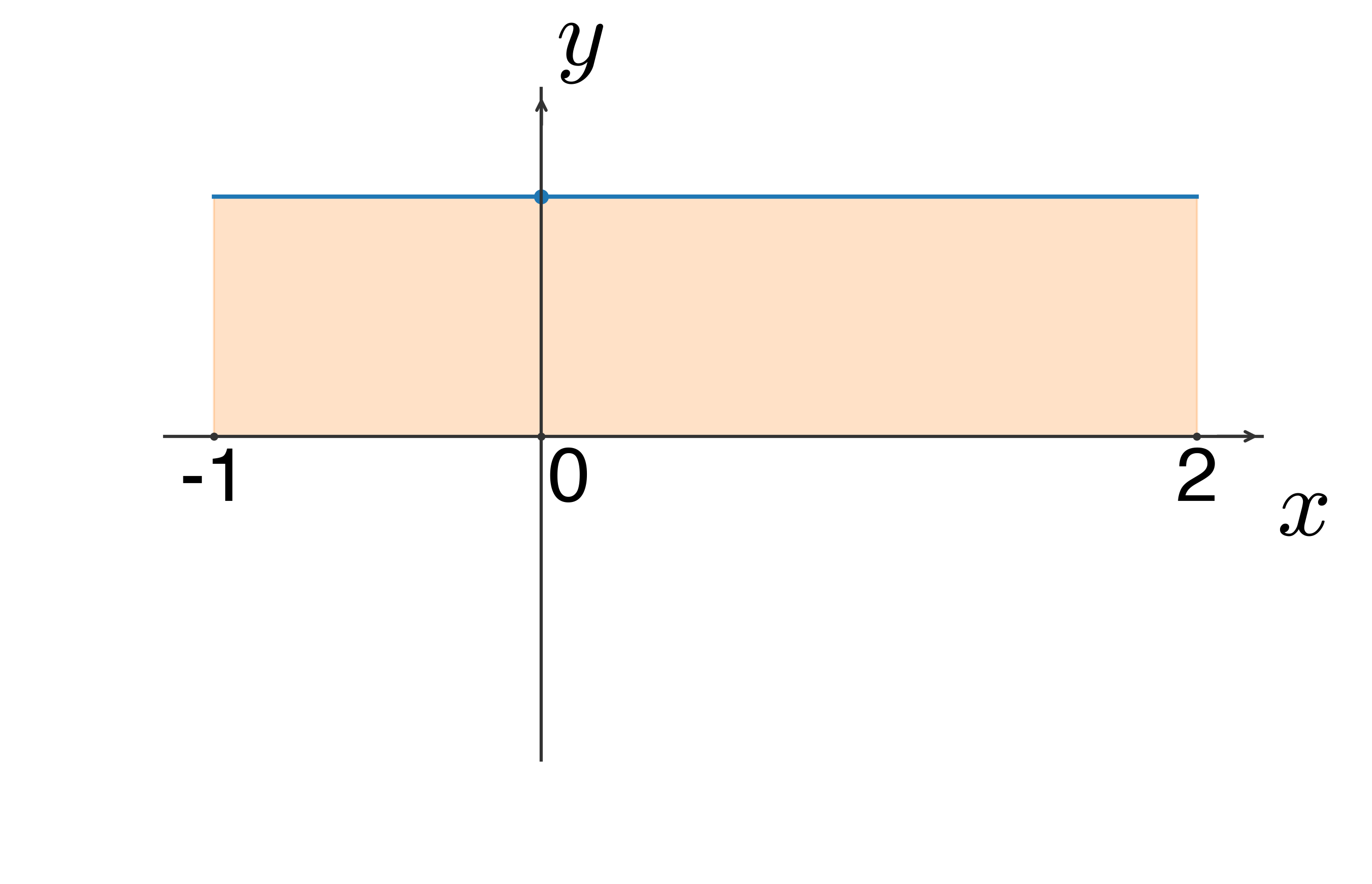} 
\caption{$\mu_1=0,\, y\equiv 1$ }
    \label{fig:6.5_sub_a}
  \end{subfigure}\hfill
  \begin{subfigure}[b]{0.32\textwidth}
    \centering
    \includegraphics[width=\linewidth]{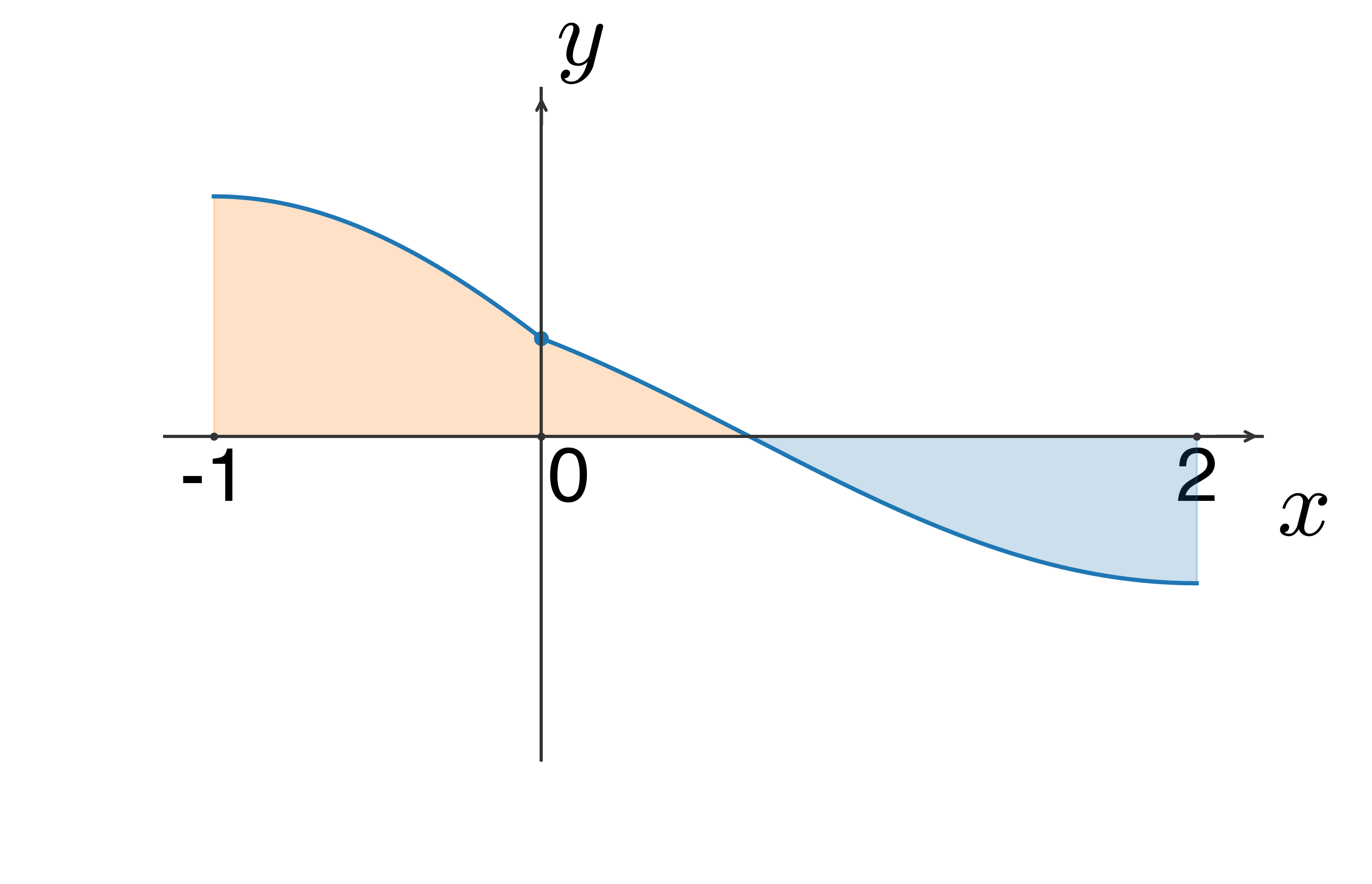}
\caption{$\mu_2 = (\eta_0)^2,\, y=v_0^{(2)} (x)$}
    \label{fig:6.5_sub_b}
  \end{subfigure}\hfill
  \begin{subfigure}[b]{0.32\textwidth}
    \centering
    \includegraphics[width=\linewidth]{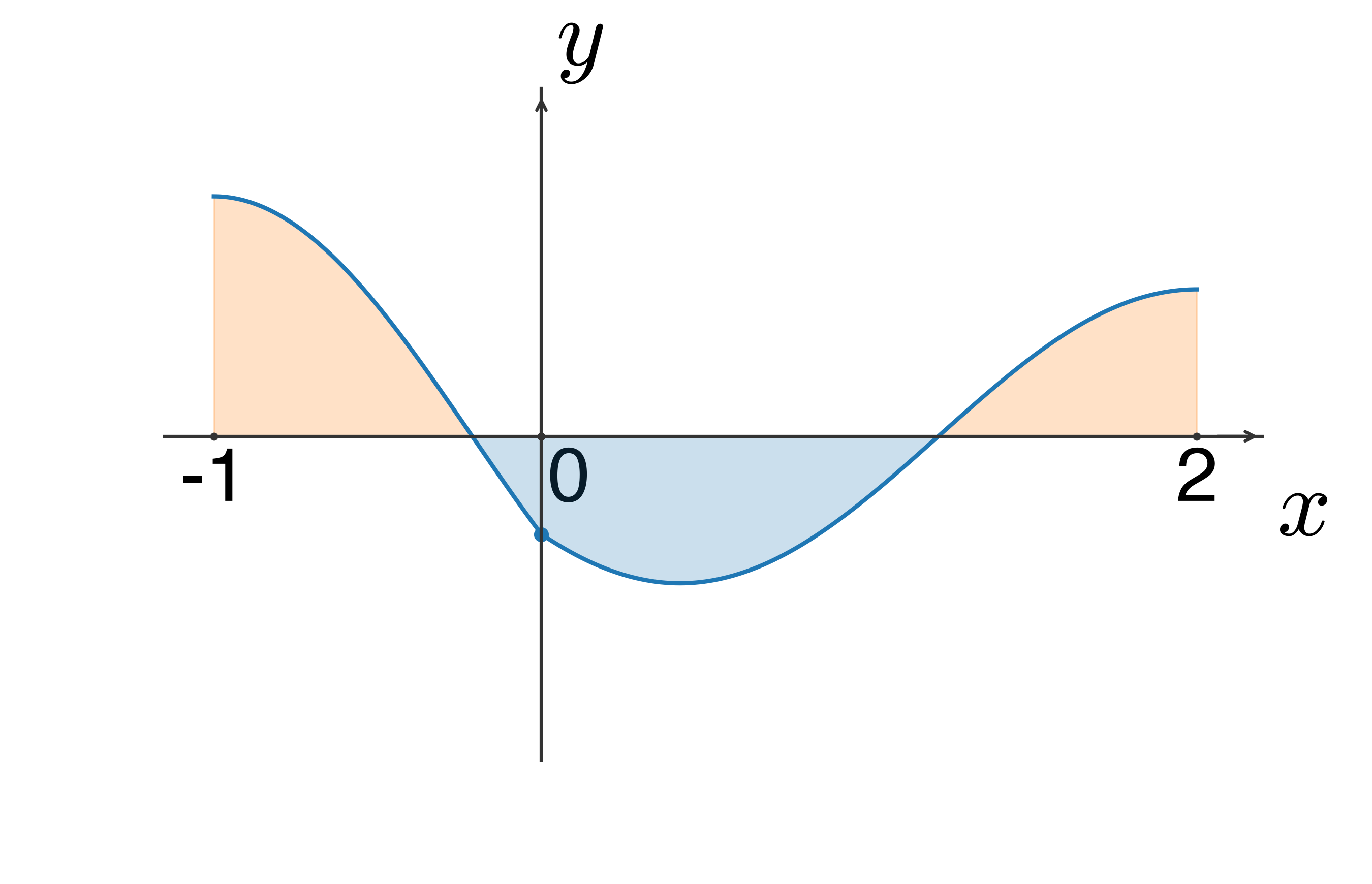}
\caption{$\mu_3  = (\eta_1)^2,\, y=v_1^{(2)}(x)$ }
    \label{fig:6.5_sub_c}
  \end{subfigure}

  \vspace{1.5ex} 

  \begin{subfigure}[b]{0.32\textwidth}
    \centering
    \includegraphics[width=\linewidth]{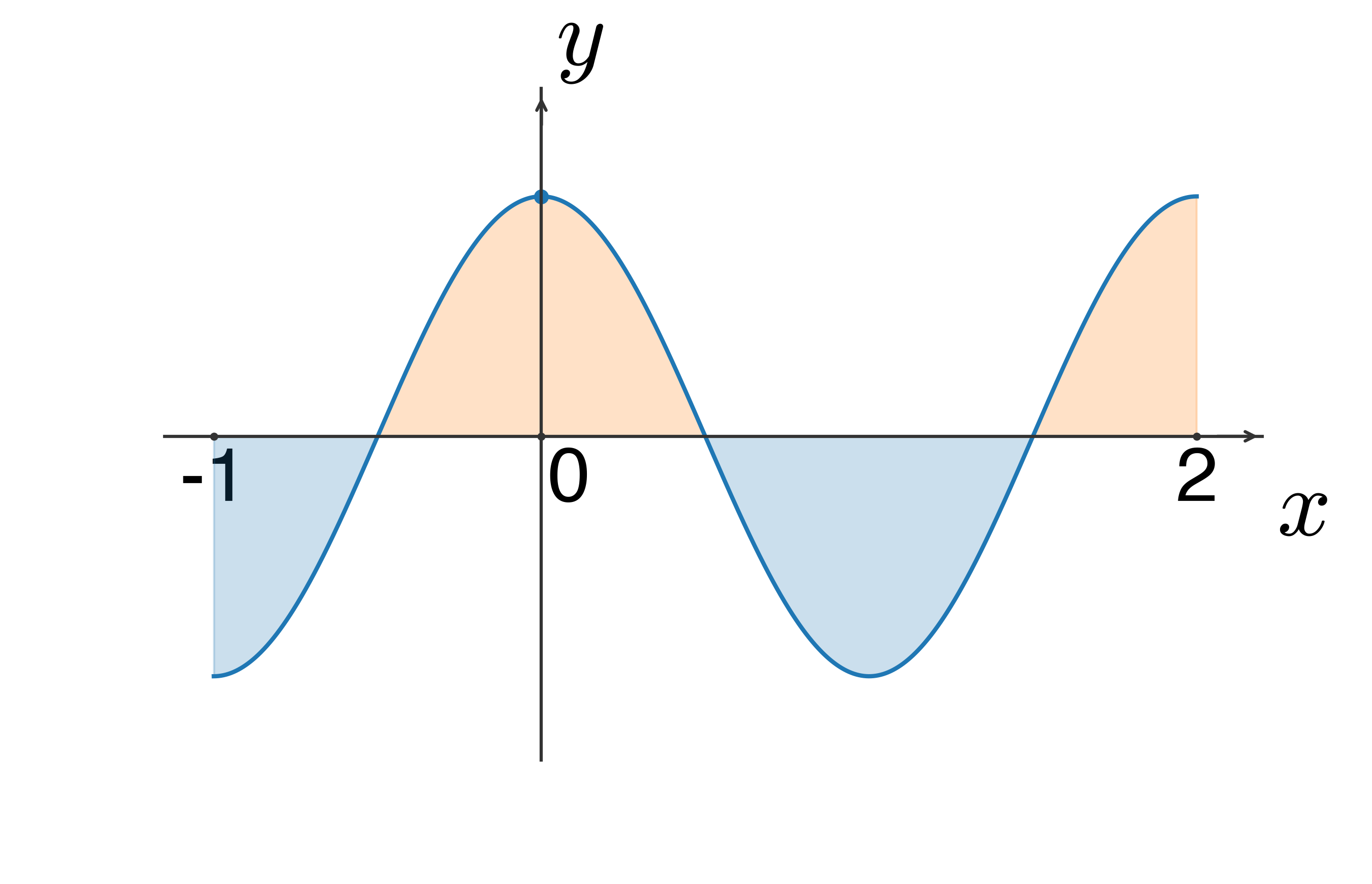}
\caption{$\mu_4  = \pi^2,\, y=\cos (\pi x)$ }
    \label{fig:6.5_sub_d}
  \end{subfigure}\hfill
  \begin{subfigure}[b]{0.32\textwidth}
    \centering
    \includegraphics[width=\linewidth]{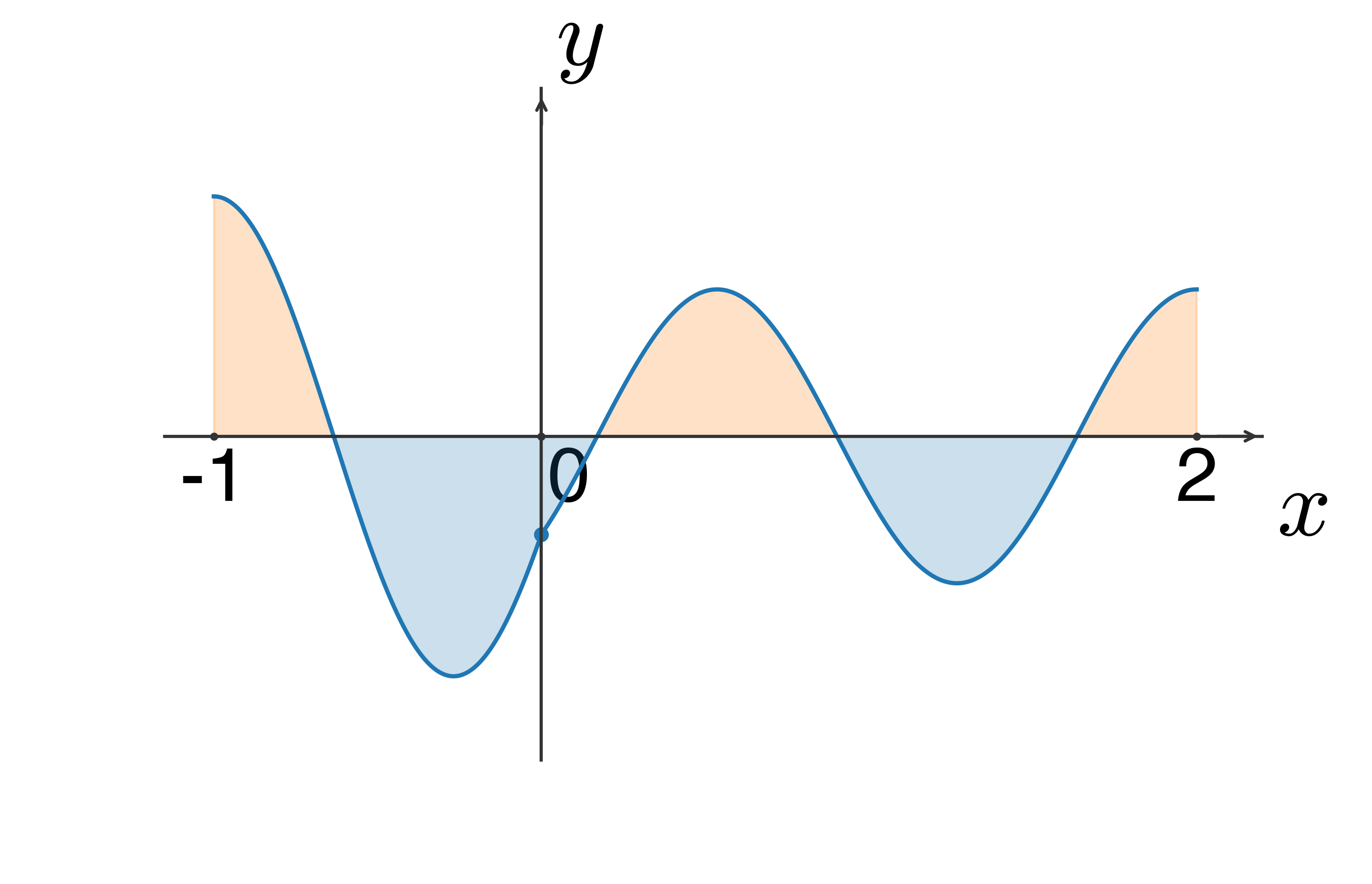}
\caption{$\mu_5  = (\eta_2)^2,\, y=v_2^{(2)}(x)$ }
    \label{fig:6.5_sub_e}
  \end{subfigure}\hfill
  \begin{subfigure}[b]{0.32\textwidth}
    \centering
    \includegraphics[width=\linewidth]{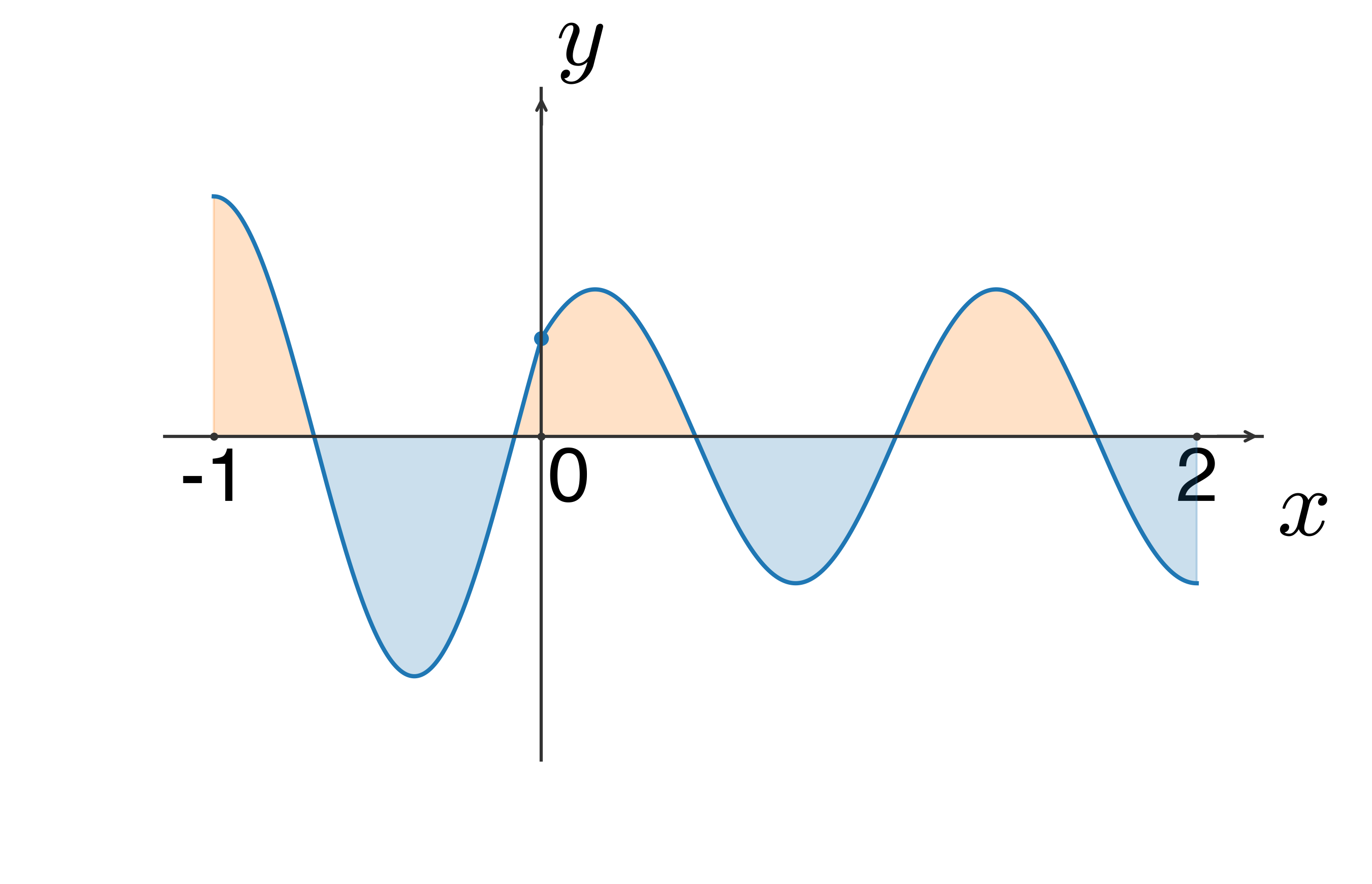}
\caption{$\mu_6  = (\eta_3)^2,\, y= v_3^{(2)}(x)$ }
    \label{fig:6.5_sub_f}
  \end{subfigure}

  \caption{Plots of  the first six eigenfunctions of \eqref{eq:6.5} with $\alpha=2$}
  \label{fig:6.5}
\end{figure}

%
\end{enumerate}

Observation (vi) motivates a Sturm‑Liouville oscillation principle for the following one‑dimensional interface problems on $\Omega_1=(-a,0)$ and $\Omega_2=(0,b)$, where $a,\,b>0$:
\begin{equation}\label{eq:6.15}
\begin{cases}
-(a_i(x)u_i')'+c_i(x)u_i=\lambda u_i,\quad x\in\Omega_i,\\[2pt]
a_i(0)u_i'(0)=\gamma_i\big(u_2(0)-u_1(0)\big), \\[2pt]
u_1'(-a)=u_2'(b)=0,
\end{cases}i=1,2,
\end{equation}
and
\begin{equation}\label{eq:6.16}
\begin{cases}
-(a_i(x)u_i')'+c_i(x)u_i=\mu u_i,\quad x\in\Omega_i,\\[2pt]
a_1(0)u_1'(0)=\alpha\,a_2(0)u_2'(0),\quad u_1(0)=u_2(0),\\[2pt]
u_1'(-a)=u_2'(b)=0,
\end{cases} i=1,2,
\end{equation}
where $a_i\in C^1(\overline\Omega_i)$, $a_i(x)\ge a_0$ on $\overline\Omega_i$ for some constant $a_0>0$, $c_i\in C(\overline\Omega_i)$, and  $\gamma_1,\gamma_2,\alpha>0$.

For \eqref{eq:6.15} with $a_i\equiv 1$, $c_i\equiv 0$, and $\gamma_1=\gamma_2$, simplicity and the exact nodal count are included in the results of Moutal and Grebenkov \cite[Supplementary Material, Sections SM III.1 and SM III.4]{MoutalGrebenkov2019}, which allow multiple semipermeable barriers and count sign changes across jumps. The following proposition treats the two-subdomain problem with variable coefficients and possibly unequal transmission parameters, as well as the continuous-trace problem \eqref{eq:6.16}.

\begin{prop} For problems \eqref{eq:6.15} and \eqref{eq:6.16}, each eigenspace is one-dimensional 
and the $k$-th eigenfunction has exactly $k$
nodal domains, hence exactly $k-1$ sign changes (with the convention
        that a sign change across $x=0$ is counted when the two one-sided values have opposite signs
).
\end{prop}
\begin{proof}[{\bf Proof.}] Each eigenspace is one-dimensional because the solution on $\Omega_1$ is uniquely determined by its value at $x=-a$ together with the Neumann condition there; the
transmission conditions then uniquely determine the solution on $\Omega_2$.


It remains to prove the exact number of nodal domains. Let $\lambda_u<\lambda_v
$ be any two eigenvalues of \eqref{eq:6.15} or of \eqref{eq:6.16}, and let $\mathbf{u}=(u_1,u_2)$ and $\mathbf{v}=(v_1,v_2)$ be the corresponding eigenfunctions.

On each subinterval $\Omega_i$, define 
\begin{equation}\label{eq:6.18}
W_i(x) =a_i(x)\bigl(u_i(x)v_i'(x)-v_i(x)u_i'(x)\bigr).
\end{equation}
Since
\begin{equation}
-(a_iu_i')'+c_iu_i=\lambda_uu_i,
\qquad
-(a_iv_i')'+c_iv_i=\lambda_vv_i,
\end{equation}
we have
\begin{equation}\label{eq:6.17}
W_i'(x)
   =(\lambda_u-\lambda_v)u_i(x)v_i(x). 
\end{equation}

Let $D=(x_1,x_2)$ be an arbitrary nodal domain of $\mathbf{u}$. Without loss of generality, assume that $D\cap\Omega_1\neq\emptyset$ and  $\mathbf{u}>0$ on $D$. We examine the sign of $\mathbf{v}$ on $D$ in three cases.

(i) $x_1<x_2<0$.  This is the standard Sturm argument. Since $x_2$ is an interior endpoint of $D$, we have $u_1(x_2)=0$ with $u_1'(x_2)\le 0$. If $x_1$ is also an interior endpoint of $D$ (i.e., $x_1>-a$), then $u_1(x_1)=0$ with $u_1'(x_1)\ge 0$. If $x_1$ is an exterior endpoint (i.e., $x_1=-a$), the Neumann condition gives $u_1'(x_1)=0$.

Assume that $\mathbf{v}$ does not change sign in $D$.  After changing the sign of $\mathbf{v}$, we may suppose that $\mathbf{v}>0$ on $D$. Then $W_1(x_1)\le 0\le W_1(x_2)$, while \eqref{eq:6.17} gives $W_1'<0$ on $D$, a contradiction. Hence $\mathbf{v}$ changes sign at least once in $D$.

(ii) $x_1<0<x_2$. Again assume that $\mathbf{v}>0$ on $D$. For
\eqref{eq:6.15}, define
\[
\widehat W_i=\frac{W_i}{\gamma_i}.
\]
The transmission conditions in \eqref{eq:6.15}$_2$ give
\[
\widehat W_1(0)=\widehat W_2(0).                  
\]
Thus the normalized Wronskian can be continued across the
interface. Since $\widehat W_i'<0$ on both sides of $x=0$, 
the normalized Wronskian is strictly decreasing on $D$. This contradicts the endpoint inequalities 
$\widehat W_1(x_1)\le 0\le \widehat W_2(x_2)$.


For problem \eqref{eq:6.16}, the argument is identical. The only
difference is to define
\[
\widehat W_1=\frac{W_1}{\alpha},\qquad
\widehat W_2=W_2.
\]
Then
\[
\widehat W_1(0)=\widehat W_2(0)
\]
follows from the transmission conditions in \eqref{eq:6.16}$_2$. Thus, for both \eqref{eq:6.15} and \eqref{eq:6.16}, $\mathbf{v}$ changes sign  at least once in $D$.

(iii) $x_2=0$, i.e., the interface point $x=0$ is the common boundary of the nodal domain $D$ and another nodal domain $\tilde{D}=(0, x_3)$ for some $x_3\in (0,b]$. Again we assume that $\mathbf{v}>0$ on $D$.

For problem \eqref{eq:6.16}, at $x_2=0$ we have $u_1(0)=0$ and $u_1'(0) < 0$, while $v_1(0)\ge 0$. Thus $x_2=0$ behaves as an ordinary interior right-hand endpoint of $D$, so $W(0)\ge 0$. This yields the same contradiction as in case (i). Hence $\mathbf{v}$ changes sign  at least once in $D$.

For problem \eqref{eq:6.15}, from $W_1'<0$ on $D=(x_1,0)$ and $W_1(x_1)\le 0$, we must have $W_1(0)<0$. Since $x=0$ is the common boundary of the nodal domains $D$ and $\tilde D$, we see that $u_2(0)\le 0$, and thus by \eqref{eq:6.15}$_{2}$ we further deduce that $u_1'(0)<0$. Together with $v_1(0)\ge 0$, $u_1(0)\ge 0$, and $W_1(0)<0$, this implies that $v_1'(0)<0$ and $u_1(0)>0$. By \eqref{eq:6.15}$_{2}$, we then have 
$W_2(0)<0$, $u_2'(0)<0$, and $v_1'(0)<0$, which, together with $u_2(0)\le 0$, shows that $v_2(0)<0$. Thus $x=0$ is also the common boundary of two nodal domains of $\mathbf{v}$.

Now suppose that $\mathbf{v}$ does not change sign in $\tilde D$, i.e., $v_2<0$ in $\tilde D$. Then \eqref{eq:6.17} gives $W'_2<0$ in $\tilde D$. However, as in case (i), we have $W_2(x_3)\ge 0$, which is impossible since $W_2(0)<0$. Thus $\mathbf{v}$ changes sign at least once in $\tilde D$.

In summary, for problem \eqref{eq:6.16}, $\mathbf{v}$  changes sign at least once in every nodal domain of $\mathbf{u}$. For problem \eqref{eq:6.15}, each nodal domain $D$ has the following property: either $\mathbf{v}$ changes sign at least once in $D$, or $D$ has $x=0$ as an endpoint, $\mathbf{v}$ changes sign 
across the interface, and $\mathbf{v}$ also changes sign at least once in the adjacent nodal domain of $\mathbf{u}$.


Together with the fact that the first eigenfunction does not change sign and the nodal bounds in Theorem \ref{thmnodal}, this property establishes the desired Sturm-Liouville oscillation principle.
\end{proof}


\backmatter

\section*{Statements and Declarations}

\bmhead{Funding}
The research of the authors was supported by the National Natural Science Foundation of China (Grant Nos. 12071203, 12371214, and 12371216).

\bmhead{Competing interests}
The authors declare that they have no competing interests.

\bmhead{Data availability}
Data availability is not applicable to this article as no datasets were generated or analyzed during the current study.

\bibliographystyle{spmpsci}
\bibliography{eigenvalue2}

\end{document}